\documentclass{article}
\usepackage{amsmath,amsfonts,latexsym, amsrefs,amssymb}

\usepackage{cite}

\usepackage{graphicx}

\DeclareMathOperator{\Prob}{\mathbb{P}}   

\newcommand{\1}{\mathds{1}}
\usepackage{enumerate}
\usepackage{mathrsfs}
\usepackage{wrapfig}

\usepackage{color}

\numberwithin{equation}{section}
\newcommand{\bla}{\mbox{\boldmath $\lambda$}}

\def\cH{{\mathcal H}}

\newcommand{\pt}{\partial}
\newcommand{\rd}{{\rm d}}

\newcommand{\bR}{{\mathbb R}}

\newcommand{\bZ}{{\mathbb Z}}

\newcommand{\bx}{{\bf{x}}}
\newcommand{\by}{{\bf{y}}}
\newcommand{\bu}{{\bf{u}}}
\newcommand{\bv}{{\bf{v}}}
\newcommand{\bw}{{\bf{w}}}

\newcommand{\al}{\alpha}

\newcommand{\be}{\begin{equation}}
\newcommand{\ee}{\end{equation}}

\newcommand{\e}{{\varepsilon}}

\newcommand{\la}{\lambda}

\newcommand{\om}{{\omega}}

\newcommand{\cL}{{\cal L}}
\newcommand{\cE}{{\cal E}}
\newcommand{\cG}{{\cal G}}

\def\cA{{\mathcal A}}

\def\cW{{\mathcal W}}

\newcommand{\fa}{{\frak a}}
\newcommand{\fb}{{\frak b}}

\newcommand{\Ai}{{\text{Ai} }}

\newcommand{\cR}{{\mathcal R}}

\newcommand{\non}{\nonumber}

\usepackage{amsmath} 
\usepackage{amssymb}
\usepackage{amsthm}

\def\RR{{\mathbb R}}
\def\ZZ{{\mathbb Z}}

\renewcommand{\cal}{\mathcal}

\newcommand{\wh}{\widehat}
\newcommand{\wt}{\widetilde}

\newcommand{\ii}{\mathrm{i}} 

\renewcommand{\epsilon}{\varepsilon}
\renewcommand{\leq}{\leqslant}
\renewcommand{\geq}{\geqslant}

\renewcommand{\le}{\leq}
\renewcommand{\ge}{\geq}

\renewcommand{\P}{\mathbb{P}}
\newcommand{\E}{\mathbb{E}}
\newcommand{\R}{\mathbb{R}}
\newcommand{\C}{\mathbb{C}}
\newcommand{\N}{\mathbb{N}}

\newcommand{\NN}{\mathbb{N}}

\DeclareMathOperator{\var}{Var}

\DeclareMathOperator{\re}{Re}
\DeclareMathOperator{\im}{Im}

\DeclareMathOperator{\sgn}{sgn}

\DeclareMathOperator{\OO}{O}
\DeclareMathOperator{\oo}{o}

\theoremstyle{plain} 
\newtheorem{theorem}{Theorem}[section]
\newtheorem*{theorem*}{Theorem}
\newtheorem{lemma}[theorem]{Lemma}
\newtheorem*{lemma*}{Lemma}
\newtheorem{corollary}[theorem]{Corollary}
\newtheorem*{corollary*}{Corollary}
\newtheorem{proposition}[theorem]{Proposition}
\newtheorem*{proposition*}{Proposition}

\newtheorem{definition}[theorem]{Definition}
\newtheorem*{definition*}{Definition}

\newtheorem*{example*}{Example}
\newtheorem{remark}[theorem]{Remark}

\newtheorem*{remark*}{Remark}
\newtheorem*{remarks*}{Remarks}

\makeatletter
\renewcommand{\section}{\@startsection
{section}
{1}
{0mm}
{-2\baselineskip}
{1\baselineskip}
{\normalfont\large\scshape\centering}} 
\makeatother

\makeatletter
\renewcommand{\subsection}{\@startsection
{subsection}
{2}
{0mm}
{-\baselineskip}
{0.5 \baselineskip}
{\normalfont\bf\itshape} } 
\makeatother

\makeatletter
\renewcommand{\subsubsection}{\@startsection{subsubsection}{3}{\z@}%
  {3.25ex \@plus 1ex \@minus .2ex}{-1em}{\normalfont\normalsize\itshape}}
\makeatother

\usepackage{dsfont}
\usepackage{stmaryrd}

\def\bR{{\mathbb R}}
\def\bZ{{\mathbb Z}}

\newcommand\Si{S}

\title{Edge Universality of Beta Ensembles}

\author{Paul Bourgade${}^1$\thanks{Partially supported
by NSF grant DMS-1208859}\quad
L\'aszl\'o Erd\H os${}^2$\thanks{Partially supported
by SFB-TR 12 Grant of the German Research Council. On leave from
Institute of Mathematics, University of Munich, Germany} \quad
Horng-Tzer Yau${}^1$\thanks{Partially supported
by NSF grant DMS1307444 and Simons Investigator Award}
 \\\\
Department of Mathematics, Harvard University\\
Cambridge MA 02138, USA \\   bourgade@math.harvard.edu \quad
 htyau@math.harvard.edu ${}^1$ \quad  \\ \\
Institute of Science and Technology Austria \\
Am Campus 1, A-3400 Klosterneuburg, Austria \\ lerdos@ist.ac.at ${}^2$
\\}

\date{May 10,  2014}

\begin{document}

\maketitle

\vspace{1cm}

\begin{abstract}
We prove the edge universality of the beta ensembles  for any
$\beta\ge 1$, provided that
the limiting spectrum is supported on a single interval, and the external potential is $\mathscr{C}^4$  and regular.
We also prove that the edge universality holds for
generalized Wigner matrices for all symmetry classes.
Moreover, our results 
allow  us to extend bulk universality for beta ensembles from analytic potentials
to potentials in class $\mathscr{C}^4$.
\end{abstract}

\vspace{1.5cm}


\medskip

\medskip


\newpage

\tableofcontents

\bigskip

\section{Introduction}

Eigenvalues  of  random matrices were  envisioned by  Wigner as  universal  models for
highly correlated  systems.
A  manifestation of this general principle   is the universality  of random matrix statistics, i.e.,
that
 the eigenvalue
 distributions of large matrices
 are universal in the sense that they
depend only on the symmetry class of the matrix ensemble, but not on the distributions
of the matrix elements.
These universal eigenvalue distributions are different
for eigenvalues in the  interior of the spectrum and for the extreme eigenvalues
near the spectral edges.
 In this paper, we will focus on the edge universality.

Let  $\lambda_N$ be the largest eigenvalue of an $N \times N$  random  Wigner matrix
with normalization chosen such that the bulk spectrum is $[-2,2]$.
  The
probability distributions of $\lambda_N$ for the classical Gaussian ensembles are
identified by Tracy and Widom  \cite{TW, TW2} to be
$$
\lim_{N \to \infty} \P( N^{2/3} ( \lambda_N -2) \le s ) =  F_\beta (s),
$$
where $F_\beta(s)$ can be computed in terms
of Painlev\'e equations and  $\beta=1,2,4$
corresponds respectively to the classical orthogonal, unitary  or symplectic ensemble.
The edge universality means   that  the distributions of $\lambda_N$
are  given by $F_\beta$  for non Gaussian ensembles as well.
In fact, this holds not only for the largest eigenvalue,
but  the joint distributions of any finitely many  ``edge eigenvalues"  are universal  as well.

The edge universality for a large  class of Wigner matrices was first  proved
via the moment method by Soshnikov \cite{Sos1999}
for unitary  and orthogonal  ensembles. This method requires that the distribution
of the matrix elements  be symmetric. 
The  symmetry  assumption was partially removed in
\cite{PecSos2007, TV2} and it was completely removed in \cite{EYYrigi}.
In addition to the symmetry assumption, the moment method also requires that sufficient high
moments of the matrix elements be finite.
This assumption was greatly relaxed  in \cite{ EKYYsparse2}
and it was finally proved by Lee and Yin \cite{LeeYin2012} that
essentially the finiteness   of
the  fourth  moment  is the sufficient and necessary condition
for  the Tracy-Widom edge universality to hold (an almost optimal necessary condition was established earlier
in \cite{AufBenPec2009}).

We now turn to the edge universality for invariant ensembles.
These are matrix models with  probability density on the space of
$N\times N$ matrices $H$
  given by $Z^{-1}e^{- N \beta  { \rm Tr} V(H)/2 }$
where
 $V$ is a real valued potential and $Z=Z_N$
is the normalization.  The parameter $\beta$ is determined
by the symmetry class of $H$.
The probability distribution of the ordered eigenvalues of $H$  on
the simplex determined by $\lambda_1\leq \dots\leq \lambda_N$
 is given by
\begin{equation}\label{01}
\mu^{(N)}(\rd\bla)\sim e^{- \beta N \cH(\bla)}\rd\bla,
\quad
\cH(\bla) =   \sum_{k=1}^N  \frac{1}{2}V(\lambda_k)-
\frac{1}{N} \sum_{1\leq i<j\leq N}\log (\lambda_j-\lambda_i) .
\end{equation}
For classical invariant ensembles, i.e., $\beta=1,2,4$,   it is well-known  that
 the correlation functions  can be expressed in terms of orthogonal polynomials.
Historically, they have been first analyzed in the bulk.
 The analysis at the edges is  not a straightforward generalization
of that in the bulk and serious technical hurdles  had  to be overcome.
 Nevertheless,   the edge universality
was proved by  Deift-Gioev  \cite{DGedge} for general polynomial potentials, by  Pastur-Shcherbina \cite{PasShc2003}
and Shcherbina \cite{Shc2009}  for real analytic,  even potentials.

The measure  $\mu$ (\ref{01}) can also be considered
for non classical values of $\beta$, i.e., $\beta \not \in \{1, 2, 4\}$, although
in this case  there is no simple  matrix ensemble
behind the model.
It is a Gibbs measure of particles
in $\bR$  with a logarithmic interaction  and with an external potential $V$,
 where the parameter $\beta$
is interpreted as the inverse temperature and can be an arbitrary positive number.
We will often
refer to the variables $\lambda_j$ as particles or points
and the system is often called the beta ensemble or  log-gas (at inverse temperature $\beta$).
We will use these two terminologies interchangeably.
When $\beta \not \in \{1, 2, 4\}$ and the potential $V$ is general,
no simple expression of the correlation functions   in terms of orthogonal polynomials
is known. For certain special potentials and  even integer $\beta$, however,  there are still explicit formulas for correlation functions
\cite{For}.
Furthermore, for general $\beta$ in the Gaussian case, i.e.,
when $V$ is quadratic, Dumitriu and Edelman   \cite{DumEde}
proved   that  the measure \eqref{01} describes  the eigenvalue distribution  of
a special tridiagonal matrix.  Using this connection,
Ram{\'{\i}}rez, Rider and Vir{\'a}g were able to characterize the edge distributions for
all $\beta$ \cite{RamRidVir2011};
a characterization of the bulk statistics of the Gaussian beta ensembles was obtained in \cite{ValVir}.  A similar approach, independently of
our current work,  resulted
in the proof of edge universality for convex polynomial potentials
\cite{KriRidVir2013}.

We now compare the notions of  edge and bulk universality.
The edge universality refers to the distributions of individual eigenvalues.
However, according to Wigner's original vision,  the bulk universality
    concerns  {\it differences} of  neighboring
eigenvalues, i.e.,  gap distributions.  The bulk universality is often
formulated in terms of local correlation functions. These two notions are equivalent
only after  a certain averaging in the energy parameter.
Strictly speaking, there are three notions of bulk universality:
(i) in a weak sense  which allows for energy averaging;
(ii) correlation function universality at a fix energy;
(iii) gap universality at a fixed label $j$.
 Clearly,  universality in the sense (ii) or (iii) implies (i).

 The bulk universality  in the sense of (ii)   for  classical invariant ensembles was
proved in \cite{BleIts,   DKMVZ1, DKMVZ2, DGKV,  PasShc97, PasShcBulk, DG,  Shc2011, Wid} using methods related to orthogonal polynomials.
For Wigner ensembles, universality   for Hermitian matrices in the sense (ii) was proved  in
\cite{EPRSY,  TV1, ERSTVY, EYComment}
and for  all symmetry classes in the sense (i)
in  \cite{ESY4,  ESYY,   EYYrigi}.   The gap universality, i.e., (iii), is in fact much harder
 to obtain;  it was proved only recently in \cite{EYsinglegap} both for invariant and Wigner ensembles
using new ideas from parabolic regularity theory  (the special case of hermitian matrices with
the first four moments of the matrix elements  matching those of GUE was proved earlier in \cite{taogap}).
The bulk universality for  log-gases   for general $\beta$
was proved in  the sense (i) and (iii)  in  \cite{BouErdYau2011,
BouErdYau2012, EYsinglegap}.  The bulk  universality in the sense (ii) for  Wigner ensembles
 with $\beta \not = 2$
and for log-gases with $\beta \not \in \{1,2, 4\}$ remains  open problems.

Returning to the edge universality,
we will establish   the  following two results in this paper: (1) edge universality
 for $\mathscr{C}^4$ potentials and for all $\beta \ge 1$;
 (2) edge universality for generalized Wigner matrices  (these are matrices with
independent but not necessarily identically distributed entries,
see Definition \ref{D1}). An important ingredient of the proof will be an optimal location estimate for the
particles up to the edge, for external potentials of class $\mathscr{C}^4$. 
This rigidity will also allow us to
remove the analyticity assumption from previous results about bulk universality \cite{BouErdYau2011,
BouErdYau2012, EYsinglegap}.

We now outline the technique   used in this paper.
For the edge universality of invariant ensembles,
the basic idea is to consider a local version of the log-gas \eqref{01}. This is the measure
on  $K$ consecutive particles that is obtained by fixing all other particles which
 act as boundary conditions.
Following the standard language in statistical physics, we will refer to  these local measures
as local log-gases.
Our  core  result is  the ``uniqueness" of this local measure in the limit  $K \to \infty $
 assuming that
the boundary conditions are ``good''.
By uniqueness, we mean that the distributions of the particles far away
from the boundaries are independent of choice of the ``good'' boundary conditions.
This idea  first appeared in \cite{BouErdYau2011} for proving the bulk universality of log-gases.
However, the uniqueness of the
 local Gibbs state in the bulk was defined slightly   differently  in \cite{BouErdYau2011};
only  the gap distributions  were required to be  independent of the boundary conditions.

It is well-known that the uniqueness of local Gibbs measures in the thermodynamical limit
  is closely related to the decay of correlation functions.
The work of Gustavsson \cite{Gus2005} for the special $\beta=2$ and Gaussian case
(i.e., the  GUE case)  indicates
that  in  general
the point-point correlation function
 decays only logarithmically  in the bulk, i.e., $\langle \lambda_i; \lambda_j \rangle_{\mu_\beta}
\sim \log |i-j|$.
Gibbs measures with  such a slow decay are typically not unique in the usual sense.
The key reason why we were able to prove the uniqueness of the gap distributions of local log-gases
in the bulk  \cite{EYsinglegap}   is the observation
that the point-gap  correlation,  $\langle \lambda_i;  \lambda_j- \lambda_{j+1} \rangle_{\mu_\beta}
\sim \partial_j \langle \lambda_i; \lambda_j \rangle_{\mu_\beta} $ is expected to
decay  much faster
due to the simple reason that $\partial_j \log |i-j|\sim 1/|j|$.
In real statistical physics system, however,
it is very difficult to compute derivatives of correlation functions
  unless they are expressed almost explicitly by some expansion method.
The {\it Dirichlet form inequality} \cite{ESY4},
a main tool  in \cite{BouErdYau2011},
allows us to take advantage  of the fact that  the observables are functions of the gaps.

In the subsequent work \cite{EYsinglegap}, the correlation functions were expressed in terms of
 off-diagonal matrix elements of heat kernels describing
random walks in random environments.
This   representation in a lattice  setting  was given  in   \cite{DeuGiaIof2000, GOS}.  In a slightly different formulation
it  already appeared  in the earlier paper of Naddaf and Spencer \cite{NS},   which  was a
probabilistic formulation  of the idea
 of Helffer and Sj\"ostrand \cite{HS}.
Using this representation,
 the decay of the point-gap  correlation amounts to the H\"older continuity
of  the heat kernel for the random walk dynamics  \cite{EYsinglegap}.
We note that  the jump rates in this random walk dynamics are long ranged and contain short distance
singularities depending on the stochastically driven environment.
The proof of the  H\"older continuity in \cite{EYsinglegap} requires extending the De Giorgi-Nash-Moser
type method of  Caffarelli, Chan and Vasseur  \cite{CCV} to the singular coefficient case and providing
a priori estimates such as rigidity and  level repulsion.

 It should be stressed that, despite these efforts,
only gap distributions but not those
of individual eigenvalues
were  identified  in \cite{EYsinglegap}. Edge universality, however,  is exactly about individual eigenvalues
and not about gaps.   The surprising fact is that
correlation functions of  log-gases  decay as a power law near the edges!
Thus we do not need the H\"older regularity argument  from  \cite{EYsinglegap}
 to analyze the edges.   Instead, in this paper
we  rely on the energy method from parabolic PDE's and on certain new Sobolev type inequalities
for nonlocal operators  to prove
the decay of off-diagonal elements of the heat kernel. For this purpose,
we will need    rigidity  and level repulsion estimates near  the edges.
We will extend the {\it multi-scale   analysis of the loop equation}, first appeared in
 \cite{BouErdYau2011},  in two directions.  First,
this analysis will be performed  along the whole spectrum,  including  the edge, where  the change of  scaling  poses a
 major difficulty; second, analyticity of the external potential is not required, thanks to a new analysis of the loop equation.

For the edge universality of Wigner ensembles,  we will  use the idea of {\it the  local relaxation flow}
initiated in \cite{ESY4, ESYY}
 and the  Green function  comparison
theorem  from  \cite{EYY1}.
  This theorem   can be used for both
the bulk or the  edge universality.  In particular,
 the  eigenvalue distributions of two Wigner ensembles near the edges are
 the same  provided that the variances of the matrix elements  of the two ensembles are identical.
This implied  the edge universality for Wigner matrices  \cite{EYYrigi}.

On the other hand,   if
the variances of the matrix elements are allowed to vary, then the matrix
 cannot be matched to a Gaussian Wigner matrix.  Thus the edge universality
for generalized Wigner matrices
cannot be proved directly with  the Green function  comparison  theorem.
Using the uniqueness of {\it local}  log-gases,   we can identify
the distributions  of the edge particles in the Dyson Brownian Motion (DBM).
This implies the edge universality for general classes of  Gaussian divisible
 ensembles with varying variance.
Finally, we will  use the  Green function  comparison
theorem to bridge the gap between   generalized Wigner matrices
and their  Gaussian divisible counterparts.

We  emphasize that  the uniqueness of local log-gases plays a  {\it central}    role both in
the edge universality of  log-gases and in our analysis of edge points in DBM.
For  log-gases, it is natural to  localize the problem so that  the external potential
can be replaced  by its first order approximation and thus it becomes universal after scaling.
However,  localization of the measure in general
introduces very large errors in strongly correlated systems. The key observation is that there are strong  cancellations
in the  effective potential for ``good''  boundary  conditions. The significance of the local log-gases
in the proof of proof of universality for Wigner matrices is  subtler, and will be explained in details in Section \ref{sec:wigner}.
\\

{\it Convention.}  We use the letters $C, c$ to denote positive constants, independent of $N$, whose values
may change from line to line. We will often estimate the probability of rare events $\Omega=\Omega_N$
that are typically either subexponentially  small, $\P (\Omega)\le \exp(-N^c)$ or
small by an $N$-power; $\P (\Omega)\le N^{-c}$.  In both cases it is understood that
the statements hold for any sufficiently large $N\ge N_0$. We will not follow
the precise values of the exponents $c$ or the thresholds $N_0$.

\section{Main results}

We will have two related results, one concerns the generalized Wigner ensembles, the
other one the general  beta ensembles.

\subsection{Edge universality of the beta ensembles}

We first define the beta ensembles.
 Let $\Xi^{(N)}\subset \RR^N$ denote the set
\begin{equation}\label{simplex}
   \Xi^{(N)}:=\{ \bla=(\lambda_1, \lambda_2, \dots , \lambda_N)\; : \;
\lambda_1\le \lambda_2\le \dots \le\lambda_N
\}.
\end{equation}
Consider the probability distribution on
 $\Xi_N$ given by
\begin{equation}\label{eqn:measure}
\mu^{(N)}_{\beta, V}=\mu^{(N)}(\rd\bla)=\frac{1}{Z_{\beta, V}^{(N)}}
e^{- \beta N \cH(\bla)}\rd\bla,
\qquad
\cH(\bla) =   \sum_{k=1}^N  \frac{1}{2}V(\lambda_k)-
\frac{1}{N} \sum_{1\leq i<j\leq N}\log (\lambda_j-\lambda_i) ,
\end{equation}
 where $Z_{\beta, V}^{(N)}$ is the normalization.
In the following, we often omit the parameters $N$, $\beta$ and $V$ in the notation
and  we will write $\mu$ for $\mu^{(N)}$. Sometimes emphasize
the dependence in the external potential by writing $\mu=\mu_V$.
We will use $\P^\mu$ and $\E^\mu$ to denote the probability and the
expectation with respect to $\mu$.

  We will view $\mu$ as a Gibbs measure of $N$ particles
in $\bR$  with a logarithmic interaction, where the parameter $\beta > 0 $
is interpreted as the inverse temperature.  We will
refer to the variables $\lambda_j$ as particles or points
and the system is called {\it log-gas}  or {\it general beta ensemble}.
We will assume that the potential $V$ is a  $\mathscr{C}^4$  real function  in $\RR$  such that
its second derivative is bounded below, i.e., we have
\begin{equation}\label{eqn:LSImu}
\inf_{x\in\RR}V''(x) \ge -2W
\end{equation}
for some constant $W\ge 0$, and
\begin{equation}\label{eqn:GrowthCondition}
V(x)> (2 + \alpha)\ln(1+|x|),
\end{equation}
for some\ $\alpha > 0$,
if $|x|$ is large enough.
It is known \cite{BouPasShc1995}
that under these (in fact, even weaker) conditions
the measure is normalizable, $Z^{(N)}<\infty$. Moreover,
the averaged density of the empirical spectral measure, defined as
$$
   \varrho^{(N)}_1(\lambda)= \varrho_1^{(N,\beta,V)}(\lambda) : = \E^{\mu} \frac{1}{N}\sum_{j=1}^N \delta(\lambda-\lambda_j),
  \qquad \lambda\in \bR,
$$
converges weakly to a continuous function $\varrho=\varrho_V$, the equilibrium density, with compact support.
We assume that  $\varrho$
is supported on a single interval $[A,B]$, and that $V$
 is {\it regular} in the sense of \cite{KuiMcL2000}. We recall that $V$ is regular if
its equilibrium  density $\varrho$ is positive on $(A,B)$
 and vanishes like a square
root at each of the endpoints of $[A,B]$, that is
\begin{align}\label{sqsing}
\varrho(t)&=s_A\sqrt{t-A}\left(1+\OO\left(t-A\right)\right),\ t\to A^+,\\
\varrho(t)&=s_B\sqrt{B-t}\left(1+\OO\left(B-t\right)\right),\ t\to B^-, \nonumber
\end{align}
for some constants $s_A,\, s_B>0$. We remark that
this regularity assumption is not a strong constraint;
regular potentials $V$ form  a dense and open subset in the space of the potentials
with a natural  topology \cite{KuiMcL2000}.

Let the {\it limiting classical location} of the $k$-th particle, $\gamma_k=\gamma_k(N)$, be
defined by
\begin{equation}\label{gammadef}
 \int_{-\infty}^{\gamma_k}\varrho(s)\rd s=\frac{k}{N}.
\end{equation}
Finally, we introduce the notation $\llbracket p,q\rrbracket = [p,q]\cap \ZZ$
for any  real numbers $p<q$.

We will be interested in the usual $n$-point correlation functions,
generalizing $\varrho_{1}^{(N)}$,
and  defined by
\begin{equation}\label{eqn:corrFunct}
\varrho^{(N)}_n(\lambda_1,\dots,\lambda_n)=
\int_{\RR^{N-n}} \mu^\#(\bla)\rd \lambda_{n+1}\dots\rd \lambda_{N},
\end{equation}
where $ \mu^\#$ is the  symmetrized version of $\mu$ given in \eqref{eqn:measure}
but defined  on $\RR^N$  instead of the simplex $\Xi^{(N)}$:
$$
\mu^{\# (N)}(\rd\bla)=\frac{1}{N!}\mu(\rd\bla^{(\sigma)}),
$$
where
$\bla^{(\sigma)}=(\lambda_{\sigma(1)},\dots,\lambda_{\sigma(N)})$, with
$\lambda_{\sigma(1)}<\dots<\lambda_{\sigma(N)}$.
Our main result is the following.

\newcommand{\ko}{2/5}	
\newcommand{\kt}{1/4}

In the following theorem, we consider two regular potentials, 
 $V$ and $\widetilde V$, such that their equilibrium densities
$\varrho_V$ and $\varrho_{\widetilde V}$ are supported on a single interval. 
Without loss of generality (by applying  a simple scaling and shift), we may also assume that the 
singularities at the left edge match and both occur at $A=0$, with the same constant $s_A=1$:
\begin{equation}\label{eqn:matchSing}
\varrho_V(t)=\sqrt{t}\left(1+\OO\left(t\right)\right),\ \varrho_{\widetilde V}(t)=\sqrt{t}\left(1+\OO\left(t\right)\right),\ t\to 0^+.
\end{equation}

\begin{theorem}[Edge universality for beta ensembles]\label{thm:beta}
Let  $\beta\ge 1$
and $V$, $\widetilde V$ be 
$\mathscr{C}^4$, regular and satisfy
(\ref{eqn:LSImu}), (\ref{eqn:GrowthCondition}). Assume 
that the equilibrium density $\varrho_V$ and $\varrho_{\widetilde V}$ are
supported on a single interval and satisfy (\ref{eqn:matchSing}).

For any constant $\kappa <  \ko $
there exists  $\chi>0$  such that the following holds.
Take any fixed $m\geq 1$ and a continuously differentiable  compactly supported function $O:\mathbb{R}^m\to\mathbb{R}$.
There exists a constant $C>0$ such that for any $N$ and $\Lambda
\subset\llbracket 1,N^\kappa\rrbracket$
with $|\Lambda|= m$, we have
\be\label{mm}
\left|
(\E^{\mu_V}-\E^{\mu_{\widetilde V}})
O
\left(\left(N^{ 2/3} j^{1/3}(\lambda_j-\gamma_j)\right)_{j\in \Lambda}\right)
\right|\leq C  N^{-\chi}.
\ee
\end{theorem}

{\it Remark.} Note that  one may define $\gamma_j$  in (\ref{gammadef}) with respect
to the measure $\varrho_V$ or $\varrho_{\widetilde V}$, it does not make any difference
in the above theorem when $\kappa<2/5$:
from (\ref{eqn:matchSing})  one obtains
$\gamma_j-\widetilde\gamma_j=\OO\left(\left(j/N\right)^{4/3}\right)$, which is  of smaller order  than
the scale   $N^{-2/3}j^{-1/3}$ detected in \eqref{mm}. We also remark that Theorem~\ref{thm:beta}
is formulated for points near the lower spectral edge $A$, but a similar statement
holds near the upper spectral edge $B$.\\

The first results on edge universality for invariant ensembles concerned the classical values of $\beta=1,2,4$. 
The case 
$\beta=2$ and real analytic $V$ was  solved in \cite{DKMVZ1, DGedge}.
The $\beta=1,4$ cases are  considerably harder than $\beta=2$. For  $\beta=1,4$
universality was first solved  for polynomial potentials in  \cite{DGedge},
then the real analytic case for $\beta=1$  in \cite{Shc2009, PasShc2003}, which also give
an alternative proof for $\beta=2$. 
Finally, independently of our work with a completely different method, edge universality for any
$\beta>0$ 
and convex  polynomial $V$ was recently proved in \cite{KriRidVir2013}.

Choosing $S=\llbracket 1,m\rrbracket$ and $\widetilde V(x)=x^2$ in the previous
 theorem allows us to  identify the universal distribution from Theorem~\ref{thm:beta}
with the Tracy-Widom distribution with parameter $\beta>0$.   This distribution
can be represented via the {\it stochastic Airy operator}.
We refer to \cite{RamRidVir2011} for its proper definition,
 the Hilbert space it acts on, and the proof that its smallest eigenvalues describe
  the asymptotic edge fluctuations of the Gaussian beta ensembles.

\begin{corollary}[Identification of the edge distribution] \label{cor:TW}
Let $\beta\ge 1$ and $m\in\mathbb{N}$ be fixed, and  $\Lambda_1<\dots<\Lambda_m$ the $m$
smallest eigenvalues of the stochastic Airy operator
$-\partial_{xx}+x+\frac{2}{\sqrt{\beta}}b_x'$ on $\bR_+$,  where $b'_x$ is a
white noise.
Let $V$ be $\mathscr{C}^4$, 
regular with equilibrium density
supported on a single interval,
and satisfy (\ref{eqn:LSImu}), (\ref{eqn:GrowthCondition}),  (\ref{eqn:matchSing}).
Then the following convergence in distribution holds:
$$
(N/2)^{2/3}(\lambda_1-A,\dots,\lambda_m-A)\to (\Lambda_1,\dots,\Lambda_m).
$$
\end{corollary}

Theorem \ref{thm:beta} can be used to show   Gaussian fluctuations
for the points in an intermediate distance from the edge. Indeed, such
fluctuations were proved by Gustavsson in \cite{Gus2005}
in  the $\beta=2$ Gaussian case (GUE)
 for all eigenvalues,  and this was extended $\beta=1$ and $4$ in \cite{ORo2010}.
   Combining these results with  Theorem \ref{thm:beta}
immediately gives
 the following statement (here $k\sim N^\vartheta$ means $\log k/\log N\to\vartheta$).

\begin{corollary}[Gaussian fluctuations]\label{cor:Gaussian}
Let $\beta=1,2$ or $4$  and the potential $V$
be $\mathscr{C}^4$,
regular such that the equilibrium density $\varrho_V$ is
supported on a single interval and satisfies (\ref{eqn:matchSing}).
Consider the measure
$\mu_{\beta, V}^{(N)}$.  We define
$$
X_i=c\ \frac{\lambda_{i}-\gamma_{i}}{(\log i)^{1/2}N^{-2/3}i^{-1/3}},
$$
where $c=(3/2)^{1/3}\pi \beta^{1/2}$.
Fix    $\kappa< \ko $.  Then for any sequence $i=i_N\to\infty$, with $i\leq N^\kappa$,
we have $X_i\to \mathcal{N}(0,1)$ in distribution.

Moreover,  for  some fixed $m>0$ and $\delta \in (0, \ko)$,
 let $k_1<\dots<k_m$ satisfy $k_1\sim N^\delta$, and
$k_{i+1}-k_i\sim N^{\vartheta_i}$, $0<\vartheta_i<\delta$.
Then $(X_{k_1},\dots,X_{k_m})$ converges to a Gaussian vector with covariance matrix
$\Lambda_{ij}=1-\delta^{-1}\max\{\vartheta_k,i\leq k<j\}$ if $i<j$, $\Lambda_{ii}=1$.
\end{corollary}

We note that if Gustavsson's result on Gaussian fluctuations were known
for the general Gaussian beta ensembles, then this corollary would prove
a central limit theorem for general beta ensembles near the edge.

An important element in the proof of Theorem \ref{thm:beta} consists in proving the following
rigidity estimate asserting that  any
particle $\lambda_k$  is very close  to  its limiting classical location.
For any $k\in \llbracket 1,  N\rrbracket$ we define
$$
    \wh k: = \min\{ k, N+1-k\}.
$$
For orientation, we note that
$$
   \gamma_k  \sim   \Big( \frac{\wh k}{N}\Big)^{2/3},
 \qquad  \gamma_{k+1}-\gamma_k
 \sim   N^{-2/3} ({\wh k})^{-1/3},
$$
where $A\sim B$ means $c\le A/B\le C$. More precisely,
by the square-root singularity of $\varrho$ near the left  edge,
\be\label{orient}
  \gamma_k \sim 
\Big( \frac{k}{N}\Big)^{2/3}\Big[ 1 + O\big( (k/N)^{2/3}\big)\big],
 \qquad  \gamma_{k+1}-\gamma_k
 \sim     N^{-2/3}k^{-1/3}
\Big[ 1 + O\big( (k/N)^{2/3}\big)\big],
\ee
and similar asymptotics hold near the right edge. The following theorem states that all particles will be close to their classical locations
on this scale,
up to a factor  $N^\xi$ with an arbitrary small exponent $\xi>0$.
Following \cite{EYYrigi}, we will call such a precise bound  on the locations of particles
a  {\it rigidity estimate}.   The rigidity estimate in some  weaker forms has already been used as a fundamental input \cite{ESYY} to prove the universality for  Wigner matrices. It also played a key role in the proof of the bulk universality for  the  log-gases in
 \cite{BouErdYau2011}. The following result extends the rigidity estimate from the bulk to the edges, and removes 
the analyticity assumption.

\begin{theorem} [Rigidity estimate for global  measures] \label{thm:rigidity} Let $\beta>0$, $V$ be
$\mathscr{C}^4$, regular with equilibrium density
supported on a single interval $[A,B]$,  and satisfy  (\ref{eqn:LSImu}), (\ref{eqn:GrowthCondition}).
For any $\xi>0$,  there are constants
 $c>0$ and $N_0$ such that for any $N\geq N_0$  and $k\in\llbracket 1,N\rrbracket$ we have
\begin{equation}\label{rig}
\P^\mu\left(|\lambda_k-\gamma_k|> N^{-\frac{2}{3}+ \xi }(\hat k)^{-\frac{1}{3}}
\right)\leq e^{- N^c}.
\end{equation}
\end{theorem}

Related bounds on the concentration of the empirical density on a scale  far from the optimal one \eqref{rig}
were established previously \cite{BenGui1997, Joh1998, Shc2011}, see also references in  \cite{AndGuiZei2010}.

 Thanks to Theorem \ref{thm:rigidity}, bulk universality holds 
for beta ensembles, as stated in 
\cite{BouErdYau2011,BouErdYau2012,EYsinglegap}, 
without the analyticity assumption.

\begin{theorem}[Bulk universality]
Let $V$ be
$\mathscr{C}^4$, regular with equilibrium density
supported on a single interval $[A,B]$,  and satisfy  (\ref{eqn:LSImu}), (\ref{eqn:GrowthCondition}).
Then the following two results hold.
\begin{enumerate}[(i)]
\item
{\bf Correlation functions.} 
For any fixed $\beta>0$, $E\in (A,B)$, $|E'|<2$, $ n\in \N$  
and $0<k\le \frac{1}{2}$ there exists a $\chi>0$ 
such that for any  continuously 
differentiable $O$ with compact support 
we have (setting  $s:=N^{-1+k}$)
\begin{align*}
\Bigg|  \int  & \rd \alpha_1 \cdots \rd \alpha_n\, O(\alpha_1,
\dots, \alpha_n) \Bigg [
  \int_{E - s}^{E + s} \frac{\rd x}{2 s}  \frac{1}{ \varrho (E)^n  }  \varrho_n^{(N)}   \Big  ( x +
\frac{\alpha_1}{N\varrho(E)}, \dots,   x + \frac{\alpha_n}{N\varrho(E)}  \Big  ) \\
&
-   \int_{E' - s}^{E' + s} \frac{\rd x}{2 s}  \frac{1}{\varrho_{sc}(E')^n}
\varrho_{{\rm Gauss}, n}^{(N)}
   \Big  ( x +
\frac{\alpha_1}{N\varrho_{sc}(E')}, \dots,   x + \frac{\alpha_n}{N\varrho_{sc}(E')}  \Big  ) \Bigg| \le CN^{-\chi}.
\end{align*}
Here $\varrho_{sc}(E)=\frac{1}{2\pi}\sqrt{4-E^2}$ is the Wigner semicircle law
and $\varrho_{{\rm Gauss}, n}^{(N)}$ are the correlation functions of the Gaussian $\beta$-ensemble, i.e.
with $V(x)=x^2$.
\item {\bf Gaps.} For any fixed  $\beta\ge 1$ 
 and $\alpha>0$, there is some $\e>0$ such that
for  any $n\in \N$ and  any differentiable $O$ with compact support, and any $k,m\in\llbracket \alpha N,(1-\alpha)N\rrbracket$, we have
\begin{multline*}
\Big|
\E^{\mu}O\left(N c^{\mu}_k(\lambda_k-\lambda_{k+1}),\dots,N c^{\mu}_k(\lambda_k-\lambda_{k+n})\right) \\
-\E^{\rm Gauss}O\left(N c^{{\rm Gauss}}_m(\lambda_m-\lambda_{m+1}),\dots,N c^{{\rm Gauss}}_m(\lambda_m-\lambda_{m+n})\right)
\Big|
\leq C N^{-\e} 
\end{multline*}
where $c_k^{\mu}=\varrho^{(\mu)}(\gamma_k)$ and
  $c_m^{\rm Gauss}=\varrho_{sc}(\gamma_m^{\rm Gauss})$
with $\gamma_m^{\rm Gauss}$ being the $m$-th quantile of the semicircle law defined by
$\int_{-2}^{\gamma_m^{\rm Gauss}}\varrho_{sc}(x) \rd x =m/N$. 
\end{enumerate}
\end{theorem}

\begin{proof}
Part $(i)$ was proved in \cite{BouErdYau2012} under the assumption that $V$ is analytic, 
a hypothesis that was only required for proving rigidity in the bulk of the spectrum. Theorem 
\ref{thm:rigidity} proves that $V$ of class $\mathscr{C}^4$ is sufficient for rigidity, and the proof of the uniqueness
of the Gibbs measure is identical to \cite{BouErdYau2011,BouErdYau2012}. 
The result in these papers were stated in a limiting form, as $N\to\infty$,
and for smooth observables $O$, but the proofs hold for any continuously
differentiable $O$ and with an effective error bound of order $N^{-\chi}$ 
with some $\chi>0$ as well.
The statement $(ii)$
holds for the same reason, being previously proved for analytic $V$ in \cite{EYsinglegap}.
\end{proof}

We finally remark that while the rigidity estimate \eqref{rig} holds for any $\beta>0$, the edge universality
in Theorem~\ref{thm:beta} was stated only for $\beta\ge 1$.  This restriction is mainly due to
that
the DBM dynamics \eqref{SDE}  is known to be well-posed  only for $\beta \ge 1$. We believe that this restriction can be removed, but we will not pursue  this
issue in this paper.

\subsection{Edge universality of  the generalized Wigner matrices}

We now define the generalized Wigner ensembles.
  Let $H=(h_{ij})_{i,j=1}^N$  be an $N\times N$ complex  Hermitian or  real symmetric matrix  where the
 matrix elements $h_{ij}=\bar {h}_{ji}$, $ i \le j$, are independent
random variables given by a probability measure $\nu_{ij}$
with mean zero and variance $\sigma_{ij}^2\ge 0$;
\be
  \E \, h_{ij} =0, \qquad \sigma_{ij}^2:= \E |h_{ij}|^2.
\label{aver}
\ee
The distribution $\nu_{ij}$ and its variance $\sigma_{ij}^2$ may depend on $N$,
 but we omit this fact in the notation.
 We also assume that the
normalized matrix elements satisfy
a uniform subexponential decay,
\be\label{subexp}
  \P( |h_{ij}|> x \sigma_{ij})\le \vartheta^{-1} \exp{(-x^{\vartheta})}, \qquad x>0,
\ee
with some fixed  constant $\vartheta$,  uniformly in $N, i, j$.

\begin{definition}\label{D1}\cite{EYY1}
 The matrix ensemble $H$  defined  above is called generalized Wigner matrix if
the following  assumptions hold on the variances of the matrix
elements \eqref{aver}
\begin{description}
\item[(A)] For any $j$ fixed
$$
   \sum_{i=1}^N \sigma^2_{ij} = 1 .
$$

\item[(B)]   There exist two positive constants, $C_1$ and $C_2$,
independent of $N$ such that
$$
\frac{C_1}{N} \le \sigma_{ij}^2\leq \frac{C_2}{N}.
$$
\end{description}
For Hermitian ensembles, we  additionally  assume that for each $i,j$ the $2\times 2$ covariance matrix
$$
\Sigma_{ij} \;=\; \begin{pmatrix} \E (\re h_{ij})^2 & \E (\re h_{ij})(\im h_{ij}) \\
  \E (\re h_{ij})(\im h_{ij}) & \E ( \im h_{ij})^2
\end{pmatrix}
$$
satisfies
$$
     \Sigma_{ij} \;\geq\; \frac{C_1}{N}
$$
in  matrix sense.
\end{definition}

 Let $\P^H$ and $\E^H$ denote the probability and the
expectation with respect to
this ensemble.
Our  result
asserts that the local  statistics on the edge of the spectrum
are universal for any general Wigner matrix, in particular
they coincide with those of the corresponding  standard
Gaussian ensemble.

\begin{theorem}  [Edge universality of generalized Wigner matrices] \label{thm:wigner}
Let $H$ be a generalized
Wigner ensemble  with subexponentially decaying matrix elements, \eqref{subexp}.
For any $\kappa < \kt$,  there exists  $\chi>0$  such that the following result holds.
Take any fixed $m\geq 1$ and a smooth compactly supported function $O:\mathbb{R}^m\to\mathbb{R}$.
Then there is a constant $C>0$ such that for any $N$ and $\Lambda\subset\llbracket 1,N^\kappa\rrbracket$
with $|\Lambda|= m$, we have
$$
\left|
(\E^{H}-\E^{\mu_{G}})
O
\left(\left(N^{ 2/3} j^{1/3}(\lambda_j-\gamma_j)\right)_{j\in \Lambda}\right)
\right|\leq C\ N^{-\chi},
$$
where $\mu_G$  is the  standard Gaussian GOE or GUE ensemble, depending
on the symmetry class of $H$
(It is well-known that $\mu_G$ is also given by \eqref{eqn:measure}
with  potential $V(x) = \frac{1}{2}x^2$ and with the choice $\beta=1, 2$, respectively).
\end{theorem}

This theorem immediately implies analogues of Corollaries \ref{cor:TW} and \ref{cor:Gaussian} 
in the case of symmetric or Hermitian generalized Wigner ensembles.

Edge universality for Wigner matrices was first proved in \cite{Sos1999}
assuming symmetry of the distribution of the matrix elements and finiteness
of all their moments.  In the consequent works, after partial
results in \cite{PecSos2007, TV2},
the symmetry condition was completely eliminated \cite{EYYrigi}.
The moment condition was improved in \cite{AufBenPec2009, EKYYsparse2}
and the optimal result  was obtained in \cite{LeeYin2012}.
All these works heavily rely on the fact that the variances of
the matrix elements are identical. The main point of Theorem~\ref{thm:wigner}
is to consider generalized Wigner
matrices, i.e., matrices with non-constant variances.
In fact, it was shown in \cite{EYYrigi}  that the edge statistics for any generalized Wigner
matrix are universal in the sense that they coincide with those of a generalized Gaussian Wigner
matrix with the same variances, but it was not shown that the statistics
are independent of the variances themselves.  Theorem~\ref{thm:wigner}
provides this missing step and thus it proves the edge universality
in the broadest sense.

\section{Local equilibrium measures} \label{sec:loc}

Recall that the support of the equilibrium density $\varrho$ was denoted by   $[A,B]$.
Without loss of generality, by a shift we  set $A=0$ and
we will study the particles near the lower edge of the support.
Fix  a small exponent $\delta$ and a parameter $K=K_N$ satisfying
\begin{equation}
N^\delta \le K \le N^{1-\delta}.
\label{NK}
\end{equation}
 Denote by
$ I = \llbracket 1, K\rrbracket$ the set of the first $K$  indices.
We will distinguish the first $K$ particles from the rest
by renaming them as
$$
(\lambda_1, \lambda_2, \dots,
\lambda_N)=(x_{1},  \dots, x_{K}, y_{K+1},
\ldots y_{N}) \in \Xi^{(N)}.
$$
 Note that the particles  keep their original indices. We recall
the notation $\Xi^{(N)}$ for the simplex  \eqref{simplex}.
In short we will write
$$
\bx= (x_{1}, \dots, x_{K})\in \Sigma^{(K)}, \qquad \mbox{and}\qquad
 \by=
   (y_{K+1},\dots, y_N) \in \Sigma^{(N-K)}.
$$
These points are always listed  in increasing order and
we will refer to the $y$'s as the  {\it external}
points and to the $x$'s as  {\it internal} points.
We will fix the external points (often called
 boundary conditions) and study
the conditional measures on the internal points.
Note that for any fixed $\by\in \Xi^{(N-K)}$,  all $x_j$'s
lie in the {\it  open  configuration interval},  denoted by
$$
 J=J_\by=(-\infty, y_{K+1}) =:(-\infty, y_+].
$$
Define  the
{\it local equilibrium measure}  (or {\it local measure} in short)
 on $J^K$ with boundary condition  $\by$ by
$$
  \mu_\by (\rd \bx)=\frac{1}{Z_\by} e^{-\beta N \cH_\by(\bx)}\rd\bx, \qquad \bx\in J^K,
$$
where we introduced the Hamiltonian
\begin{align}
   \cH_\by (\bx):= &\frac{1}{2} \sum_{i\in I}   V_\by (x_i) -\frac{1}{N}
   \sum_{i,j\in I\atop i<j} \log |x_j-x_i|,
\non \\
\non
   V_\by(x):= & V(x) - \frac{2}{N}\sum_{j\not\in I} \log |x-y_j| .
\end{align}
Here  $V_\by(x)$ can be viewed as the external potential of
a log-gas of the points $\{ x_i : i\in I\}$.  Although this is the natural local  measure,
it does not have good uniform  convexity in the regime $x_1 \ll 0$. It is more convenient to consider
the following modified  measure $\sigma$ and its local version $\sigma_\by$.  For the proof of
the universality of the original measure $\mu$ it will actually
 be sufficient to consider only the local measure $\sigma_\by$.

We will fix a small parameter $\xi>0$ whose actual value is immaterial; it will be used
to provide an multiplicative error bar of size $N^{C\xi}$ in various estimates on the location
of the particles. We will not carry $\xi$ in the notation and at
the end of the proof it can be chosen sufficiently small, depending on all
other exponents along the argument.

We introduce a confined measure by adding an extra
quadratic potential $\Theta$ to prevent the $x_i$'s from deviating far in the left direction:
\begin{align}
\label{sig}
  \sigma (\rd\bx)&:=\frac{Z }{Z^\sigma} e^{- 2 \beta  \sum_{i=1 }^N
   \Theta\left(N^{\frac{2}{3}-\xi}x_i \right) } \mu (\rd\bx) = \frac{1}{Z^\sigma}
e^{-\beta N \cH^\sigma (\bx)} \rd \bx ,\\
   \cH^\sigma (\bx)&:=
    \cH  (\bx) +\frac{2}{N} \sum_{i =1 }^N
   \Theta\left(N^{\frac{2}{3}-\xi}x_i \right), \quad  
  \Theta(u)= (u+1)^2\mathds{1}\{u<-1\}.\non
\end{align}
The local version  of the measure $\sigma$ is defined in the obvious way,
\be\label{sigyext}
  \sigma_\by (\rd\bx):=\frac{1}{Z^\sigma_\by} e^{-\beta N \cH^\sigma_\by(\bx)}\rd\bx, \quad
   \cH^\sigma_\by (\bx):=
    \cH_\by (\bx) +\frac{2}{N} \sum_{i \in I }
   \Theta\left(N^{\frac{2}{3}-\xi}x_i \right).
\ee
 For technical reasons we will also need the following variants of  $\sigma$ and
$\sigma_\by$ where we added slightly less convexity through $\Theta$:
\begin{align}\non
\wh \sigma(\rd\bx) & : =\frac{Z }{\wh Z^c} e^{-  \beta  \sum_{i \in I }
   \Theta\left(N^{\frac{2}{3}-\xi}x_i \right) } \mu(\rd\bx),   \\
    \wh\sigma_\by (\rd\bx)& :=\frac{1}{\wh  Z^c_\by} e^{-\beta N \wh\cH^\sigma_\by(\bx)}\rd\bx, \qquad
\wh\cH^\sigma_\by (\bx) := \cH^\sigma_\by +\frac{1}{N} \sum_{i \in I }
   \Theta\left(N^{\frac{2}{3}-\xi}x_i \right).\non
\end{align}
The measures $\sigma$, $\wh \sigma$ and their local versions
 depend on the parameters $V,\beta,K$ and $\xi$ but we do not carry this dependence in the
notation.

Rigidity estimates  proved for the global measure $\mu$ (Theorem \ref{thm:rigidity})
also hold for the  local measures $\sigma_\by$
provided $\by$ lies in the  set
of {\it ``good'' boundary conditions} that is defined as follows:
\be\label{Rdef1}
  \cR=\cR_K=\cR_{K,V,\beta}(\xi) : = \{ \by\; : \; |y_k-\gamma_k|\le N^{-2/3+ \xi} \hat k^{-1/3},
 \; k\not\in I \}.
\ee
 The rigidity exponent $\xi$ will always be chosen much smaller than the exponent $\delta$ in
\eqref{NK}. This guarantees that the typical length of the configuration interval,
$|J|\sim \gamma_{K}-\gamma_{1} \ge c (K/N)^{2/3}$,
be bigger than the largest rigidity precision, $N^{-\frac{2}{3}+\xi}$.

We will need the following two modifications of $\cR$. The first one
requires that $x_k$ be  good  in an expectation sense
w.r.t. $\sigma_\by$, 
and that $x_1$ is not too negative.
Thus we define the set
\be\label{R*def}
  \cR^*=\cR^*_{K,V,\beta}(\xi) :
= \{ \by \in\cR_{K}(\xi) \; :  \forall k\in I,\; \left|\E^{\sigma_{\bf y}}x_k-\gamma_k\right|\leq N^{-\frac{2}{3}+\xi}k^{-\frac{1}{3}},\, \P^{\wh\sigma_\by}(x_1\ge \gamma_1- N^{-\frac{2}{3}+\xi})  \ge 1/2 \}.
\ee
Notice that for technical reasons to be clear later on the constraint on
$x_1$ is w.r.t.  the measure $\wh\sigma_\by$. 
This condition will be  important in Sect. \ref{subsec:condRig}.

Another  modification adds the condition
of a level repulsion near the boundary, i.e., we define
\be\label{wtRdef}
   \cR^\#= \cR^\#_{K,V,\beta}(\xi) : = \{ \by\in \cR_{K, V, \beta}( \xi/3 ) \; : \;
 |y_{K+1}-y_{K+2}|\ge N^{-2/3-\xi}K^{-1/3}\}.
\ee

In the following  theorems  we establish
rigidity and level repulsion estimates for the local log-gas  $\sigma_\by$
with  good boundary conditions $\by$ up to the spectral edges.
These theorems  extend similar  estimates for the local measure $\mu_\by$ in the
 bulk of the spectrum established in
\cite{EYsinglegap}  to the edges for the measure $\sigma_\by$.

\begin{theorem}  [Rigidity estimate for local measures]  \label{thm:condRig}
Fix $\beta,\xi>0$ and, using the above notations, assume that
$ {\bf y}\in\mathcal{R}_{K}^*(\xi)$.
 Then there exists constants $C,c>0$ (independent of ${\bf y},K$) such that for
large enough $N$ we have, for any $k\in I$, and $u>0$,
\be\label{32the}
\P^{\sigma_{\bf y}}\left(|x_k-\gamma_k|>C N^{-\frac{2}{3}+\xi} k^{-\frac{1}{3}}u\right)\leq e^{- {c u^2}}.
\ee
\end{theorem}

As a side comment we remark  that the Gaussian decay in \eqref{32the} is an artifact of the
 additional confinement in the local measure $\sigma_\by$.
  For the measures $\mu$ or $\mu_\by$, the  tail probability
of $x_1$ has a slower decay $\exp{[- C (\gamma_1-x_1)^{3/2}]}$ in the regime $x_1 \ll \gamma_1$
in accordance with the tail behaviour of the
Tracy-Widom law  (for the Gaussian beta ensemble, see \cite{LedRid2010} for a detailed 
analysis of the edge tail behavior).  However, 
Theorem \ref{thm:local} below asserts  that $\sigma_\by$ has the correct distribution when
$x_1-\gamma_1 \sim N^{-\frac{2}{3}}$.

We also have the following level repulsion estimates.   Similar  bounds for the measure $\mu_\by$ in the bulk were proved
 in   \cite{BouErdYau2011, EYsinglegap}.

\begin{theorem}  [Level repulsion estimates for local measures] \label{lr2}
Let $\beta>0$, let $\xi$ be an arbitrary
fixed positive constant and assume that $K$ satisfies \eqref{NK}.
Then there are constants $C, c>0$ such that for  $\by \in \cR=\cR_{K}(\xi) $
and for any $s>0$ we have
\begin{align}\label{k5n}
\P^{ \sigma_{\by}}[ y_{K+1} -x_{K} \le s K^{-1/3}N^{-2/3}  ] & \le
  C \left (  K^2s \right ) ^{\beta + 1}, \\ \label{k5nlow}
\P^{ \sigma_{\by}}[ y_{K+1}-x_K \le s K^{-1/3} N^{-2/3} ] & \le
  C \left ( N^{C\xi} s \right ) ^{\beta + 1} + e^{-N^c}.
\end{align}
\end{theorem}

Note that these two bounds are complementary. The first one
 gives optimal level repulsion for arbitrary small $s$, but
the constant $K^2$ is not optimal. The second bound improves this constant but 
at the expenses of an exponentially small additive error.

We remark that statements similar to \eqref{k5nlow} hold for any gap $x_{i+1}-x_i$, not only for
the last one with $i=K$. The proofs are very similar, after
conditioning on the points $x_{i+1}, x_{i+2}, \dots, x_K$ being close to their
classical locations.

 We will prove the rigidity and level repulsion results only for $\sigma_\by$
since  these bounds are needed in the proof of the main theorems. The proof
of the level repulsion bounds for $\sigma_\by$, however, verbatim applies to
$\mu_\by$. For the rigidity bound,  from Theorem \ref{thm:condRig} there exists a set  $\cal Y$ of
almost full $\mu$-measure such that for any $\by \in {\cal Y}$ we have
$$
\P^{\mu_{\bf y}}\left(|x_k-\gamma_k|> N^{-\frac{2}{3}+\xi+\e} k^{-\frac{1}{3}}\right)\leq  e^{- N^c}.
$$
The role of the confinement in the definition of $\sigma$ is to prevent the
first particle $x_1$ to be very negative, since it would destroy
the good convexity bound on the Hessian.
The  reason we have to introduce $\sigma$ and $\sigma_\by$
is that in a technical step (establishing rigidity for the interpolation between local equilibrium measures
with two different boundary conditions, see Section~\ref{sec:inter}) we need
a superexponential decaying tail probability of the
rigidity estimate. We establish such bound only for the confined measure $\sigma_\by$
and not for $\mu_\by$.

\medskip

 Our main technical result, Theorem~\ref{thm:local} below,  asserts that,
for  $K$ in a restricted range, the local gap statistics
is essentially independent
of $V$ and $\by$ for good boundary conditions $\by$  (see \eqref{Rdef1}).
 For a fixed $\by\in \cR$, we define the classical locations $\alpha_j=\alpha_j(\by)$
of $x_j$ by the formula
\be
\label{aldef}
   \int_{0}^{\alpha_j}\varrho(s)\rd s = \frac{j}{K+1}   \int_{0}^{y_+}\varrho(s)\rd s,
\qquad j\in\llbracket 1,K\rrbracket,
\ee
i.e., $\alpha_j$'s are the $j$-th $(K+1)$-quantiles of the density in $J_\by$.
 Recall that  the support
of $\varrho$ starts from $A=0$ even though the configuration interval
starts from  minus infinity.

The core universality result on the local measures is the following theorem.
It compares two local measures with potentials $V$ and $\wt V$ and
external configurations $\by$ and $\wt\by$. For notational simplicity,
we will use tilde to refer to objects related to the measure $\wt\mu: =\mu_{\wt V}$.

\begin{theorem}[Edge universality for local measures]\label{thm:local}
Let  $\beta\ge 1$
and $V$, $\widetilde V$ be $\mathscr{C}^4$ be regular and satisfy (\ref{eqn:LSImu}) and (\ref{eqn:GrowthCondition}).
Assume 
that the equilibrium density $\varrho_V$ and $\varrho_{\widetilde V}$ are
supported on a single interval and satisfy (\ref{eqn:matchSing}).
Fix small positive parameters $\xi, \delta>0$ and  a parameter $0< \zeta < 1 $
that satisfy
\be\label{con1}
C_0\xi < \delta(1-\zeta), 
\ee
 with a sufficiently large universal  constant $C_0$,
and assume that
\be\label{Kcon1}
N^\delta\le K\leq N^{ 2/5-\delta}.
\ee
Then there is a small  $\chi >0$, independent  of  $N, K$, with the following property.
Let $\by\in \cR^\#_{K,V,\beta}(\xi)\cap \cR^*_{K,V,\beta}(\xi)$ and $\widetilde\by\in \cR^\#_{K,\widetilde V,\beta}(\xi)\cap \cR^*_{K,\widetilde V,\beta}(\xi)$ be
two different boundary conditions.
Fix $m\in \N$.
Then
for any $\Lambda\subset \llbracket 1, K^\zeta\rrbracket$,
$|\Lambda|=m$, and  any smooth, compactly supported
observable $O:\bR^m\to \bR$, we have for $N$ large enough
\be\label{eq:compy1}
  \Bigg| \E^{\sigma_\by} O\Bigg( \Big(  N^{2/3} j^{1/3}(x_j -\al_j)
\Big)_{j\in \Lambda}\Bigg)
  -   \E^{\wt \sigma_{\wt\by}} O\Bigg( \Big ( N^{2/3}  j^{1/3}(x_j -\wt \al_j)\Big)_{j\in \Lambda}\Bigg)\Bigg| \le
N^{- \chi}.
\ee
\end{theorem}

We remark that,  thanks to the conditions (\ref{eqn:matchSing}) and \eqref{Kcon1}, the points
 $\al_j$ and $\widetilde \al_j$ (defined by (\ref{aldef}) with $\varrho$ and $\widetilde\varrho$)
 in \eqref{eq:compy1} can both be replaced by $\gamma_j$.
 To see this, we claim that
for any $\by\in \cR_{K}$ we have
\be\label{35}
  \big | \alpha_j- \gamma_j\big | \le  C  \frac{j  N^{-1+ \xi} }{K \gamma_j^{1/2}}   \le C N^{ -2/3+ \xi} \,  \frac {  j^{2/3}} K,
\ee
and these estimates are more accurate than
the precision detected by the smooth observable $O$ in
\eqref{eq:compy1} for any $j\le K^\zeta$.
To prove \eqref{35},
we recall $\gamma_K \sim N^{-2/3} K^{2/3}$ and for   $\by \in   \cR$, we have
$ |y_{K+1}-\gamma_{K+1}|\le N^{-2/3+ \xi}  K^{-1/3}$.
Since the density has a square root singularity near $A=0$ \eqref{sqsing},  by assumption
 $ K \ge N^\delta \gg N^\xi$ we have  for  $\by \in   \cR$ that
$$
\Big |  \int_{0}^{y_+}\varrho(s)\rd s - \int_{0}^{ \gamma_{K+1}}\varrho(s)\rd s
 \Big | \le C \gamma_{K+1}^{1/2} N^{-2/3+ \xi}  K^{-1/3}
\le C N^{-1+ \xi}.
$$
Therefore, for   $\by \in   \cR$ we obtain that
$$
   \int_{0}^{\alpha_j}\varrho(s)\rd s = \frac{j}{K+1}   \Big [ \frac {K+1} N + \OO(N^{-1+ \xi}) \Big ]
   =   \int_{0}^{\gamma_j}\varrho(s)\rd s +  \frac{j }{K}\OO(N^{-1+ \xi}) .
$$
This implies \eqref{35}.  

\bigskip

As a consequence  of the proof of Theorem \ref{thm:local}, we also   have the  following  correlation decay estimate.

\begin{theorem}[Correlation decay near the edge]\label{thm:cordec}
Let  $\beta\ge 1$, $V$ be $\mathscr{C}^4$, regular,
and satisfy (\ref{eqn:LSImu}), (\ref{eqn:GrowthCondition}). Assume that $\varrho_V$ satisfies (\ref{eqn:matchSing}).
Fix small positive parameters $\xi, \delta>0$ and  assume (\ref{con1}), (\ref{Kcon1}).
Consider the local measure $\sigma_\by$ with
 $\by\in \cR^\#_{K,V,\beta}(\xi)\cap \cR^*_{K,V,\beta}(\xi)$.
Then there is a constant $C$, independent of  $N, K$,
such that for any two differentiable functions $f, q$ on $J_\by$  and large enough $N$, we have
\be\label{qfdec}
   \langle q(x_i) ; f(x_j)\rangle_{\sigma_\by} \le \frac{N^{C\xi}  }{ N^{4/3} j^{4/9}} \|q'\|_\infty
\|f'\|_\infty , \qquad i\le j\le K,
\ee
where $\langle f; g\rangle_\om := \E^\om fg - \E^\om f \, \E^\om g$ denotes the covariance.
In particular,
\be
   \Big\langle N^{2/3} i^{1/3}  (x_i-\gamma_i) ; N^{2/3}  j^{1/3} (x_j-\gamma_j)\Big\rangle_{\sigma_\by} \le
 \frac{N^{C\xi}i^{1/3}}{ j^{1/9}},
      \qquad i\le j\le K.
\label{ij}\ee
\end{theorem}

We remark that the rigidity estimate \eqref{32the}  shows that $ N^{2/3} i^{1/3}  (x_i-i^{2/3})\le N^{C\xi}$
with a very high probability. Therefore, as long as $i\ll j^{1/3}$, \eqref{ij} is stronger than the trivial bound
$$
  \big\langle N^{2/3} i^{1/3}  (x_i-\gamma_i) ;  N^{2/3} j^{1/3} (x_j-\gamma_j)\big\rangle_{\sigma_\by}
  \le N^{C\xi}
$$
obtained from the rigidity estimate.
We believe that the optimal estimate on the correlation decay  is of the following form:
\be
  \big\langle N^{2/3} i^{1/3}  (x_i-\gamma_i) ;  N^{2/3} j^{1/3} (x_j-\gamma_j)\big\rangle_{\sigma_\by} \lesssim
  \left ( \frac{ i}{ j} \right )^{1/3},
      \qquad i\le j\le K,
\label{opdecay}
\ee
and the same decay rate holds for the global measures $\sigma$ and $\mu$.
A heuristic argument  that this is the optimal decay rate, at least w.r.t. the GUE measure,
 will be given in  Appendix~\ref{app:cordec}. It is based on  an
extension of the argument in  \cite{Gus2005}.
We note that this  decay is quite different from the logarithmic correlation decay
 in the bulk
$$
    \big\langle N (x_i-\gamma_i) ;  N (x_j-\gamma_j)\big\rangle_\mu \sim \log \frac{N}{|i-j|+1}
$$
which is proven for the GUE measure $\mu$ in  \cite{Gus2005} and conjectured to hold for
other ensembles as well.

\medskip

Theorem~\ref{thm:local} is our key  result.
In  Sections \ref{sec:beta} and \ref{sec:wigner} we will show how to use
Theorem~\ref{thm:local} to prove the main Theorems~\ref{thm:wigner}
and \ref{thm:beta}. The proofs of these two theorems follow  the arguments
used in \cite{EYsinglegap}.
The proofs of the auxiliary
Theorem \ref{thm:condRig}
will be given in Subsection \ref{subsec:condRig}, and Theorem \ref{lr2} in Appendix \ref{app:levelRep}.
The proof of Theorem~\ref{thm:local} will start from Section \ref{sec:AnLocGibbs}
 and will continue until the end of the paper.

\section{Edge universality of beta ensembles:  proof of Theorem~\ref{thm:beta} }\label{sec:beta}

In this section, we shall  use the edge universality
 Theorem~\ref{thm:local} to prove
global edge universality Theorem~\ref{thm:beta}.
Recall the definition of the measure  $ \sigma $ with normalization factor $Z_\sigma$.
 We start with he following lemma on properties of $\sigma$, defined by (\ref{sig}).

\begin{lemma}\label{lm:comm}
For any bounded observable $O$ we have
\begin{align} \label{com}
   \Big |    \E^\sigma O - \E^\mu  O  \Big |  \le   \| O \|_\infty e^{- N^{c}} .
\end{align}
In particular, this implies that  $\mu$ and $\sigma$ have the same local statistics and
 $ \sigma$ also satisfies the following rigidity estimate: for any $\xi>0$ there exists $N_0$ and $c>0$ such that for all $N\geq N_0$, $k\in\llbracket 1,N\rrbracket$, we have
\begin{equation}\label{rigcons}
\P^\sigma\left(|\lambda_k-\gamma_k|> N^{-\frac{2}{3}+ \xi }(\hat k)^{-\frac{1}{3}}
\right)\leq \ e^{- N^c}.
\end{equation}
Moreover,
\be\label{RR}
 \P^\sigma( \cR^\#_K(\xi)\cap \cR^*_K(\xi))\ge  1-  N^{-c'}
\ee
with some positive constant $c'>0$.
\end{lemma}

\begin{proof}
Clearly, by $\Theta\ge 0$, we have
the relation $Z^\sigma\le Z$ among the normalization constants for $\sigma$ and $\mu$.
For a lower bound,
from the rigidity estimate \eqref{rig} for $x_1$
we have
$$
1 \ge \frac {Z^\sigma} Z  = \int    e^{- 2 \beta  \sum_{i }
   \Theta\left(N^{\frac{2}{3}-\xi}x_i \right) } \rd \mu
\ge \P^{\mu} ( x_1> - N^{-\frac{2}{3}+\xi} )  \ge  1 - e^{- N^{c}}.
$$
For any bounded nonnegative observable  $O$, we have from the rigidity estimate on $\mu$ that
$$
  \E^\mu  O -   \| O \|_\infty \P^\mu  (x_1 \le  - N^{-\frac{2}{3}+\xi})    \le \frac Z {Z^\sigma}
   \E^\mu  O   \mathds{1}(x_1 > - N^{-\frac{2}{3}+\xi})
   \le   \E^\sigma O \le   ( 1 - e^{- N^{c}})^{-1} \E^\mu  O.
$$
 Using this separately
 for the positive and negative parts of an arbitrary bounded observable, this  proves \eqref{com}.
From the rigidity   estimate \eqref{rigcons} we have
\be
   \P^{ \sigma} ( \cR_{K+1}(\xi/3))\ge 1-  \exp{(- N^{c})}
\label{muR}
\ee
(notice that the index of $\cR$ is $K+1$ instead of $K$
and we use $\xi/3$ instead of $\xi$ for later convenience).
Furthermore,   for any  $\by\in \cR_{K+1}(\xi/3)$,
the level repulsion estimate  w.r.t. $ \sigma_\by$
  in the form proved  in \eqref{k5nlow} implies
\be\label{sssy}
   \P^{\sigma_\by}\big[ y_{K+2}- x_{K+1} \le sN^{-2/3}k^{-1/3}\big] \le C\big(N^{7\xi/9}s\big)^{\beta+1}
\ee
for $s \ge \exp(-K^\theta)$. Using \eqref{muR}, we see that
\eqref{sssy} also  holds \ with $\sigma_\by$ replaced by $\sigma$.   Applying this with
  $s=N^{-\xi} \ll N^{-7\xi/9}$, we have
\be\label{c3}
   \P^{ \sigma}  ( \cR^\#_{K})\ge 1-  N^{-c}.
\ee
 The estimates \eqref{com}--\eqref{c3} also   hold for the measure $\wh \sigma$
instead of $\sigma$ with the same proof.

{F}rom  the rigidity estimate w.r.t. $ \sigma$, \eqref{rigcons},   we have for any $\e > 0$ that
\be\label{32}
\E^{ \sigma} 
\P^{ \sigma_{\by}} \left ( |x_j-\gamma_j| \le N^{-\frac{2}{3}+\e} j^{-\frac{1}{3}} \; : \;\forall j\in I \right) =
\P^{ \sigma} \left (   |x_j-\gamma_j| \le N^{-\frac{2}{3}+\e} j^{-\frac{1}{3}} \; : \; \forall j\in I \right)  \ge1- e^{- N^{c}}.
\ee
By  the estimate \eqref{35} on  $\gamma_j-\alpha_j$, \eqref{32} also holds
if $\alpha_j$ is replaced by $\gamma_j$.
From the rigidity estimate w.r.t. $  \wh\sigma$, we have
$$
 \P^{ \wh\sigma }  (A )\ge 1 - e^{-N^c}, \quad A:=
   \{ \by \in\cR_{K}(\xi) \, :  \P^{\wh\sigma_\by}(x_1\ge \gamma_1- N^{-\frac{2}{3}+\xi})
  \ge 1/2 \}.
$$
By \eqref{com} we have  and also the parallel  version  with $\sigma$ replaced by $\wh \sigma$, we have
$$
 |\P^{\wh \sigma }  (A) - \P^{ \sigma } ( A) |  \le e^{-N^c}.
$$
This guarantees that the second constraint in
  the definition of $\cR^*$ from \eqref{R*def} is satisfied for a set of $\by$'s
with a high $\sigma$-probability.
The first constraint is easily satisfied for a large set of $\by$'s by
 the rigidity w.r.t. $\sigma$.
Thus  we obtain
\be\label{muR*}
   \P^{ \sigma} (\cR^*_{K})\ge 1-   e^{-N^c}.
\ee
Combining \eqref{c3} and \eqref{muR*}, we obtain \eqref{RR}.
\end{proof}

\begin{proof}[Proof of Theorem~\ref{thm:beta}]
Fix a configuration $\wt \by  \in \wt \cR^\#_K \cap\wt\cR^*_K$ where $\wt \cR_K := \cR_{K, \tilde V,\beta}$ and
with similar notations for the other sets.
Thus we can  take  expectation of \eqref{eq:compy1} with respect to $\sigma$
and use \eqref{RR} to have
$$
  \Bigg| \E^{ \sigma}  \mathds{1}_{ \by  \in \cR^\#_K \cap\cR^*_K}  \E^{  \sigma_{\by}} O\Bigg( \Big(  N^{2/3} j^{1/3}(x_j -\al_j (\by))\Big)_{j\in \Lambda}\Bigg)
  -   \E^{ {\wt \sigma}_{\wt\by}} O\Bigg( \Big ( N^{2/3}  j^{1/3}(x_j -\wt \al_j (\wt \by) )\Big)_{j\in \Lambda}\Bigg)\Bigg| \le N^{-\chi},
$$
where we have explicitly indicated the dependence of $\alpha_j$ on $\by$.
From  \eqref{35} and $j\le K^{\zeta}$ we have
$$
N^{2/3} j^{1/3}| \al_j(\by) - \gamma_j | \le N^\xi  j K^{-1} \le N^{-\chi}
$$
provided that
$$
 N^{ \xi+ \chi} K^{\zeta-1} \le 1.
$$
This condition is guaranteed by the condition \eqref{con1} if $\chi>0$ is chosen sufficiently small.
Under this condition, we have thus proved that
$$
   \Bigg| \E^{ \sigma}  \mathds{1}_{ \by  \in \cR^\#_K \cap\cR^*_K}  \E^{ \sigma_{\by}}  O\Bigg( \Big( N^{2/3} j^{1/3}(x_j -\gamma_j )\Big)_{j\in \Lambda}\Bigg)
  -  \E^{ {\wt \sigma}_{\wt\by}}O\Bigg( \Big ( N^{2/3}  j^{1/3}(x_j -\wt \al_j)\Big)_{j\in \Lambda}\Bigg)\Bigg| \le N^{-\chi}.
$$
Recall that the $\sigma$-probability of
the complement of the set $ \cR^\#_K  \cap \cR^*_K$ is  small, see  (\ref{RR}),  and
we can choose $\chi < c'$ where $c'$ is the constant in \eqref{RR}.
  Together with the fact that   $O$ is bounded, we can drop
 the characteristic function $ \mathds{1}_{ \by  \in \cR^\#_K \cap\cR^*_K}$ at a negligible error and we have
 $$
     \Bigg| \E^\sigma O\Bigg( \big( N^{2/3} j^{1/3}(x_j-\gamma_j)\Big)_{j\in \Lambda}\Bigg)
      - \E^{{\wt \sigma}_{\wt\by}}O\Bigg(\Big( N^{2/3} j^{1/3} (x_j-\wt\al_j(\wt\by))\Big)_{j\in \Lambda}\Bigg)\Bigg|\le N^{-\chi}.
  $$
  We can now repeat the same argument for the tilde variables. Taking expectation
over $\wt\by$ with respect to $\wt\sigma$,
  we see that $\E^{{\wt\sigma}_{\wt\by}}$ can be replaced
  with $\E^{\wt\sigma}$ with a negligible error. Finally, using \eqref{com} we can replace $\sigma$ with $\mu$ and
  $\wt\sigma$ with $\wt\mu$.
 This proves the global edge universality Theorem~\ref{thm:beta}.
\end{proof}

\section{Edge universality of Wigner matrices:  proof of
Theorem~\ref{thm:wigner} }\label{sec:wigner}

We will first prove Theorem~\ref{thm:wigner} under the assumption that the matrix elements of
the normalized matrix satisfy
a uniform subexponential decay \eqref{subexp}.
This will be done in the following two steps.
First we show that edge universality
holds for Wigner matrices with a small Gaussian component.
This argument is based upon
the analysis of the Dyson Brownian Motion (DBM).
 In the second step we remove the small Gaussian component by a moment matching
perturbation argument.

\subsection{Edge universality with a small Gaussian component}

We first recall  the notion of
Dyson's Brownian motion.
It describes the evolution of the eigenvalues
of a  flow of Wigner matrices, $H=H_t$,
if each matrix element $h_{ij}$ evolves according to independent
(up to symmetry restriction)
Ornstein-Uhlenbeck processes.
 In the Hermitian case, this process for the rescaled matrix elements
 $v_{ij}: = N^{1/2}h_{ij}$ is given by the
stochastic differential equation
$$
  \rd v_{ij}= \rd {\rm B}_{ij} - \frac{1}{2} v_{ij}\rd t, \qquad
i,j\in\llbracket 1,N\rrbracket
$$
where $ {\rm B}_{ij}$,  $i <  j$, are independent complex Brownian
motions with variance one and ${\rm B}_{ii}$ are real
Brownian motions of the same variance.  The real symmetric
case is analogous, just $\beta_{ij}$ are real Brownian motions.

Denote the distribution of
the eigenvalues $\bla=(\lambda_1, \lambda_2,\dots, \lambda_N)$
 of  $H_t+2$  at  time $t$
by $f_t (\bla)\mu (\rd \bla)$
where the Gaussian measure $\mu$ is given by \eqref{eqn:measure}
with $V(x) =\frac{1}{2} (x-2)^2$.  (This simple  shift ensures that  the convention $A=0$
made at the beginning of Section~\ref{sec:loc} holds.)
The density $f_t=f_{t,N}$ satisfies the forward equation
\be\label{dy}
\partial_{t} f_t =  \cL f_t,
\ee
where
\be
\cL=\cL_N:=   \sum_{i=1}^N \frac{1}{2N}\partial_{i}^{2}  +\sum_{i=1}^N
\Bigg(- \frac{\beta}{4} \lambda_{i} +  \frac{\beta}{2N}\sum_{j\ne i}
\frac{1}{\lambda_i - \lambda_j}\Bigg) \partial_{i}, \quad
\partial_i=\frac{\partial}{\partial\lambda_i},
\label{L}
\ee
with $\beta=1$ for the real symmetric case and $\beta=2$
in the complex hermitian case.
The initial data $f_0$ given by the original generalized Wigner matrix.
The main result of this section is that edge universality
holds for the measure $f_t\mu$ if $t$ is at least a small negative power of $N$.

Note that, in this section, we always consider the cases $\beta=1$ or $2$, although the proof of the following theorem could be adapted to general $\beta\geq 1$.

\begin{theorem}\label{lm:COMP} Let  $\mu$ be
the Gaussian beta ensemble, \eqref{eqn:measure}, with quadratic $V$,
and $f_t$ be the solution of \eqref{dy} with initial data $f_0$ given by
the original generalized Wigner matrix.
 Fix an integer $m>0$ and  $\kappa<1/4$.
  Then
there are positive constants $\fb$ and $\chi$ 
such that
for any $t\ge N^{-\fb}$   and for any compactly supported smooth observable $O$
we have
$$
   \Big| \big[\E^{f_t \mu}  -\E^{\mu}\big]   O\Big( N^{2/3}  {p_1}^{1/3}(x_{p_1}
 - \gamma_{p_1} ), \dots, N^{2/3} p_m^{1/3}(x_{p_m} - \gamma_{p_m}
)\Big)\Big|
    \le C N^{-\chi},
$$
for any $p_1 , \dots, p_m  \le  N^{\kappa } $.
\end{theorem}

For any $\tau>0$  define   an  auxiliary  potential $W=W^\tau$ by
$$
    W^\tau(\bla): =    \sum_{j=1}^N
W_j^\tau (\lambda_j)   , \qquad W_j^\tau (\lambda) := \frac{1}{2 \tau } (\lambda_j -\gamma_j)^2.
$$
The parameter $\tau>0$ will be chosen as $\tau \sim  N^{-\fa}$ where $\fa$ is some positive
exponent with $\fa <\fb$.

\begin{definition} \label{def:locallyConstrained}
We define the probability measure $\rd\mu^{\tau}:= Z_{\tau}^{-1} e^{- N \beta  \cH^\tau} $, where
the total Hamiltonian is given by
$$
  \cH^\tau: = \cH +W^\tau.
$$
 Here   $\cH$ is the Gaussian Hamiltonian given by \eqref{eqn:measure}
 with $V(x) = x^2/2$
and  $Z_\tau=Z_{\mu^{\tau}}$ is the partition function.
The measure
$\mu^{\tau}$ will be referred to as the relaxation measure.
\end{definition}

Denote by $Q$ the  following quantity
$$
 Q:= \sup_{t\ge0 }
 \frac{1}{N}
 \int \sum_{j=1}^N(\lambda_j-\gamma_j)^2
 f_t( \bla )\mu(\rd \bla).
$$
Since  $H_t$ is a generalized Wigner matrix for all $t$,  the following rigidity estimate
(Theorem 2.2 \cite{EYYrigi} and  Theorem 7.6 \cite{EKYYsparse2})
holds:
\begin{equation}\label{rig1}
\P^{f_t \mu} \left(|\lambda_k-\gamma_k|> N^{-\frac{2}{3}+ \delta \xi }(\hat k)^{-\frac{1}{3}}
\right)\leq e^{-N^c}.
\end{equation}
where $\gamma_k$ is computed w.r.t. the semicircle law
and we have used $\delta \xi$ as the small positive exponent  needed in the rigidity estimate \cite{EKYYsparse2}
so that $N^{\delta \xi} \le K^\xi$.  Together with a trivial tail estimate
from \eqref{subexp},
\be\label{tail}
\P^{f_t\mu}(|\lambda_i|\ge s)
 \le N^2 \P^{f_t\mu}( |h_{ij}(t)|\ge s) \le N^2 \exp(- (s/\sqrt{N})^c), \qquad s>0,
\ee 
this 
implies that
$$
  Q\le N^{-2+2\nu}
$$
for any $\nu>0$ if $N\ge N_0(\nu)$ is large enough.

Recall the definition of the Dirichlet form w.r.t. a probability measure $\om$
\be\label{Ddef}
 D^\om(\sqrt{g}):= \sum_{i=1}^ND_i^\om(\sqrt{g}), \qquad
 D_i^\om(\sqrt{g}):= \frac{1}{2N} \int  |\partial_i \sqrt{ g}|^2 \rd\om
 =   \frac{1}{  8  N}\int  |\partial_i \log g |^2 g\rd\om,
\ee
and the definition of the relative entropy of two probability measures $g\om$ and $\om$
$$
  S(g\om|\om) := \int g\log g \rd\om.
$$
The $1/N$ prefactor in the definition of the Dirichlet form as
well as in \eqref{L}
 originates from the $N^{-1/2}$-rescaling of the matrix elements $h_{ij} = N^{-1/2} v_{ij}$.

By the Bakry-\'Emery criterion \cite{BakEme1983}),
 the local relaxation measure satisfies the logarithmic Sobolev inequality, i.e.,
$$
  S(f\mu^\tau |\mu^\tau)\le C\tau^{-1} D^{\mu^\tau}(\sqrt{f})
$$
for any  probability measure $f\mu^\tau$.

Now we recall Theorem 2.5 from \cite{EYBull} (the equation (2.37) in  \cite{EYBull}
 has a typo and the correct
form should be $S(f_\tau \mu|\om) \le C N^m$).  This theorem was first proved in \cite{EYY1};
a closely related result was obtained earlier in
in \cite{ESY4}.

\begin{lemma}\label{thm1}  Let $0<\tau\le 1$ be
a (possibly $N$-dependent) parameter.
Consider the local relaxation
measure $ \mu^{\tau}$.
Set $\psi:= \frac {  d \mu^{\tau}} { d \mu } $ and
let  $g_t: = f_t/\psi$.
Suppose there is a constant $m$ such that
\be\label{entA}
S(f_{\tau} \mu  | \mu^\tau )\le CN^m.
\ee
Fix an $\e'>0$.  Then for any $ t \ge \tau N^{\e'}$
the entropy and the Dirichlet form satisfy the estimates:
\be\label{1.3}
S(g_t \mu^\tau | \mu^\tau) \le
 C   N^2    Q \tau^{-1}, \qquad
D^{\mu^\tau} (\sqrt{g_t})
\le CN^2  Q \tau^{-2},
\ee
 where the constants depend on  $\e'$  and $m$.
\end{lemma}

We remark that the condition \eqref{entA} is trivially satisfied
in our applications for any  $\tau\ge N^{-2/3+\xi}$
since
\be\label{ZZ}
S(f_\tau \mu|\mu^\tau )\le S(f_\tau \mu|\mu )+ \log (Z_\tau/Z) + N\int W^\tau(\bla) f_\tau(\bla)\rd\mu(\bla)
\ee
and
$S(f_\tau \mu|\mu ) \le S(H_\tau|H_\infty) = N^2 S((h_\tau)_{ij}| (h_\infty )_{ij}) \le CN^m$,
where $H_\infty$ is the GOE/GUE matrix.
The other two terms in \eqref{ZZ} satisfy
a similar bound by  \eqref{rig1}.

Recall  the  probability  measure $\sigma$ \eqref{sig} and   define $q_t$ by
$$
q_t \sigma = f_t \mu = g_t \mu_\tau.
$$
From  \eqref{rig1}--\eqref{tail}  and \eqref{1.3} (and recalling that we have shifted the eigenvalues in such a way
that the left spectral edge $-2$ is now shifted to $0$),  we can check that
\be\label{513}
D^\sigma (\sqrt {q_t})  \le  2  D^{\mu^\tau} (\sqrt {g_t})  + C N^{4/3}  \sum_j  \E^{f_t \mu}  |\nabla    \Theta (N^{2/3 - \xi} x_j)|^2\le 2N^2 Q \tau^{-2} + e^{-N^c}\ee
for any $t\ge \tau N^{\e'}$.

Recall that $\sigma_\by$ denotes the conditional measure of $\sigma$ given $\by$ and  $\cH^\sigma_\by $
its  Hamiltonian \eqref{sigyext}.
The Hessian of $\cH^\sigma_\by $  satisfies for all $\by \in \cR_K$ and all $\bu\in\R^K$ that
\be\label{Hyconv}
  \langle \bu,  (\cH^\sigma_\by )''  \bu \rangle  \ge      \Bigg[
N^{4/3-  2 \xi}  \sum_{j \in I}   \mathds{1} ( x_j \le - N^{-2/3 + \xi} )   u_j^2   +   \sum_{j \in I}
V''(x_j ) u_j^2 + \frac 1 N \sum_{j \in I, k \in I^c} \frac {u_j^2}  { (x_j -y_k)^2} \Bigg] \ge
 c N^{1/3}  K^{-1/3} \sum_{j\in I} u_j^2.
\ee
In this estimate we used that $V''$ is bounded from below, see \eqref{eqn:LSImu}, and that
$$
 \frac 1 N  \sum_{k \in I^c} \frac 1 { (x-y_k)^2}  \sim
\frac 1 N   \sum_{k \ge K+1} \frac 1 { (N^{ \xi-2/3}+N^{-2/3} k^{2/3})^2} \ge c N^{1/3} K^{-1/3}
$$
holds
for any $x\ge - N^{ - 2/3+\xi}$ and
$\by\in \cR_K$.

\bigskip

Define $ q_{t, \by} $ to be the conditional density of $ f_t \mu = q_t \sigma  $ w.r.t. $\sigma_\by$ given $\by$,
i.e., it is defined by the relation $  q_{t, \by}  \sigma_\by = (f_t\mu)_\by$.
From the bound \eqref{Hyconv} we have the logarithmic Sobolev inequality
\be\label{lsi}
S( q_{t, \by}  \sigma_\by | \sigma_\by)  \le  C   \frac {K^{1/3}} {N^{1/3}}
 \sum_{i\in I} D_i^{\sigma_\by} ( \sqrt { q_{t, \by}} ).
\ee
Combining it with the entropy inequality,  we have
\be\label{lsi2}
\int \rd \sigma_\by | q_{t, \by} - 1| \le  C \sqrt { S( q_{t, \by}  \sigma_\by | \sigma_\by)}
  \le C  \sqrt {  \frac { K^{1/3}} {N^{1/3}}
 \sum_{i\in I} D_i^{\sigma_\by} ( \sqrt { q_{t, \by}} )}  .
\ee
The following Lemma controls
 the Dirichlet forms $D_i^{\sigma_\by}$ for most external
configurations $\by$.

\begin{lemma}\label{lm:rig1} Fix $0<\fa  \le 1$,   $\xi, \nu >0$,  and $\tau \ge N^{-  \fa}$. Suppose  the initial data $f_0$
of the DBM is given by a generalized Wigner ensemble.
Then,  for any  $\e, \e'>0$
and
 $ t \ge \tau N^{\e'}$
there exists a set $\cG_{ K, t }\subset\cR_{K}(\xi)$  of good boundary conditions $\by$
with
\be\label{PG}
   \P^{f_t \mu} (\cG_{K,t} )\ge 1-CN^{-\e}
\ee
such that for any $\by\in \cG_{ K,t}$ we have
\be\label{29}
\sum_{i\in I} D_i^{\sigma_\by}( \sqrt { q_{t, \by}} ) \le  C N^{3\e +2\fa+2\nu }.
\ee
Furthermore, for any bounded observable  $O$, we have
\be\label{300}
  \big| [\E^{ q_{t, \by} \sigma_\by  } - \E^{\sigma_\by} ] O(\bx)  \big|
 \le C K^{ 1/6} N^{2\e +\fa+ \nu -   1/6}.
\ee
 We also have
\be\label{30}
  \E^{  q_{t, \by}  \sigma_\by  } |x_k - \gamma_k| \le { CN^{-2/3+\xi } k^{-1/3}}, \qquad k\in I.
\ee
The same bounds hold if $\sigma_\by$ and $q_{t,\by}$ are replaced with $\wh\sigma_\by$
and $\wh q_{t, \by}$ where $\wh q_t$ is defined by $\wh q_t \wh\sigma =f_t \mu$.
\end{lemma}

\begin{proof} In this proof, we omit the subscript $t$, i.e., we use $f=f_t$, etc.
By definition of the conditional measure and by  \eqref{1.3} and \eqref{513}, we have for any $\nu> 0$ that
$$
  \E^{ q \sigma}   \sum_{i\in I} D_i^{\sigma_\by}( \sqrt { q_{t, \by}}) =  \sum_{i\in I} D_i^{\sigma}( \sqrt q )
\le  N^2 Q \tau^{-2} + e^{-N^c}  \le   C N^{2\fa+2\nu}.
$$
Therefore, by the Markov inequality, \eqref{29}
holds for all $\by$ in a set $\cG^1_{K}$ with $\P^{f\mu}(\cG^1_{ K})\ge 1- CN^{-3\e}$.
Recall from the rigidity estimate \eqref{rig} that  $ \P^{q \sigma} (\cR_{K}^c)=  \P^{f \mu} (\cR_{K}^c) $ is exponentially small.
Hence  we can choose $\cG^1_{K}$ such that  $\cG^1_{K}\subset \cR_{K}$.
The estimate \eqref{300} now follows from \eqref{29}, \eqref{lsi} and \eqref{lsi2}.

Similarly,
the rigidity bound \eqref{rig1}
with respect to $f \mu$ can  be  translated to the measure  $f_\by\mu_\by$
for most $\by$,  i.e.,
there exists a set $\cG^2_{ K}\subset\cR_K$ with
$$
  \P^{q \sigma}(\cG^2_{ K})  =  \P^{f\mu}(\cG^2_{ K}) \ge   1- \exp{\big( -N^{c}\big)},
$$
such that for any $\by\in \cG^2_{ K} $ and
for any $k\in I$, we have
\be\label{520}
   \P^{ q_{ \by}\sigma_\by} \Big(  |x_k - \gamma_k| \ge  N^{-2/3} K^\xi k^{-1/3} \Big) \le  \exp{\big(-N^{c}\big)}.
\ee
In particular, by
setting $\cG_{K} := \cG^1_{K}\cap  \cG^2_{K}$ we can conclude \eqref{30}
 for any $\by\in \cG_{K}$.
This  proves the lemma.
\end{proof}

\begin{lemma}\label{ec}   Fix $0<\fa  <1/6$, $\xi, \nu >0$,  and $\tau \ge N^{-  \fa}$.
 Suppose  the initial data $f_0$ of the DBM is given by a generalized Wigner ensemble.
Then, for any  $\e'>0$,
 $ t \ge \tau N^{\e'}$, $ k \in I$ and   $\by\in   \cG_{K,t}$
 (defined in Lemma \ref{lm:rig1}), we have
\be\label{Ex1}
 \big|\E^{\sigma_\by}x_k - \gamma_k\big|\le      N^{ -2/3}  k^{-1/3} K^\xi,
 \qquad k\in I,
\ee
provided that
\be\label{22}
 K^{1/3} N^{-5/6+\nu+\fa+2\e'}
 \le N^{-2/3}K^{-1/3+ \xi}.
\ee
\end{lemma}

Notice that  we need
$
 \fa < 1/6
$
in order that \eqref{22} has a solution with $K \to \infty$.
 In our application
we will choose  $\fa$ arbitrarily close to 0,
 then we can take any $K$ with $K \le  N^{1/4-\delta}$ and still find sufficiently small
positive exponents $\nu, \fa,\e'$ with $\fa+\e' \le \fb$ so that  \eqref{22} holds. We will not trace
the precise interrelation among these exponents.
This explains the restriction $\kappa< 1/4$ in Theorem~\ref{thm:wigner}.

The following proof is essentially the same as  the one for Lemma 5.5 in \cite{EYsinglegap}.

\begin{proof}
We claim that the estimate \eqref{Ex1} follows from
\be\label{33}
 | \E^{\sigma_\by} x_k - \E^{  q_{t, \by} \sigma_\by } x_k|\le K^{1/3} N^{-5/6+\nu+\fa+2\e'} .
\ee
To see this,  we have
$$
 \big|\E^{\sigma_\by}x_k - \gamma_k\big|\le  | \E^{\sigma_\by} x_k - \E^{  q_{t, \by} \sigma_\by } x_jk|
 +  | \E^{  q_{t, \by} \sigma_\by } x_k - \gamma_k |
 \le    N^{ -2/3}  k^{-1/3} K^\xi,
$$
where we have used    \eqref{30}, \eqref{33}
and   \eqref{22}.
To prove \eqref{33},  we run the reversible dynamics
$$
   \partial_s h_s = \cL_\by h_s
$$
starting from initial data $ h_0=   q_{t, \by} $,  where the generator $ \cL_\by$ is the unique
reversible generator with the Dirichlet form $D^{\sigma_\by}$, i.e.,
$$
-\int f \cL_\by \,  g  \,  \rd \sigma_\by = \sum_{i \in I} \frac {1}{2 N} \int   \nabla_i f \cdot \nabla_i g \, \rd \sigma_\by.
$$

Recall  that from the convexity bound \eqref{Hyconv},  $\tau_K = K^{1/3}/N^{1/3}$  is an upper  bound for the  time
to equilibrium of this dynamics.  After
 differentiation and integration we get,
$$
\Big [\E^{  q_{t, \by} \sigma_\by } -   \E^{\sigma_\by}  \Big ]  (x_k-\gamma_k) =
\int_0^{K^{\e'} \tau_K } \rd u \frac{1}{2N} \int
   (\partial_k  {h_u} ) \rd \sigma_\by+ O(\exp{(-cK^{\e'})}).
$$
{F}rom the Schwarz inequality  with a free parameter $R$, we can bound the last line by
$$
 \frac{1}{N} \int_0^{K^{\e'} \tau_K } \rd u  \int \Big(  R (\partial_k \sqrt{h_u} )^2 +h_u R^{-1}  \Big)
\rd\sigma_\by+  O(\exp{(-cK^{\e'})}).
$$
Dropping the trivial subexponential error term and
using that the time integral of the Dirichlet form
is bounded by the initial entropy, we can bound the last line by
$$
   RS( q_{t, \by} \sigma_\by | \sigma_\by) + \frac{K^{\e'} \tau_K   }{NR}.
$$
Using the logarithmic Sobolev inequality for $\sigma_\by$ and  optimizing
the parameter $R$, we can bound the last term by
\begin{align}\non
  \Big| \E^{\sigma_\by} x_k - \E^{  q_{t, \by} \sigma_\by }  x_k \Big|  & \le \tau_K  R
 \sum_{i\in I} D_i^{\sigma_\by} (\sqrt{ q_{t, \by}})  + \frac{K^{\e'} \tau_K }{NR}+ O(\exp{(-cK^{\e'})}) \nonumber  \\ \non
&  \le  \frac{ K^{\e'}  \tau_K }{\sqrt N}\Big(  \sum_{i\in I} D_i^{\sigma_\by} (\sqrt{ q_{t, \by}})  \Big)^{1/2}+ O(\exp{(-cK^{\e'})}).
\end{align}
 Combining this bound with \eqref{29} with the choice $\e=\e'$, we
obtain \eqref{33}.
\end{proof}

We note that
if we applied \eqref{300}
with the special choice $O(\bx) = x_k$ to control \eqref{33},
then the error estimate would have been much worse.
We stress that \eqref{Ex1} is not an obvious fact although we know that it holds
for $\by$ with high probability w.r.t. the equilibrium measure $\mu$.  The key point of \eqref{Ex1}
is that it holds for any $\by\in \cG_{ K}$, i.e., for a set of $\by$'s with "high probability" w.r.t  $f_t \mu$!
 We also remark that \eqref{Ex1} holds only in the sense of expectation of $x_k$ and have not yet established that
$$
 \E^{\sigma_\by} \big| x_k - \gamma_k\big|\le      N^{ -2/3}   k^{-1/3} K^\xi,
 \qquad k\in I.
$$
We will finally prove this estimate (Theorem~\ref{thm:condRig}) but only after we prove
 the rigidity estimate for  $\sigma_\by$.

\medskip

We can now prove the main result of this section.

\begin{proof}[Proof of Theorem~\ref{lm:COMP}]
We will consider only the case $m=1$ since  the general case is only notationally more involved.
From the assumption \eqref{22}  the right hand side of  \eqref{300}
is smaller than $K^{-1/2}$. Choosing $\chi$ sufficiently small,
we thus  have
\be\label{302}
   \Big| \big[\E^{ (f_{t} \mu)_\by}  -\E^{\sigma_\by }\big] O\Big( N^{2/3}  {p}^{1/3}(x_{p}
 - \gamma_{p} )\Big )\Big| \le C K^{-1/2}
    \le CN^{-\chi},
\ee
for all $\by\in \cG_{ K}$  and $p \le K^\zeta $ with the $f_t\mu$-probability of  $\cG_{K}$ satisfying  \eqref{PG}.

We now apply Theorem \ref{thm:local} to the same Gaussian beta ensemble with two different
boundary conditions so  that
$$
   \Big| \big[\E^{\sigma_\by}  -\E^{  \sigma_{\wt \by} }\big] O\Big( N^{2/3}  {p}^{1/3}(x_{p}
 - \gamma_{p} )\Big ) \Big|
    \le CN^{-\chi},
$$
for all $\by, \wt \by \in \cR^\#_{K} \cap \cR^\ast_{ K} $ and $p \le K^\zeta  $.
Since
 $\P^{\sigma}( \cR^\#_{ K} \cap \cR^\ast_{ K})\ge 1 - N^{-c'}$  (see \eqref{RR}),
taking the expectation of $\wt \by$ w.r.t. $\sigma$ we   have thus proved that
$$
    \Big| \big[\E^{\sigma_\by}  -\E^{  \sigma }\big] O\Big( N^{2/3}  {p}^{1/3}(x_{p}
 - \gamma_{p} )\Big ) \Big|
    \le CN^{-\chi}.
$$
We know from \eqref{com} that
$$
 \Big| \big[\E^{\sigma}  -\E^{  \mu }\big] O\Big( N^{2/3}  {p}^{1/3}(x_{p}
 - \gamma_{p} )\Big ) \Big| \le CN^{-\chi}.
$$
Together with \eqref{302},  we thus have
\be \label{30311}
    \Big| \big[\E^{(f_{t} \mu)_\by}  -\E^{  \mu }\big] O\Big( N^{2/3}  {p}^{1/3}(x_{p}
 - \gamma_{p} )\Big ) \Big|
    \le CN^{-\chi},
\ee
for all $\by \in \cG_K \cap  \cR^\#_{K} \cap \cR^\ast_{ K}$.
Once we prove that
\be\label{591}
\P^{f_t \mu }
  ( \cG_K \cap  \cR^\#_{K} \cap \cR^\ast_{ K} )\ge 1 - N^{-\chi}
\ee
then by averaging \eqref{30311} in $\by$ w.r.t. $f_t \mu$ we have
$$
   \Big| \big[\E^{f_{t} \mu}  -\E^{ \mu }\big] O\Big( N^{2/3}  {p}^{1/3}(x_{p}
 - \gamma_{p_1} )\Big )\Big|
    \le CN^{-\chi}.
$$
and this proves  Lemma~\ref{lm:COMP}.

Finally, we have to prove \eqref{591}.
 By  \eqref{PG} we have that $\P^{f_t \mu } \cG_{K} \ge 1 - N^{-\e}$. We now prove that
similar inequality holds for the set $\cR^\#_{K}$ and show that $\cG_{K}  \subset \cR^*_K$. This will conclude \eqref{591}
and complete the proof of  Lemma~\ref{lm:COMP}.

\medskip

\noindent
{\it Step 1:} We first prove that
\be\label{c9}
   \P^{f_t \mu} ( \cR^\#_{K})\ge 1-  N^{-c'}.
\ee
Since $f_t \mu$ represents the probability distribution of a generalized Wigner matrix ensemble,
from the rigidity estimate \eqref{rig1}, we have
\be
   \P^{f_t \mu} ( \cR_{K+1  })\ge 1- \exp{(-N^{c})}.
\label{muR1}
\ee
 From  the level repulsion estimate   \eqref{k5nlow} with $k = K+ 1$,
we have for any $\by \in \cR_{K+1}$ that
$$
\P^{ \sigma_{\by}} [  y_{K+2} - x_{K+1} \le s N^{ -2/3}  K^{-1/3}   ] \le
  C \left ( N^{7\xi'/3}  s  \right ) ^{\beta + 1}.
$$
Applying \eqref{300} with $O(\bx) = \mathds{1}(  y_{K+2} - x_{K+1} \le s N^{ -2/3}  K^{-1/3})$  and using the condition \eqref{22}, we obtain a similar  estimate w.r.t. the measure
$ (f_{t} \mu)_{\by}$, i.e.,
\be\label{k58}
\P^{ (f_{t} \mu)_{\by}} [  y_{K+2} - x_{K+1} \le s N^{ -2/3}  K^{-1/3}   ] \le
  C \left ( N^{7\xi'/3}  s  \right ) ^{\beta + 1} +  C K^{ 1/6} N^{2\e' +\fa+ \nu -  1/6}.
\ee
This  estimate \eqref{k58}  and the  bound \eqref{muR1} with $K+1 $ replaced by $K$
 imply   \eqref{c9}
provided  $7\xi'/3 \ll \xi$ and \eqref{22} is satisfied.

\medskip
\noindent
{\it Step 2:} We now prove that
$$
    \cG_{K} \subset \cR^\ast_{ K}.
$$
By Lemma \ref{ec},   the inequality
$\left| \E^{\sigma_{\bf y}}x_k-\gamma_k\right|\leq N^{-\frac{2}{3}+\xi}k^{-\frac{1}{3}}$ holds for all $\by \in \cG_{K}$.  This verifies the first defining condition of $\cR^*$.
  To   check the other defining condition
of $\cR^*_K$, we now show that
\be\label{5211}
\P^{\wh \sigma_\by} \big[\Omega\big]
    \ge 1/2, \quad \Omega := \{ x_1\ge \gamma_1- N^{-\frac{2}{3}+\xi} \}
\ee
holds for $\by \in \cG_{K}$.
To prove \eqref{5211},  for $\by \in \cG_{K}$ we have from \eqref{300} (applied to $\wh\sigma_\by$)
 that
$$
| \P^{ \wh\sigma_\by} \Omega  - \P^{ \wh q_{t, \by} \wh\sigma_\by} \Omega |
 \le C K^{1/6} N^{2\e' +\e + \fa - 1/6}.
$$
Under the assumption \eqref{22}, the right hand side of the last equation vanishes
as $N\to\infty$.
Thus we have
$$
 \P^{ \wh \sigma_\by} \big[ \Omega\big]  \ge  \P^{ \wh q_{t, \by} \wh  \sigma_\by} \big[\Omega \big] - 1/4.
$$
From \eqref{520},
we have  $\P^{\wh q_{t, \by}  \wh \sigma_\by} \Omega \ge 1 - e^{-N^{c}} $ and thus
$ \P^{\wh \sigma_\by} \big[\Omega\big]  \ge  1/2$  for $\by \in \cG_{K}$.
\end{proof}

\subsection{Removal of the Gaussian convolution}

\newcommand{\f}[1]{\boldsymbol{\mathrm{#1}}}

The last step to complete the proof of edge universality is to approximate arbitrary Wigner matrices by a Gaussian divisible
ensemble. We will need  the following result.

\begin{theorem}[Universality of extreme eigenvalues, Theorem 2.4 of \cite{EYYrigi}] \label{twthm}
Suppose that we have
two  $N\times N$  generalized Wigner matrices, $H^{(v)}$ and $H^{(w)}$, with matrix elements $h_{ij}$
given by the random variables $N^{-1/2} v_{ij}$ and
$N^{-1/2} w_{ij}$, respectively, with $v_{ij}$ and $w_{ij}$ satisfying
the uniform subexponential decay condition \eqref{subexp}. Let $\P^\bv$ and
$\P^\bw$ denote the probability and $\E^\bv$ and $\E^\bw$
the expectation with respect to these collections of random variables.
Suppose that Assumptions {\bf (A)} and  {\bf (B)}  hold for both ensembles.  If
the first two   moments of
 $v_{ij}$ and $w_{ij}$ are the same, i.e.,
$$
    \E^\bv \bar v_{ij}^l v_{ij}^{u} =  \E^\bw \bar w_{ij}^l w_{ij}^{u},
  \qquad 0\le l+u\le 2,
$$
then there is an $\e>0$ and $\delta>0$
depending on $\vartheta$ in \eqref{subexp}
 such that
or any real parameter $s$ (may depend on $N$)
we have
$$
 \P^\bv ( N^{2/3} ( \lambda_N -2) \le s- N^{-\e} )- N^{-\delta}
  \le   \P^\bw ( N^{2/3} ( \lambda_N -2) \le s )   \le
 \P^\bv ( N^{2/3} ( \lambda_N -2) \le s+ N^{-\e} )+ N^{-\delta}
$$
for $N\ge N_0$ sufficiently  large, where $N_0$ is independent of $s$.
Analogous result holds for the smallest eigenvalue $\lambda_1$ and also for extensions to the joint distributions
of any finite number of eigenvalues $\lambda_{N-i_1}, \dots, \lambda_{N-i_k}$ as long as $ |i_k| \le N^\e$  (or similar results
for  the smallest eigenvalues).

\end{theorem}

Given Theorem~\ref{twthm}, we can now complete the proof of the edge universality for generalized
 Wigner matrices with
subexponential decay.
Recall that $H_t$ is the generalized Wigner matrix  whose    matrix elements evolve by independent
 OU processes.
In Theorem~\ref{lm:COMP}
we have proved that the  statistics of eigenvalues at the spectral edge of   $H_t$,  for $t\ge N^{-\fa +\e'}$
and for {\it any} initial  generalized matrix $H$,
 is the same as the standard  Gaussian one in the corresponding symmetry class.
We  now construct an auxiliary Wigner matrix $ H_0$
 (see, e.g., Lemma 3.4 of \cite{EYYBernoulli}
which allows us to match )
such that  the first two  moments of
$ H_t$  (with $t = N^{-c'}$ for some small $c'> 0$)
and   the first two moments of the {\it original} matrix $H^{\f v}$ are identical.
 The edge statistics of $H^{\f v}$ and $H_t$ coincide by Theorem~\ref{twthm}
and the edge statistics of $H_t$
 are identical to those of  the standard GOE/GUE by Theorem~\ref{lm:COMP}.
 This completes our proof of Theorem \ref{thm:wigner}.
\qed

\section{Rigidity of the particles}

Most of this section is devoted to proving  Theorem~\ref{thm:rigidity} which asserts
the rigidity of the particles under the measure $\mu$ at the optimal scale up to the edge 
(which, for us, means a  control throughout the support of
the equilibrium measure including the edge).
We recall that the same statement holds for the measure $\sigma$ (Lemma~\ref{lm:comm}).

 Our method to prove rigidity is a multiscale analysis, initiated for the bulk 
particles in \cite{BouErdYau2011,BouErdYau2012}. 
It is a bootstrap argument where concentration and accuracy bounds are
proved in tandem, gradually for smaller and smaller scales. 
Concentration bound means a control on the fluctuation of a particle around its mean;
this is obtained by a local logarithmic Sobolev inequality
(for non-convex $V$ we need an extra convexification argument). To estimate the log-Sobolev
constant we use  rigidity on a larger scale. The next step is to identify the
mean, this is achieved by   the first loop equation, where the error term
involves the improved concentration bound. This leads to a better accuracy
and thus better rigidity. This information can be used to improve the
concentration bound on a smaller scale, etc. 
In this paper we prove rigidity up to the edge,
which involves new difficulties: the loop equation  is less stable since
the density vanishes near the edge.
 Moreover, the loop equation 
is used to improve the accuracy of one specific particle (the leftmost one, $\lambda_1$), whose
rigidity cannot originate in the pairwise interaction from surrounding particles.

This extra difficulty (lack of a natural boundary on the left) is also critical in
the last subsection, where we prove Theorem \ref{thm:condRig}, i.e., the rigidity of
the particles under the conditional measure  $\sigma_{\bf y}$ with a Gaussian tail. Extra convexity 
(hence rigidity) on the left of the first particle is the reason for introducing the modification $\sigma_\by$ of $\mu_{\bf y}$
which artificially confines the first particle.

Another extra difficulty consists in improving the accuracy without assuming that $V$ is analytic. This
analyticity condition was essential 
in the works \cite{Joh1998,Shc2011} and the previous 
optimal bulk rigidity estimates \cite{BouErdYau2011,BouErdYau2012}. It turns out that the
analyticity condition can be replaced by  a much weaker smoothness assumption by
a more careful  analysis of the first loop equation, 
see (\ref{eqn:loopSmooth}) and (\ref{eqn:loopModified}).

In this section  we disregard the shift convention which sets $A=0$.

\subsection{Statement of the results}

For any fixed $N$,
let the {\it classical position}  $\gamma^{(N)}_k$ of the $k$-th particle under $\mu^{(N)}$ be
defined by
\begin{equation}\label{gammadefN}
 \int_{-\infty}^{\gamma^{(N)}_k}\varrho_1^{(N)}(s)\rd s=\frac{k}{N},
\end{equation}
where $\varrho_1^{(N)}$ is the density of $\mu^{(N)}$.
Recall that $\gamma_k$ from \eqref{gammadef}  denotes the limiting classical location.

\begin{definition} In the following definitions, the potential $V$ and $\beta>0$ are fixed.
\begin{enumerate}
\item We say that {\bf rigidity} at scale $a$ holds if for
any $\e>0$, there are  constants
$c>0$ and $N_0$ such that for any
$N\geq N_0$ and $k\in\llbracket 1, N\rrbracket$  we have
$$
\P^\mu\left(|\lambda_k-\gamma_k|> N^{-\frac{2}{3}+a+\e}(\hat k)^{-\frac{1}{3}}\right)\leq \ e^{- N^c }.
$$
\item We say that {\bf concentration} at scale $a$ holds if for
any $\e>0$, there are constants
$c>0$ and $N_0$ such that for any
$N\geq N_0$ and $k\in\llbracket 1, N\rrbracket$  we have
$$
\P^\mu\left(|\lambda_k-\E^\mu(\lambda_k)|> N^{-\frac{2}{3}+a+\e}(\hat k)^{-\frac{1}{3}}\right)\leq \ e^{- N^c}.
$$
\item We say that {\bf accuracy} at scale $a$ holds if for
any $\e>0$, there is a constant $N_0$  such that for any $N\geq N_0$
and $k\in\llbracket 1,N\rrbracket$ we have
$$
\left|\gamma_k-\gamma_k^{(N)}\right|\leq
N^{-\frac{2}{3}+a+\e}(\hat k)^{-\frac{1}{3}}.
$$
\end{enumerate}
\end{definition}

  For the proof of Theorem \ref{thm:rigidity} the main steps
are the concentration and accuracy improvements hereafter,
proved in the following subsections.

\begin{proposition}\label{prop:ImprConc}
Let $V$ be $\mathscr{C}^2$,
regular with equilibrium density
supported on a single interval $[A,B]$,  and satisfy  (\ref{eqn:LSImu}), (\ref{eqn:GrowthCondition}).
Then rigidity at scale $a$ implies concentration at scale $a/2$.
\end{proposition}

\begin{proposition}\label{prop:ImprAcc}
Let $V$ be $\mathscr{C}^4$,
regular with equilibrium density
supported on a single interval $[A,B]$,  and satisfy  (\ref{eqn:LSImu}), (\ref{eqn:GrowthCondition}).
Then rigidity at scale $a$ implies accuracy at scale $11a/12$.
\end{proposition}

\begin{remark} Notice that the accuracy improves from scale $a$  only to 
scale $11a/12$ instead of $3a/4$ as it was achieved in the bulk case (see Proposition 3.13 in \cite{BouErdYau2011}).
This weaker control is due to some difficult estimates near the edge that have not been optimized.
\end{remark}

\begin{proof}[Proof of Theorem \ref{thm:rigidity}]
It is known that rigidity at scale 1 holds.
More precisely, for any
 $\e>0$ there are positive constants $c_1$, $c_2$ such that, for all $N\geq 1$,
\begin{align}
\P^\mu\left( \exists k\in\llbracket1,N\rrbracket\mid
 | \lambda_k-  \gamma_k| \ge \e \right)\leq c_1 e^{-c_2 N}. \label{eqn:largDev1}
\end{align}
For eigenvalues in the bulk, (\ref{eqn:largDev1}) follows from the large
 deviations for the empirical spectral measure with speed $N^2$, see
\cite{BenGui1997, AndGuiZei2010}. For the extreme eigenvalues the large deviations principle
with speed $N$
is proved in \cite{BenDemGui2001} for the GOE case, and extended in
\cite{AndGuiZei2010} Theorem 2.6.6, for the general case (up to a condition on the partition function
that follows from Theorem 1 (iii) in \cite{Shc2011}).

We now use Propositions
\ref{prop:ImprConc} and \ref{prop:ImprAcc}  to  obtain that
concentration and accuracy hold at scale $11/12$.
We just need to prove that
concentration and accuracy at some scale $b>0$
imply rigidity the same scale $b$.
 Then a simple induction on scales  shows that rigidity
holds on scale $(11/12)^m$ for any integer $m$, i.e., it holds
at any positive scale $\xi$.

To show the key part of the induction step,
 assume that concentration and accuracy hold at scale $b$.
Fix any $k\in \llbracket 1, N \rrbracket$.
Then for any $\varepsilon>0$ we have
\begin{multline*}
\E^\mu\#\Big\{\lambda_i\leq \E^\mu(\lambda_k)-N^{-\frac{2}{3}+b+\frac{\varepsilon}{2}}(\hat k)^{-\frac{1}{3}}\Big\}=
\sum_{\ell=1}^N\P^\mu\Big\{\lambda_{\ell}<\E^\mu(\lambda_k)-N^{-\frac{2}{3}+b+\frac{\varepsilon}{2}}(\hat k)^{-\frac{1}{3}}\Big\}\\
\leq k-1+(N-k+1)\P^\mu\Big\{\lambda_{k}<\E^\mu(\lambda_k)-N^{-\frac{2}{3}+b+\frac{\varepsilon}{2}}(\hat k)^{-\frac{1}{3}}\Big\}\leq k
\end{multline*}
for large enough $N$,
independently of $k$,  since the probability in the last line
is subexponentially small  by concentration on scale $b$. As $\gamma_k^{(N)}$ is defined by
$
\E^\mu(\#\{\lambda_i\leq \gamma_k^{(N)}\})=k,
$
this implies that $\gamma_k^{(N)}\geq \E^\mu(\lambda_k)-N^{-\frac{2}{3}+b+\frac{\varepsilon}{2}}(\hat k)^{-\frac{1}{3}}$ for some large enough $N$, independent of $k$. In the same way one can get the upper bound, which yields
$$
|\gamma_k^{(N)}-\E^\mu \lambda_k|\leq N^{-\frac{2}{3}+b+\varepsilon}(\hat k)^{-\frac{1}{3}}
$$
 for large enough $N$. As we have accuracy at scale $b$, the same conclusion holds when replacing $\gamma_k^{(N)}$
by $\gamma_k$. We thus proved,
for any $\varepsilon>0$,
the existence of some $C>0$ such that for all $N$ and $k$ we have
\begin{equation}\label{eqn:ExpTyp}
|\gamma_k-\E^\mu\lambda_k|\leq C N^{-\frac{2}{3}+b+\frac{\varepsilon}{2}}(\hat k)^{-\frac{1}{3}}.
\end{equation}
The conclusion now easily follows from
\begin{multline*}
\P^\mu\Big\{|\lambda_k-\gamma_k|\geq N^{-\frac{2}{3}+b+\varepsilon}(\hat k)^{-\frac{1}{3}}\Big\}
\\ \leq
\P^\mu\left\{|\lambda_k-\E^\mu\lambda_k|
\geq \frac{1}{2}N^{-\frac{2}{3}+b+\varepsilon}(\hat k)^{-\frac{1}{3}}\right\}
+\1\left(|\gamma_k-\E^\mu\lambda_k| \geq
\frac{1}{2} N^{-\frac{2}{3}+b+\varepsilon}(\hat k)^{-\frac{1}{3}}\right).
\end{multline*}
The first term can be bounded by the concentration hypothesis, the second term is 0 for large enough $N$, thanks to (\ref{eqn:ExpTyp}).
\end{proof}

\subsection{Initial estimates for non-analytic potentials}

Let $h$ be a continuous and bounded function. Consider the probability distribution on the simplex $\lambda_1\leq\dots\leq \lambda_N$ given by
$$
\mu^{(N,h)}(\rd\bla)\sim e^{-\beta (N\mathcal{H}(\lambda)+\sum_{k=1}^N h(\lambda_k))}\rd\bla,
$$
where 
$\mathcal{H}$ is defined in (\ref{01}). We denote by $m_{N,h}$ the Stieltjes transform for the measure
$\mu^{(N,h)}$:   
\begin{equation}\label{mhdef}
m_{N,h}(z)=\E^{\mu^{(N,h)}}\left(\frac{1}{N}\sum_{k=1}^N\frac{1}{z-\lambda_k}\right).
\end{equation}
In the following, it will be useful to  have the density supported strictly in a compact interval:
for given $\kappa>0$,
define the following variant of $\mu^{(N,h)}$ conditioned to have all particles in $[A-\kappa,B+\kappa]$:
\begin{equation}\label{eqn:truncMeasure}
\mu^{(N,h,\kappa)}(\rd\lambda)=\frac{1}{Z_{N,\kappa}}
\prod_{1\leq i<j\leq N}|\lambda_i-\lambda_j|^\beta\prod_{k=1}^N 
e^{-\beta\left(\frac{N}{2}V(\lambda_k)+h(\lambda_k)\right)}\mathds{1}_{\lambda_k\in[A-\kappa,B+\kappa]}
\rd\lambda_1\dots\rd\lambda_N.
\end{equation}

We will choose $\kappa$ to be small, fixed number.
Let  $\varrho_k^{(N,h,\kappa)}$ denote
the correlation functions and $m_{N,h,\kappa}(z)$ the Stieltjes transform, defined in the same way as
(\ref{eqn:corrFunct}) and \eqref{mhdef}, but for the underlying measure $\mu^{(N,h,\kappa)}$.
Then Lemma 1 in \cite{BouPasShc1995} (strictly speaking 
this result is given in \cite{BouPasShc1995} only for 
$h\equiv 0$, but the proof works for any fixed $h$)  states that
under condition (\ref{eqn:GrowthCondition}), for some large enough $\kappa$ there exists
 some $c>0$, depending only on $V$,  such that for any $x_1,\dots,x_k\in[A-\kappa,B+\kappa]$, we have
\begin{equation}\label{eqn:BPS1}
\left|\varrho^{(N,h,\kappa)}_k(x_1,\dots,x_k)-\varrho^{(N,h)}_k(x_1,\dots,x_k)\right|
\leq \varrho^{(N,h,\kappa)}_k(x_1,\dots,x_k)e^{-c N},
\end{equation}
and for $x_1,\dots,x_j\not\in[A-\kappa,B+\kappa]$, $x_{j+1},\dots,x_k\in[A-\kappa,B+\kappa]$,
\begin{equation}\label{eqn:BPS2}
\varrho^{(N,h)}_k(x_1,\dots,x_k)\leq e^{-c N\sum_{i=1}^{j}\log |x_i|}.
\end{equation}
The estimates (\ref{eqn:BPS1}) and (\ref{eqn:BPS2})
actually also hold for arbitrarily small fixed $\kappa>0$
thanks to the large deviations estimates (\ref{eqn:largDev1}),  which holds not only for $\mu^{(N)}$ but also for $\mu^{(N,h)}$.  From now we fix this small parameter $\kappa>0$. The following Lemma  relates estimates on $m_{N,h}-m$ and concentration of linear statistics of the particles.

\begin{lemma}\label{lem:concFar}
Let $V$ be $\mathscr{C}^4$, 
regular such that the equilibrium density $\varrho_V$ is
supported on a single interval $[A,B]$  and satisfies
 (\ref{eqn:LSImu}), (\ref{eqn:GrowthCondition}). 
Let  $h_1,h_2$  be $\mathscr{C}^2$ functions such that 
$\|h_1\|_\infty,\|h_1'\|_\infty,\|h_1''\|_\infty<\infty$, 
 and the same for $h_2$.
Let $a\in(0,1/2)$ and $\e>0$.  Assume that for any $\Im(z)=\eta\in(N^{-1/2},N^{-a})$ and $s\in(-\beta,\beta)$ we have
\begin{equation}
\left|m_{N,(1+s) h_1}(z)-m(z)\right|\leq C \frac{N^\e}{|(z-A)(z-B)|^{1/2}\,N\eta^2}\label{eqn:estim1}.
\end{equation}
Then there is
a constant $c>0$ such that, for any $N\geq 1$, we have
\begin{equation}\label{eqn:h}
\Prob^{\mu^{(N, h_1)}}\left(\left|\sum_{k=1}^N h_2(\lambda_k)-N\int h_2(s)\varrho(s)\rd s\right|>N^{2a+ 2\e}\right)\leq e^{-N^c}.
\end{equation}
\end{lemma}

\begin{proof}
Let $\kappa>0$ be a small constant and 
$C>0$  be chosen such that for any $E\in I_\kappa:=[A-\kappa,B+\kappa]$, $\eta\in(0,N^{-a})$, 
and $s\in(-\beta,\beta)$ we have
\begin{align}
\left|m_{N,(1+s) h_1}(z)-m(z)\right|&\leq C\frac{(\log N)^{1/2}}{N^{1/2}\eta}\label{eqn:estim2}.
\end{align}
This inequality was proved in \cite{PasShcBulk}, Theorem 2.3 (ii) Let $\chi$ be a
smooth nonnegative cutoff function;   $\chi=1$ 
on $[0,N^{-a}/2]$, $\chi=0$ on $[N^{-a},\infty)$,
$\|\chi'\|_\infty=\OO(N^{a})$. Let $\widetilde h_2$ be $\mathscr{C}^2$, compactly supported on
$I_\kappa$, such that $h_2=\widetilde h_2$ on $I_{\kappa/2}$, for some $\kappa>0$. From the large 
deviations estimate (\ref{eqn:largDev1}) we have, for any $\e>0$,
$$
\Prob^{\mu^{(N, h_1)}}\left(\left|\sum_{k=1}^N (h_2(\lambda_k)-\widetilde h_2(\lambda_k))\right|>N^\e\right)\leq e^{-N^c}.
$$
for some $c>0$.
As a consequence, to prove (\ref{eqn:h}), we can assume that $h_2$ is supported on $I_\kappa$.
By the Helffer-Sj\"ostrand formula (see formula (B.13) in \cite{ERSY}),
$$
\int h_2(u)(\varrho_1^{(N,(1+s) h_1)}(u)-\varrho(u))\rd u=
\OO\left(\iint_{x\in I_\kappa,\eta>0}(\eta\chi(\eta)+|\chi'(\eta)|)\left|m_{N,(1+s) h_1}(x+\ii\eta)-m(x+\ii\eta)\right|\rd x\rd\eta\right).
$$
The term involving $\chi'$ can be evaluated using (\ref{eqn:estim1}), and is bounded by 
$N^{-1+2a}$. For the $\chi$ term, we bound $m_{N,(1+s) h_1}(z)-m(z)$ by 
(\ref{eqn:estim1}) if $\eta\geq N^{-1/2}$ and
by (\ref{eqn:estim2}) if $\eta\in(0,N^{-1/2})$. We obtain
$$
\int h_2(u)(\varrho_1^{(N,(1+s) h_1)}(u)-\varrho(u))\rd u=\OO\left(\frac{N^{2a}}{N}\right).
$$
The remainder of the proof is a classical argument: using the above estimate we get 
$$
\frac{\rd}{\rd s}\log \E^{\mu^{(N,h_1)}}\left(e^{s\left(\sum_{k=1}^N h_2(\lambda_k)-
N\int h_2(u)\varrho(u)\rd u\right)}\right)=\E^{\mu^{(N,(1+s/\beta)h_1)}}\left(\sum_{k=1}^N h_2(\lambda_k)-N\int h_2(u)\varrho(u)\rd u\right)
=\OO(N^{2a}).$$
This yields
$$
\E^{\mu^{(N,h_1)}}\left(e^{\sum_{k=1}^N h_2(\lambda_k)-N\int h_2(s)\varrho(s)\rd s}\right)
+
\E^{\mu^{(N,h_1)}}\left(e^{-\left(\sum_{k=1}^N h_2(\lambda_k)-N\int h_2(s)\varrho(s)\rd s\right)}\right)
\leq e^{c N^{2a}},
$$
and one concludes by the exponential  Markov inequality. 
\end{proof}

The following lemma provides almost optimal estimates for $m_{N,h}-m$ for $\eta=\Im(z)$ till 
order 1. For non-analytic $V$, it improves previous estimates by Pastur and 
Shcherbina by a factor $\sqrt{N}$, and relies on their initial estimates proved in \cite{PasShcBulk}.

\begin{lemma}\label{lem:keyFar}
Let $V$ be $\mathscr{C}^4$,
regular such that the equilibrium density $\varrho_V$ is
supported on a single interval $[A,B]$  and satisfy
 (\ref{eqn:LSImu}), (\ref{eqn:GrowthCondition}). 
Let $h$ be a $\mathscr{C}^2$ function with $\|h\|_\infty,\|h'\|_\infty,\|h''\|_\infty<\infty$.
Then for any $\e>0$ there exists a constant $C=C(V,\e,\|h'\|_\infty)$ such that, for any 
$E\in[A-\kappa,B+\kappa]$, $\eta\in (0,N^{-\e})$, we have
$$
|(z-A)(z-B)|^{1/2}\left|m_{N,h}(z)-m(z)\right|\leq C\ \frac{N^\e}{N\eta^2}.
$$
\end{lemma}

\begin{proof}
Let $I_\kappa=[A-\kappa,B+\kappa]$ and $d(\xi)=\inf_{s\in I_\kappa}|\xi-s|$. 
Thanks to the estimates (\ref{eqn:BPS1}) and (\ref{eqn:BPS2}), we just need to prove the lemma for the Stieltjes transform $m_{N,h,\kappa}$ instead of $m_{N,h}$.

For any $a\in (0,1)$, let $\mathcal{P}(a)$ be the following property: for any $\e>0$ 
there exists a constant $C=C(V,a, \e,\|h'\|_\infty)$ such that, for any
$E\in[A-\kappa,B+\kappa]$, we have
\begin{align}
&|(z-A)(z-B)|^{1/2}\left|m_{N,h,\kappa}(z)-m(z)\right|\leq C\ \frac{N^\e}{N\eta^2}\ \quad \mbox{for}\ \eta\in(0,N^{-a}),
\label{Pa1}\\
&|(z-A)(z-B)|^{1/2}\left|m_{N,h,\kappa}(z)-m(z)\right|\leq C\ \frac{ N^{2a+\e}}{N}\ \quad  \mbox{for}\ \eta\in[N^{-a},1].
\label{Pa2}
\end{align}
We will prove that $\mathcal{P}(a)$ implies $\mathcal{P}(a/2)$, which concludes the proof of the lemma
by induction, as $\mathcal{P}(1/2)$ holds: 
Pastur and Shcherbina (see \cite{PasShcBulk}\footnote{Strictly speaking these estimates were proved for $h\equiv 0$,
but the analysis in \cite{PasShcBulk} extends to our context in a straightforward way, when
$\|h\|_\infty,\|h'\|_\infty,\|h''\|_\infty<\infty$.} Theorem 2.3 (ii)):
proved that
\begin{equation}\label{eqn:estPas1}
m_{N,h,\kappa}(\xi)-m(\xi)=\OO\left(\frac{(\log N)^{1/2}}{N^{1/2}d(\xi)}\right), \
\frac{1}{N^2}\var_{ \mu^{(N, h, \kappa)} }\left(\sum_{k=1}^N\frac{1}{\xi-\lambda_k}\right)=\OO\left(
\frac{\log N}{N d(\xi)^2}
\right),
\end{equation}
the second estimate being useful later along the proof. 
Here we used that $\eta|(z-A)(z-B)|^{1/2}\le d(z)$.

Assume that $\mathcal{P}(a)$ holds. To prove $\mathcal{P}(a/2)$, 
we will need the quasi-analytic extension of $V$ of order three:
\begin{equation}
\widetilde V(z)=V(E)+\ii\eta V'(E)-\frac{\eta^2}{2}V''(E). 
\end{equation}
Note that $V = \wt V$ on the real axis.
One easily checks that 
\begin{equation}\label{eqn:quasiAnalytic}
\partial_{\bar z} \partial_E \widetilde V(z)=-\frac{\eta^2}{2} V^{(4)}(E).
\end{equation}
The first loop equation and its limit are (see \cite{Joh1998, Eyn2003, Shc2011} for various proofs), for any $\xi\not\in\mathbb{R}$,
\begin{align}
&m_{N,h}(\xi)^2+\int_{\mathbb{R}}\frac{V'(s)+N^{-1}h'(s)}{\xi-s}\varrho_1^{(N, h)}(s)\rd s
=\frac{1}{N}\left(\frac{2}{\beta}-1\right)m_{N,h}'(\xi)+\frac{1}{N^2}\var_{\mu^{(N,h)}}
\left(\sum_{k=1}^N\frac{1}{\xi-\lambda_k}\right), \label{eqn:loopNfixed}\\
&m(\xi)^2+\int_{\mathbb{R}}\frac{V'(s)}{\xi-s}\varrho(s)\rd s=0.\label{eqn:loopAsympt}
\end{align}
Here $\var_{\mu} X : = \E^\mu X^2 - (\E^\mu X)^2$, in particular $\var_\mu X$ may be complex.
We choose to write the difference of both equations in the following way:
\begin{multline*}
(m_{N,h}(\xi)-m(\xi))^2+(2m(\xi)-\partial_E \widetilde V(\xi))(m_{N,h}(\xi)-m(\xi))
+\int_{\mathbb{R}}\frac{\partial_E\widetilde V(\xi)-V'(s)}{\xi-s}(\varrho_1^{(N,h)}(s)-\varrho(s))\rd s\\
+\frac{1}{N}\int_{\mathbb{R}}\frac{h'(s)}{\xi-s}\varrho_1^{(N,h)}(s)\rd s
-\frac{1}{N}\left(\frac{2}{\beta}-1\right)m_{N,h}'(\xi)-\frac{1}{N^2}\var_{\mu^{(N,h)}}\left(\sum_{k=1}^N\frac{1}{\xi-\lambda_k}\right)
=0.
\end{multline*}
Thanks to the estimates (\ref{eqn:BPS1}) and (\ref{eqn:BPS2}), the above equation also holds when all considered 
quantities are with respect to the measure $\mu^{(N,h,\kappa)}$ instead of $\mu^{(N,h)}$,
up to an exponentially small error term which is uniform in $\{d(\xi)>N^{-10}\}$:
\begin{align}\label{eqn:loopSmooth}
&(m_{N,h,\kappa}(\xi)-m(\xi))^2+(2m(\xi)-\partial_E \widetilde V(\xi))(m_{N,h,\kappa}(\xi)-m(\xi))
+b_N(\xi)-c_N(\xi)=\OO\left(e^{-c N}\right),\\
&b_N(\xi):=\int_{\mathbb{R}}\frac{\partial_E\widetilde V(\xi)-V'(s)}{\xi-s}(\varrho_1^{(N,h,\kappa)}(s)-\varrho(s))\rd s,\\
&c_N(\xi):=-\frac{1}{N}\int_{\mathbb{R}}\frac{h'(s)}{\xi-s}\varrho_1^{(N,h,\kappa)}(s)\rd s
+\frac{1}{N}\left(\frac{2}{\beta}-1\right)m_{N,h,\kappa}'(\xi)+\frac{1}{N^2}\var_{\mu^{(N,h,\kappa)}}
\left(\sum_{k=1}^N\frac{1}{\xi-\lambda_k}\right).
\end{align}
Take $z$ such that
 $\Im{z}=\eta\in(N^{-a},1)$, let $\delta\in(N^{-a}/4,\eta/2)$ be chosen later, and consider the  
domain $\Omega_\delta=\{\xi\mid d(\xi)\leq\delta\}$, and $\partial\Omega_\delta$ its boundary, 
encircling $I_\kappa$ but not $z$. We also use the notation, for $\xi\not\in I_\kappa$,
$$
r(\xi)=\frac{((A-\xi)(B-\xi))^{1/2}}{2 m(\xi)-\partial_E \widetilde V(\xi)},
$$
where the branch of the numerator is chosen so that $((A-\xi)(B-\xi))^{1/2}\sim \xi$
as $|\xi|\to\infty$.
One can check that $r$ is continuous in $\mathbb{C}$: thanks to the equilibrium 
equation $m(s)=\frac{1}{2}V'(s)$ ($s\in[A,B]$) the real part of $2 m(\xi)-\partial_E \widetilde V(\xi)$ 
vanishes on $[A,B]$, and the imaginary parts of the numerator and the denominator both change signs 
across $[A,B]$. Moreover, thanks to the square root singularity of $\rho$ at $A^-$ and $B^+$, the following  bounds easily hold:
\begin{equation}
c\leq |r(\xi)|\leq c^{-1} \label{eqn:r1}
\end{equation}
uniformly in $\Omega_\eta$, for some $c>0$.
Multiplying $(\ref{eqn:loopSmooth})$ by $r(\xi)$ and integrating counterclockwise,   one can write
\begin{align}
\int_{\partial \Omega_\delta}\frac{(m_{N,h,\kappa}(\xi)-m(\xi))((A-\xi)(B-\xi))^{1/2}}{z-\xi}\rd \xi&=
\int_{\partial \Omega_\delta}\frac{-(m_{N,h,\kappa}(\xi)-m(\xi))^2+c_N(\xi)}{z-\xi} r(\xi)\rd \xi\label{eqn:loop1}\\
&-  \int_{\partial \Omega_\delta}\frac{b_N(\xi)}{z-\xi}r(\xi)\rd \xi+\OO(e^{-c N}).\label{eqn:loop2}
\end{align}
Since $m_{N,h,\kappa}(\xi)$ and $m(\xi)$ are both Stieltjes transforms of a probability measure, 
we have $|m_{N,h,\kappa}(\xi)-m(\xi)|= \OO(|\xi|^{-2})$, thus 
$(m_{N,h,\kappa}(\xi)-m(\xi))((A-\xi)(B-\xi))^{1/2}=\OO(|\xi|^{-1})$ as $|\xi|\to\infty$. 
So
the left hand side of (\ref{eqn:loop1}) is $2\pi\ii((A-z)(B-z))^{1/2}(m_{N,h,\kappa}(z)-m(z))$, by the residue theorem. Moreover, we have the estimates (\ref{eqn:estPas1}) and
the trivial bounds
$$
\frac{1}{N}m_{N,h,\kappa}'(\xi)=\OO\left(\frac{1}{N d(\xi)^2}\right),\ \quad
\frac{1}{N}\int_{\mathbb{R}}\frac{h'(s)}{\xi-s}\varrho_1^{N,h,\kappa}(s)\rd s=
\OO\left(\frac{1}{N d(\xi)}\right).
$$
Together with (\ref{eqn:estPas1}), this implies that
the right hand side of (\ref{eqn:loop1}) is $\OO\left(\frac{(\log N)^2}{N\delta^2}\right)$.

Finally, to estimate (\ref{eqn:loop2}), we will use the induction hypothesis $\mathcal{P}(a)$. 
We first introduce the notations (for $t>0$)
$$
\Omega_{\delta,t}=\{\omega\in\Omega_\delta\mid \Im(\omega)>t\},\ \Omega_{\delta,-t}=\{\omega\in\Omega_\delta\mid \Im(\omega)<-t\}.
$$
By first using the continuity  of $r$ and $b_N$ at $\Im(\xi)=0$ and then Green's formula separately in 
$\Omega_{\delta,t}$, $\Omega_{\delta,-t}$, we obtain (all contour integrals being counterclockwise)
\begin{align}
\int_{\partial \Omega_\delta}\frac{b_N(\xi)}{z-\xi}r(\xi)\rd \xi&=
\lim_{t\to 0^+}\left(
\int_{\partial \Omega_{\delta,t}}\frac{b_N(\xi)}{z-\xi}r(\xi)\rd \xi+
\int_{\partial \Omega_{\delta,-t}}\frac{b_N(\xi)}{z-\xi}r(\xi)\rd \xi\notag
\right)
\\&=
\OO\left(\iint_{\Omega_{\delta}\setminus\mathbb{R}} \frac{1}{|z-\xi|}|\partial_{\bar\xi}(b_N(\xi)r(\xi))|\notag \rd\xi\rd\bar\xi\right)\\&=
\OO\left(\iint_{\Omega_{\delta}\setminus\mathbb{R}} \frac{1}{|z-\xi|}|\partial_{\bar\xi}(b_N(\xi))| \rd\xi\rd\bar\xi\right)+
\OO\left(\iint_{\Omega_{\delta}\setminus\mathbb{R}} \frac{1}{|z-\xi|}|b_N(\xi)\partial_{\bar\xi}r(\xi)| \rd\xi\rd\bar\xi\right),
\label{eqn:doubleError}
\end{align}
where we used (\ref{eqn:r1}).
A straightforward calculation from \eqref{eqn:quasiAnalytic} and \eqref{eqn:r1}  yields
\begin{equation}\label{eqn:bothErrorTerms}
\partial_{\bar\xi}b_N(\xi)=-\frac{\Im(\xi)^2}{2} V^{(4)}(E)(m_{N,  h, \kappa}(\xi)-m(\xi)),\qquad 
\partial_{\bar\xi}r(\xi)=\OO\left(\frac{\partial_{\bar \xi}(2m(\xi)-
\partial_E\widetilde V(\xi))}{|(A-\xi)(B-\xi)|^{1/2}}\right)=
\OO\left(
\frac{(\Im(\xi))^{ 2})}
{|(A-\xi)(B-\xi)|^{1/2}}
\right).
\end{equation}
Moreover, as $V$ is of class $\mathscr{C}^4$, 
the functions 
$s\mapsto \Re\left(\frac{\partial_E\widetilde V(\xi)-V'(s)}{\xi-s}\right), s\mapsto \Im\left(\frac{\partial_E\widetilde V(\xi)-V'(s)}{\xi -s}\right)$  have their first two derivatives on $I_\kappa$ uniformly bounded for $z$ in any compact set.
Consequently, we can use Lemma \ref{lem:concFar} 
with  $h_2$  playing the role of these functions,
 and we easily get, assuming $\mathcal{P}(a)$ (which
in particular guarantees the condition \eqref{eqn:estim1} in Lemma  \ref{lem:concFar}) that 
\begin{equation}\label{eqn:VVanishesa}
b_N(\xi)=\OO\left(N^{-1+2a+\e}\right),
\end{equation}
for any $\e>0$,
uniformly for $z$ in any compact set of $\mathbb{C}$.  Here we also used (\ref{eqn:largDev1})
to control the non-compact regime. From \eqref{Pa2} we also have
$$
 \partial_{\bar\xi}b_N(\xi) = \OO\left( \frac{ N^{-1+2a + \e } (\Im (\xi))^2}{|(A-\xi)(B-\xi)|^{1/2}}\right)  
$$
in the integration regime in    (\ref{eqn:doubleError}).  
By the estimates (\ref{eqn:bothErrorTerms}) and (\ref{eqn:VVanishesa}) we finally obtain that
both error terms in (\ref{eqn:doubleError}) are
$\OO(N^{-1+2a + \e  }\delta^{5/2})$.
We proved that the right hand side of (\ref{eqn:loop1}) and (\ref{eqn:loop2}) together have a
size bounded by $CN^\e 
 \left(\frac{1}{N\delta^2}+\frac{N^{2a}\delta^{5/2}}{N}\right)$.
 If $\eta\in(N^{-a},N^{-a/2})$ we choose $\delta=\eta/2$, which yields an error term at most $C N^\e  /(N\eta^2)$. 
If $\eta\in(N^{-a/2},1)$ we choose $\delta=N^{-a/2}/2$, which yields an error at most $C N^{-1+a +\e}$. This
shows that $\mathcal{P}(a/2)$ holds and it  concludes the proof.
\end{proof}

An immediate consequence of Lemmas \ref{lem:concFar} and \ref{lem:keyFar} is the following concentration of linear statistics.

\begin{corollary}\label{lem:concFar1}
Let $V$ be $\mathscr{C}^4$, 
regular such that the equilibrium density $\varrho_V$ is
supported on a single interval $[A,B]$  and satisfy
 (\ref{eqn:LSImu}), (\ref{eqn:GrowthCondition}). 
Let $h$ be a $\mathscr{C}^2$ function such that 
$\|h\|_\infty,\|h'\|_\infty,\|h''\|_\infty<\infty$. 
Then for any $\e>0$ there exists
a constant $c>0$ such that, for any $N\geq 1$, we have
\begin{equation}\label{eqn:h1}
\Prob^{\mu^{(N)}}\left(\left|\sum_{k=1}^N h(\lambda_k)-N\int h(s)\varrho(s)\rd s\right|>N^{\e}\right)\leq e^{-N^c}.
\end{equation}
\end{corollary}

As we mentioned in the proof of Lemma \ref{lem:keyFar}, for fixed $z$, as $V$ is $\mathscr{C}^4$, the functions 
$s\mapsto \Re\left(\frac{V'(E)-V'(s)}{z-s}\right), s\mapsto \Im\left(\frac{V'(E)-V'(s)}{z-s}\right)$  have their first two derivatives uniformly bounded for $z$ in any compact set.
Consequently, using Corollary \ref{lem:concFar1} and Lemma \ref{eqn:largDev1}, we have, for any $\e>0$,
\begin{equation}\label{eqn:VVanishes}
\int_{\mathbb{R}}\frac{V'(E)-V'(s)}{z-s}\left(\varrho^{(N)}_1(s)-\varrho(s)\right)\rd s=\OO\left(\frac{N^\e}{N}\right)
\end{equation}
uniformly for $z$ in any compact set of $\mathbb{C}$.
This estimate will be useful in the Subsection \ref{subsec:ImprAcc}.

\subsection{Proof of Proposition \ref{prop:ImprConc}}\label{subsec:Conc}

For the proofs of Propositions~\ref{prop:ImprConc} and \ref{prop:ImprAcc} we will
assume that $k\le N/2$, thus $k=\wh k$ and we remove the hat from the indices.

\subsubsection{Convexification}
This paragraph modifies the original measure $\mu^{(N)}$ into a log-concave one, without changing the rigidity properties.
This convexification first appeared in \cite{BouErdYau2012}. We state the main steps hereafter for the sake of completeness, and because
the explicit form of the convexified measure will be required in the next multiscale analysis, subsection \ref{subsec:multiscale}.

Let $\theta$ be a continuous nonnegative function with $\theta=0$ on $[-1,1]$ and $\theta''\geq 1$ for $|x|>1$.
We can take for example $\theta(x)=(x-1)^2 \mathds{1}_{x>1}+(x+1)^2 \mathds{1}_{x<-1}$ in the following.

\begin{definition} \label{def:locallyConstrained1}
For any fixed $ s, \ell > 0$,  independent of $N$,
define the Gibbs probability measure
$$
\rd\nu^{(s, \ell,N,c_1,\varepsilon)}=e^{-\beta N\cH_\nu}:=\frac{1}{Z^{(s,\ell) }}
e^{-\beta N \psi^{(s)}-\beta N\sum_{i,j}\psi_{i,j} -\beta  N (W+1) \sum_{\alpha= 1 } ^\ell  X_\alpha^2 }\rd\mu,
$$
with Hamiltonian
\be\label{Hnu}
\cH_\nu = \psi^{(s)}+\sum_{i,j}\psi_{i,j}+(W+1) \sum_{\alpha= 1 } ^\ell  X_\alpha^2 +
   \sum_{k=1}^N  \frac{1}{2}V(\lambda_k)-
\frac{1}{N} \sum_{1\leq i<j\leq N}\log (\lambda_j-\lambda_i),
\ee
where
\begin{itemize}
\item $W$ is the constant appearing in the lower bound (\ref{eqn:LSImu});
\item the function $g_\alpha$  is chosen such that
$\|g_\alpha\|_\infty+\|g_\alpha'\|_\infty+\|g_\alpha''\|_\infty<\infty$ and, for any $N$
and $k\in\llbracket 1,N\rrbracket$,
$$
g_\alpha'(\widetilde \gamma_k)=\sqrt{2}\cos\left(2\pi \left(k-\frac{1}{2}\right)\frac{\alpha}{2N}\right),
$$
where $\widetilde \gamma_k$ is defined by $\int_{-\infty}^{\widetilde \gamma_k}\varrho_V(s)\rd s=\frac{1}{N}(k-\frac{1}{2})$;
\item $X_\alpha = N^{-1/2} \sum_j  \left( g_\alpha(\lambda_j) -  g_\alpha(\widetilde\gamma_j)\right)$;
\item $\psi^{(s)}(\lambda)=N\theta\left(\frac{s}{N}\sum_{i=1}^N(\lambda_i-\widetilde \gamma_i)^2\right)$;
\item $\psi_{i,j}(\lambda)=\frac{1}{N}\theta\left(\sqrt{c_1\, N\, Q_{i,j}}(\lambda_i-\lambda_j)\right)$,
 where $c_1$ is a positive constant (to be chosen large enough but independent of $N$ 
in the next Lemma \ref{lem:convex}) and $Q_{ij}$ is defined in the following way. Let the function $m(n)$ be defined on $\ZZ$ by
$m(n)\in\llbracket -N+1,N\rrbracket$ and $m(n)\equiv n \; {\rm mod} (2N)$; let $d(k,\ell)=|m(k-l)|$ and $\varepsilon>0$
be a fixed small parameter; let
$$
R_{k,\ell}=\frac{1}{N}{\frac{\varepsilon^{2/3}}{\frac{d(k,\ell)^2}{N^2}+\varepsilon^2}}
$$
for any $k,\ell\in\llbracket -N+1,N\rrbracket$;  $Q=Q(\varepsilon)$ is then finally defined, for $i,j\in\llbracket 1,N\rrbracket$, by
$$
Q_{i,j}=R_{i,j}+R_{1-i,j}+R_{i,1-j}+R_{1-i,1-j}.
$$
\end{itemize}
Note that the measure $\nu^{(s,\ell,N,c_1,\varepsilon)}$ depends on all  five parameters but we will take the liberty to
omit some or all of them in formulas where they are irrelevant.
\end{definition}

Thanks to these linear statistics $X_\alpha$, the
convexity of $\nu$ is improved compared to the one of $\mu$,
in particular the following result was proved as Lemma 3.5 in \cite{BouErdYau2012}\footnote{
Note that in Lemma 3.5 in \cite{BouErdYau2012}, the constant $c$ was just required to be positive but following
 the reasoning  in \cite{BouErdYau2012} it can be made arbitrary large
by choosing $\ell$ sufficiently large.
}.

\begin{lemma}\label{lem:convex}
For any $C>0$ there are constants $\ell,s,c_1,\varepsilon>0$ depending only on $V$ and $C$,  such that for $N$ large enough
$\nu=\nu^{(s,\ell,N,c_1,\varepsilon)}$ satisfies, for any $\bv\in\RR^N$,
$$
\langle \bv,(\nabla^2 \cH_\nu) \bv\rangle\geq C\, \|\bv\|^2.
$$
\end{lemma}

An important fact for the measure $\nu^{(s,\ell)}$ is that it does not deviate much from $\mu$
concerning events with very small probability. More precisely,
the following result holds.

We say that a sequence of events $(A_N)_{N\geq 1}$ is {\it
exponentially small} for a sequence of probability measures
$(m_N)_{N\geq 1}$ if there are constants $ C, c>0$ such that, for any $N$, we have
$$
m_N(A_N)\leq C e^{- N^{c}}.
$$

\begin{lemma}\label{lem:equivalence}
For any fixed choice of the parameters  $s,\ell,c_1,\varepsilon$ defining $\nu^{(N)}$, the measures
$(\mu^{(N)})_{N\geq 1}$ and  $(\nu^{(N)})_{N\geq 1}$
have the same exponentially small events.
In particular,
for any $a>0$,
concentration
 at scale $a$ for $(\mu^{(N)})_{N\geq 1}$ is equivalent to concentration at scale $a$ for  $(\nu^{(N)})_{N\geq 1}$.
\end{lemma}

\begin{proof}
The first statement can be proved as Lemma 3.6 in \cite{BouErdYau2012}, except that in that paper we used
$\E(N X_\alpha^2)<(\log N)^2$, an estimate true in the context of an analytic potential $V$. Here we 
only assume that $V$ is $\mathscr{C}^4$; then by Lemma \ref{lem:concFar}, for any $\e>0$
and for large enough $N$,
we have $\E(N X_\alpha^2)\leq  N^\e$. As $\e$ is arbitrarily small the remainder of the proof goes in the same way 
as Lemma 3.6 in \cite{BouErdYau2012}.

Note that the second statement of the lemma is not a completely direct application of the first one: if
rigidity at scale $a$ holds for $(\mu^{(N)})_{N\geq 1}$, by using the first statement
we obtain that for any $\varepsilon>0$ there are $c>0$ and $N_0$
such that for all $N\ge N_0$ and $k\in\llbracket 1,N\rrbracket$
\begin{equation}\label{eqn:conseq}
\P^{\nu^{(N)}}\left(|\lambda_k-\E^{\mu^{(N)}}\lambda_k|> N^{-\frac{2}{3}+a+\e}(\wh k)^{-\frac{1}{3}}\right)
\leq  e^{- N^c}.
\end{equation}
However, to obtain concentration for $\nu^{(N)}$, we need to estimate the
difference $\E^{\mu^{(N)}} \lambda_k -\E^{\nu^{(N)}} \lambda_k$.

We know from (\ref{eqn:BPS2}) that for any $\kappa>0$, there is a $C>0$ such that
for $x\not\in[A-\kappa,B+\kappa]$, we have
\begin{equation}\label{eqn:Boutet}
\varrho_1^{(N,\mu)}(x)\leq (|x|+1)^{-C N}.
\end{equation}
Similarly to Lemma 3.6 in \cite{BouErdYau2012}, for any $\e>0$ there is a $c>0$ such that for any event $A$,
\begin{equation}\label{eqn:A}
\Prob_\nu(A)\leq e^{c N^\e}\Prob_\mu(A).
\end{equation}
Equations (\ref{eqn:Boutet}) and (\ref{eqn:A}) imply that for some positive constants $c$ and $c'$,
\begin{equation}\label{eqn:corNu}
\varrho_1^{(N,\nu)}(x)\leq (|x|+1)^{-c N}e^{c' N^\e}.
\end{equation}
Equation (\ref{eqn:conseq}) together with the large-deviation type estimate (\ref{eqn:corNu}) imply that
$$
|\E^{\mu^{(N)}}(\lambda_k)-\E^{\nu^{(N)}}(\lambda_k)|=\OO(N^{-\frac{2}{3}+a+\e}(\wh k)^{-\frac{1}{3}}),
$$
and subsequently that concentration holds for $\nu^{(N)}$ at scale $a$.
That concentration for $\nu$ implies concentration
for $\mu$ can be proved in a similar way (it is easier because (\ref{eqn:A}) is not needed,
the necessary decay follows directly from \eqref{eqn:Boutet}).
\end{proof}

\subsubsection{The multiscale analysis}\label{subsec:multiscale}

This subsection is similar to subsection 3.2 in \cite{BouErdYau2011}, but we  adapted the arguments
in the scalings
to improve the rigidity scale up to the edges.

In this subsection, $s,\ell,c_1,\varepsilon$ are chosen so that $\nu^{(N)}$ satisfies the convexity relation from
Lemma \ref{lem:convex} with $C=10W$.
We now define the locally constrained measures, up to the edge; these measures ensure strict convexity bounds
when knowing rigidity at scale $a$.

\begin{definition} \label{def:locallyConstrained2}
Let $\varepsilon>0$.
For any given $k\in\llbracket1,N\rrbracket$ and any integer $1\leq  M\leq  N/2$, we denote
$$
I^{(k,M)}=
\left\{
\begin{array}{ccc}
\llbracket k,k+M-1\rrbracket&\mbox{if}& k\leq N/2\\
\llbracket k-M+1,k\rrbracket&\mbox{if}& k\geq N/2
\end{array}
\right..
$$
Moreover, let
$$
\phi^{(k,M)}=
\sum_{i<j,i,j\in I^{(k,M)}}\theta\left(\frac{N^{\frac{2}{3}-\epsilon}(\hat k)^{\frac{1}{3}}}
{M}(\lambda_i-\lambda_j)\right).
$$
We define the probability measure
\begin{equation}\label{eqn:omega}
\rd\omega^{(k,M)}:=\frac{1}{Z}e^{-\beta\phi^{(k,M)}}\rd\nu,
\end{equation}
where $Z=Z_{\omega^{(k,M)}}$.
The measure
$\omega^{(k,M)}$ will be referred to as locally constrained
transform of $\nu$, around $k$, with width $M$.
The dependence of
the measure on $\epsilon$ will be suppressed in the notation.
\end{definition}

We will also frequently use the following notation
for block averages in any sequence $(x_i)_i$:
$$
x_k^{[M]}:=\frac{1}{M}\sum_{i\in I^{(k,M)}}x_i.
$$

The reason for introducing these locally constrained measures is that they improve the convexity  in $I^{(k,M)}$
on the subspace orthogonal to the constants,
 as explained in the following lemma which is a slight modification of
Lemma 3.8 of \cite{BouErdYau2011}.

\begin{lemma}\label{lem:localConvexity}
Write  the
probability measure
$\omega^{(k,M)}$  from \eqref{eqn:omega} as
$\omega^{(k,M)}=\frac{1}{\widetilde Z}
e^{-\beta N(\mathcal{H}_1+\mathcal{H}_2)}\rd\lambda,$
where we denote
\begin{align*}
\mathcal{H}_1&:=\frac{1}{N}\phi^{(k,M)}-
\frac{1}{{ 2}N}\sum_{i< j, i,j\in I^{(k,M)}}\log|\lambda_i-\lambda_j|
,\\
\mathcal{H}_2&:=\mathcal{H}_\nu+\frac{1}{{ 2}N}
\sum_{i< j, i,j\in I^{(k,M)}}\log|\lambda_i-\lambda_j|.
\end{align*}
Then $\nabla^2 \mathcal{H}_2\geq 0$
and  denoting $\bv  =(v_i)_{i\in I^{(k,M)}}$, we also have
$$
\langle\bv,(\nabla^2 \mathcal{H}_1)\bv\rangle\geq
\frac{1}{2N}\left(\frac{N^{\frac{2}{3}-\varepsilon}(\hat k)^{\frac{1}{3}}}{M}\right)^2
\sum_{i,j\in I^{(k,M)},i<j} (v_i-v_j)^2.
$$
\end{lemma}

\begin{proof}
Note that in the modification $\mathcal{H}_2$ of $\mathcal{H}_\nu$, we only removed half of the
pairwise interactions\footnote{This minor point was not made explicit in \cite{BouErdYau2012}.}
 between the $\lambda$'s in $I^{(k, M)}$. This allows us to use Lemma \ref{lem:convex}
  (with the choice $c=10W$)
to prove the convexity of $\mathcal{H}_2$.
Denoting $\mathcal{V}=\mathcal{V}(\bla): = \frac{1}{2}\sum_j V(\lambda_j)$,
  we indeed have
$$\nabla^2\mathcal{H}_2=\nabla^2(\mathcal{H}_2-\frac{1}{2}\mathcal{V})+\frac{1}{2}\nabla^2\mathcal{V}
\geq \frac{1}{2}\nabla^2\mathcal{H}_\nu+\frac{1}{2}\nabla^2\mathcal{V}
\geq \frac{1}{2}10 W-\frac{1}{2}W\geq 0.
$$
In the first inequality we used that
$$
 \cH_2 - \frac{1}{2}\mathcal{V} -    \frac{1}{2}\mathcal{H}_\nu = \frac{1}{2}
 \Big( \psi^{(s)}+\sum_{i,j}\psi_{i,j}+(W+1) \sum_{\alpha= 1 } ^\ell  X_\alpha^2\Big)
$$
from \eqref{Hnu} and each term on the right hand side is convex by their explicit definitions.

Concerning the lower bound for $\nabla^2\mathcal{H}_1$, a simple calculation gives
$$
\langle \bv,(\nabla^2 \mathcal{H}_1)\bv\rangle\geq
\frac{1}{2N}\sum_{i< j, i,j\in I^{(k,M)}}(v_i-v_j)^2\left(\frac{1}{(\lambda_i-\lambda_j)^2}
+
\left(\frac{N^{\frac{2}{3}-\epsilon}(\hat k)^{\frac{1}{3}}}{M}\right)^2
  \mathds{1}\Big\{|\lambda_i-\lambda_j|>
\frac{M}{N^{\frac{2}{3}-\epsilon}(\hat k)^{\frac{1}{3}}}\Big\}
\right),$$
which concludes the proof.
\end{proof}

The above convexity bound on $\mathcal{H}_1$ allows us
 to get an improved concentration for functions depending on differences
between particles, as shown in the following lemma:

\begin{lemma}[Lemma 3.9 in \cite{BouErdYau2011}]
\label{lem:localLSI} Decompose  the coordinates $\lambda=(\lambda_1, \dots ,\lambda_N)$
of a point in  $\RR^N = \RR^m \times \RR^{N-m}$ as $\bla=(x,y)$, where $x\in \RR^m$, $y\in \RR^{N-m}$.
Let $\omega=\frac{1}{Z}e^{-N \mathcal{H}}$ be a probability measure on $\RR^N = \RR^m \times \RR^{N-m}$ such that
$\mathcal{H}=\mathcal{H}_1+\mathcal{H}_2$, with $\mathcal{H}_1=\mathcal{H}_1(x)$ depending only on
the $x$ variables and $\mathcal{H}_2 =\mathcal{H}_2(x,y)$
depending on all coordinates. Assume that, for any $\lambda\in\RR^N$,
$\nabla^2\mathcal{H}_2(\lambda)\geq 0$.
Assume moreover that $\mathcal{H}_1(x)$ is independent of
 $x_1 + \dots + x_m$, i.e., $\sum_{i=1}^m \partial_i \mathcal{H}_1(x) =0$
 and  that  for any $x, v\in\RR^m$,
$$
\langle\bv, (\nabla^2\mathcal{H}_1(x))\bv\rangle\geq  \frac{\xi}{m}  \sum_{i, j =1}^m  |v_i-v_j|^2 \;
$$
with some positive $\xi>0$.
Then for any  function  of the form $f(\lambda)=F( \sum_{i=1}^mv_i x_i)$, where
$\sum_iv_i=0$  and $F:\RR\to\RR$ is any smooth function,
 we have
$$
\int f^2\log f^2\rd\omega-\left(\int f^2\rd\omega\right)\log\left(\int f^2\rd\omega\right)
 \leq \frac{1}{\xi  N}\int|\nabla f|^2\rd\omega.
$$
\end{lemma}

A direct application of Lemmas \ref{lem:localConvexity} and \ref{lem:localLSI}
gives, by Herbst's lemma,
the following concentration estimate.

\begin{corollary}\label{cor:ConcentrationDifferences}
For any function $f(\{\lambda_i,i\in I^{(k,M)}\})=\sum_{I^{(k,M)}}v_i\lambda_i$ with $\sum_i v_i=0$
and for any $u>0$ we have,
for some constant $c>0$ that depends only on $\beta$ and $V$,
$$
\P^{\omega^{(k,M)}}(|f-\E^{\omega^{(k,M)}}(f)|>u)
\leq 2\exp\left(-\frac{c}{M |v|^2}\left( N^{\frac{2}{3}-\epsilon}(\hat k)^{\frac{1}{3}}\right)^2u^2\right).
$$
\end{corollary}
When the function $f$ is chosen of type $\lambda_k^{[M_1]}-\lambda_k^{[M]}$, we get in particular
the following concentration.

\begin{lemma}\label{lem:concGapsOmega} Take any $\epsilon>0$. There are constants $c>0$, $N_0$ such that for any
$N\geq N_0$, any integers $1\leq M_1\leq M\leq  N/2$, any $k\in\llbracket1,N\rrbracket$,
and $\omega^{(k,M)}$ from Definition \ref{def:locallyConstrained2} associated with $k,M,\epsilon$,
we have for any $u>0$,
$$
\P^{\omega^{(k,M)}}\left(\left|\lambda_k^{[M_1]}-\lambda_k^{[M]}
-\E^{\omega^{(k,M)}}\left(\lambda_k^{[M_1]}- \lambda_k^{[M]}\right)\right|>u
(N^{\frac{2}{3}-\varepsilon}(\hat k)^{\frac{1}{3}})^{-1}
\sqrt{\frac{M}{M_1}}\right)\leq e^{-c u^2}.
$$
\end{lemma}

\begin{proof}
Relying on Corollary ~\ref{cor:ConcentrationDifferences}, writing $\lambda_k^{[M_1]}
-
\lambda_k^{[M]}=\sum v_i\lambda_i$ with some constants $v_i$, one only needs to prove
$|v|^2\leq 1/M_1$ to conclude. An explicit computation gives $|v|^2=1/M_1-1/M$.
\end{proof}

The following three Lemmas are slight modifications of Lemmas 3.15, 3.16
and 3.17 from \cite{BouErdYau2011}.

\begin{lemma}\label{lem:diffExpectations}
Assume that for $\mu$ rigidity at scale $a$ holds.
Take arbitrary $\epsilon>0$.
There exist constants $c, N_0>0$ such that for any
$N\geq N_0$,  any integer $M$ satisfying $N^a\leq M\leq  N/2$, any $k,j\in\llbracket 1,N\rrbracket$
we have
$$
|\E^\nu(\lambda_j)-\E^{\omega^{(k,M)}}(\lambda_j)|\leq  e^{-N^c},
$$
where the measure $\omega^{(k,M)}$ is given by Definition \ref{def:locallyConstrained2}  with parameters $k,M,\epsilon$.
\end{lemma}

\begin{proof}
 Note  that $\theta(x)=0$ if $|x|<1$, so if the $\phi^{(k,M)}$ term in the definition of $\om$
is non-zero then either
$|\lambda_k-\gamma_k|$ or $|\lambda_{k+M}-\gamma_{k+M}|$ is greater than
 $\frac{1}{3}M N^{-\frac{2}{3}+\varepsilon}k^{-1/3}$,
where we used that $\gamma_{k+M}-\gamma_k \le \frac{1}{3}
 M N^{-\frac{2}{3}+\e} k^{-1/3}$.
Since rigidity at scale $a$ holds for $\mu$, it also holds for $\nu$ by Lemma \ref{lem:equivalence}, so both events
have exponentially small probability (remember that $M\geq N^a$).
This easily implies  that $\int e^{-\beta\phi^{(k,M)}} \rd \nu>1/2$ for large enough $N$, and therefore
$\P^{\omega^{(k,M)}}(A)\leq 2 \P^{\nu}(A)$ for any event $A$.
Consequently (\ref{eqn:corNu}) holds when replacing $\nu$ by $\omega$:
\begin{equation}\label{eqn:corOmega}
\varrho_1^{(N,\omega^{(k,M)})}(x)\leq 2 (|x|+1)^{-c N}e^{c'(\log N)^2}
\end{equation}
for some constants $c,c'$.
The total variation norm is bounded by the square root of the entropy (defined
for a probability measure $\nu$ and a probability density $f$ (w.r.t. $\nu$), by
$S_\nu(f)=\int f\log f\rd\nu$); moreover, by (\ref{eqn:corNu}) and (\ref{eqn:corOmega})
the particles are bounded with very high probability,
both for the measure $\nu$ and $\omega^{(k,M)}$.
We therefore have
$$
|\E^\nu(\lambda_j)-\E^{\omega^{(k,M)}}(\lambda_j)|\leq C  \sqrt { S_{\omega^{(k,M)}}(\rd \nu/\rd\omega^{(k,M)})} +\OO(e^{- c N})
$$
for some  $c,C>0$ independent of $N,k, j$.
In order to bound this entropy, note that the measure $\nu$ satisfies a logarithmic Sobolev inequality
with constant of order $N$ (this follows from the convexity estimate obtained in Lemma \ref{lem:convex} and an
application of the Bakry-\'Emery criterion \cite{BakEme1983}): for any smooth $f\geq 0$ with $\int f\rd\nu=0$, we have
\begin{equation}\label{eqn:LSInu}
\int f\log f\rd\nu\leq \frac{1}{c N}\int |\nabla \sqrt{f}|^2\rd\nu,
\end{equation}
for some small fixed $c>0$. We therefore obtain,
for some large fixed $C>0$,
\begin{equation}\label{eqn:expBound}
S_{\omega^{(k,M)}}(\rd \nu/\rd\omega^{(k,M)})
\leq
N^C\,\E^\nu\left(\theta'\left(\frac{(\lambda_{k+M}-\lambda_{k})N^{\frac{2}{3}-\epsilon}(\hat k)^{\frac{1}{3}}}{M}\right)^2\right).
\end{equation}
We claim that the above expectation can be bounded by $e^{-N^c}$
for some fixed $c>0$ if $N$ is large. To prove this exponential bound,
we  assume $k<N/2$ for simplicity.
As we saw at the beginning of this proof, if the above $\theta'$ term is non-zero then either
$|\lambda_k-\gamma_k|$ or $|\lambda_{k+M}-\gamma_{k+M}|$ is greater than
 $\frac{1}{3}M N^{-\frac{2}{3}+\varepsilon}k^{-1/3}$,
and both events
have exponentially small probability.
Together with  $\theta'(x)^2<4 x^2$ and  (\ref{eqn:corNu}), this proves the desired estimate (\ref{eqn:expBound}).
\end{proof}

\begin{lemma}\label{lem:toMu}
Assume that for $\mu$ rigidity  at scale $a$ holds. Take arbitrary $\epsilon>0$.
There are constants $c>0$  and $N_0$ such that for any
$N\geq N_0$, any integers $N^a\leq M\leq N/2$, $1\leq M_1\leq M$, and $k\in\llbracket 1,N\rrbracket$, we have
$$
\P^{\nu}\left(\left|\lambda_k^{[M_1]}-\lambda_k^{[M]}
-\E^{\nu}\left(\lambda_k^{[M_1]}-
\lambda_k^{[M]}\right)\right|>(N^{\frac{2}{3}-\varepsilon}(\hat k)^{\frac{1}{3}})^{-1}\sqrt{\frac{M}{M_1}}\right)\leq e^{-N^c}.
$$
\end{lemma}

\begin{proof}
By Lemma \ref{lem:concGapsOmega}
we know that the result holds when considering $\omega^{(k,M)}$ instead of
$\nu$. Moreover,  by Lemma \ref{lem:diffExpectations} the difference
$$
\left|\E^{\nu}\left(\lambda_k^{[M_1]}-
\lambda_k^{[M]}\right)
-
\E^{\omega^{(k,M)}}\left(\lambda_k^{[M_1]}-
\lambda_k^{[M]}\right)\right|
$$
is exponentially small. So we just need to prove that
$$
(\P^{\nu}-\P^{\omega^{(k,M)}})\left(\left|\lambda_k^{[M_1]}-\lambda_k^{[M]}
-\E^{\nu}\left(\lambda_k^{[M_1]}-
\lambda_k^{[M]}\right)\right|>(N^{\frac{2}{3}-\varepsilon}(\hat k)^{\frac{1}{3}})^{-1}\sqrt{\frac{M}{M_1}}\right)
$$
is bounded by $e^{-N^c}$. This is true because
$|\P^\nu(A)-\P^{\omega^{(k,M)}}(A)|$ is bounded by $( S_{\omega^{(k,M)}}(\rd \nu/\rd\omega^{(k,M)}))^{1/2}$,
which is exponentially small, as proved below (\ref{eqn:expBound}).
\end{proof}

\begin{lemma}\label{lem:concA2}
Assume that for $\mu$ rigidity at scale $a$ holds.
For any $\epsilon>0$, there are constants $c, N_0>0$ such that for any $N\geq N_0$
and $k\in\llbracket 1,N\rrbracket$, we have
$$
\P^\nu\left(\left|\lambda_k-\lambda_k^{[N/2]}
-\E^\nu(\lambda_k-\lambda_k^{[N/2]}) \right|
> N^{-\frac{2}{3}+\frac{a}{2}+\epsilon}(\hat k)^{-\frac{1}{3}} \right)
\leq
e^{-N^c}.
$$
\end{lemma}

\begin{proof}
Note first that
$$
\left|\lambda_k-\lambda_k^{[ N/2]}-\E^\nu(\lambda_k-\lambda_k^{[N/2]})
\right|\leq
\left|\lambda_k-\lambda_k^{[N^a]}-\E^\nu(\lambda_k-\lambda_k^{[N^a]})
\right|+
\left|\lambda_k^{[N^a]}-\lambda_k^{[N/2]}-\E^\nu(\lambda_k^{[N^a]}
-\lambda_k^{[N/2]})
\right|.
$$
By the choice $M_1=1$, $M=N^a$ in Lemma \ref{lem:toMu}, the $\nu$-probability that the
first term is greater than $N^{-\frac{2}{3}+\frac{a}{2}+\epsilon}(\hat k)^{-\frac{1}{3}}$ is exponentially
small, uniformly in $k$, as desired.
Concerning the second term, given some $r>0$ and $q\in\NN$ defined by $1-r\leq a+q r<1$, it is bounded by
$$
\sum_{\ell=0}^{q-1}
\left|\lambda_k^{[N^{a+(\ell+1)r}]}-\lambda_k^{[N^{a+\ell r}]}
-\E^\nu\left(\lambda_k^{[N^{a+(\ell+1)r}]}-\lambda_k^{[N^{a+\ell r}]}\right)
\right|
+
\left|\lambda_k^{[a+qr]}-\lambda_k^{[N/2]}
-\E^\nu\left(\lambda_k^{[N^{a+qr}]}-\lambda_k^{[N/2]}\right)
\right|.
$$
By Lemma \ref{lem:toMu}, for any $\epsilon>0$, each one of these $q+1$ terms has an exponentially small probability of being greater than
$
N^{-\frac{2}{3}+\epsilon+\frac{r}{2}}(\hat k)^{-\frac{1}{3}}
$. Consequently, choosing any $\epsilon$ and $r$ (and therefore $q$) such that $\epsilon+\frac{r}{2}<a/2$ concludes the proof.
\end{proof}

\begin{proof}[Proof of Proposition \ref{prop:ImprConc}.]
Obviously,
$$
|\lambda_k-\E^\nu(\lambda_k)|\leq
|\lambda_k-\lambda_k^{[N/2]}-\E^\nu(\lambda_k-\lambda_k^{[N/2]})|+
|\lambda_k^{[N/2]}-\E^\nu(\lambda_k^{[N/2]})|.
$$
By Lemma \ref{lem:concA2}, the first term has exponentially small probability to be greater than
$N^{-\frac{2}{3}+\frac{a}{2}+\epsilon}(\hat k)^{-\frac{1}{3}}$.
Moreover, as $\nu$ satisfies (\ref{eqn:LSInu}), by the classical Herbst's lemma (see e.g. \cite{AndGuiZei2010}), the second term has exponentially small probability to be greater than $N^{-1+\varepsilon}$.
This concludes the proof of  concentration at scale $a/2$ for the measure $\nu$.

Consequently, by Lemma \ref{lem:equivalence},
for any $\e>0$, there are constants
$c, N_0>0$ such that for any
$N\geq N_0$ and $k\in\llbracket 1, N\rrbracket$  we have
$$
\P^\mu\left(|\lambda_k-\E^\nu\lambda_k|> N^{-\frac{2}{3}+\frac{a}{2}+\e}(\hat k)^{-\frac{1}{3}}\right)\leq
 \ e^{- N^c}.
$$
This probability bound together with (\ref{eqn:Boutet}) implies that
$$
|\E^\nu\lambda_k-\E^\mu\lambda_k|=\OO(N^{-\frac{2}{3}+\frac{a}{2}+\e}(\hat k)^{-\frac{1}{3}})
$$
uniformly in $N$ and $k$, and concludes the proof of concentration at scale $a/2$ for $\mu$.
\end{proof}

\subsection{Proof of Proposition \ref{prop:ImprAcc}}\label{subsec:ImprAcc}

We aim at improving the accuracy from scale $a$ to scale $11a/12$,
now that we know concentration at scale $a/2$
from the proven Proposition~\ref{prop:ImprConc}.
In \cite{BouErdYau2011} and \cite{BouErdYau2012} 
we  proved that, in the bulk of the spectrum,
$$
m_N(z)-m(z) \sim \var_N(z),\qquad \var_N(z):=\frac 1 { N^2} \var_{\mu^{(N)}}\left(\sum_{k=1}^N\frac{1}{z-\lambda_k}\right).
$$
 Concentration at scale $a/2$ then allowed us to properly bound the above variance term, which yielded
good estimates on $m_N-m$ and therefore an improved accuracy. In these previous works analyticity
of $V$ was essential, as it was in \cite{Joh1998} and \cite{Shc2011}.

We first explain the method for the accuracy improvement, for non-analytic $V$. 
The following modification of the loop equation will be useful:
from the difference of (\ref{eqn:loopNfixed})  (with $h=0$)  and (\ref{eqn:loopAsympt}) we obtain
(noting $z=E+\ii\eta$)
 \begin{multline}\label{eqn:loopModified}
(m_{N}(z)-m(z))^2+(2m(z)-V'(E))(m_{N}(z)-m(z))
+\int_{\mathbb{R}}\frac{V'(E)-V'(s)}{z-s}(\varrho_1^{(N)}(s)-\varrho(s))\rd s\\
-\frac{1}{N}\left(\frac{2}{\beta}-1\right)m_{N}'(z)-
\var_N(z)
=0.
\end{multline}
 In the above equation, the  integral term can be neglected thanks to  (\eqref{eqn:VVanishes}).
 The $(m_N-m)^2$ and $N^{-1}m_N'$ terms 
are easily shown to be of negligible 
order too, so for $z$ close to $[A,B]$ we have
$$
(2m(z)-V'(E))(m_{N}(z)-m(z))\sim \var_N(z).
$$
For $z$ close to the bulk of the spectrum, $2m(z)-V'(E)$ is bounded away from $0$, so this 
equation yields an accurate upper bound on $m_N-m$.

The rest of the proof of accuracy 
improvement involves a major technical difficulty:  optimal estimates up to the edge are difficult 
to obtain, because $2m(z)-V'(E)$ vanishes when $z$ is close to $A$ or $B$.
As a main difference from the accuracy improvement in \cite{BouErdYau2011}, 
our current use of the
loop equation will allow finer estimates, improving accuracy of {\it one} given particle
(the first one), in Lemma \ref{lem:smallest}.  The accuracy improvement for both extreme particles
together with the amelioration for $\lambda_k$'s with $\hat k\geq N^{3a/4}$ will imply improvement for all particles. 
The following series of lemmas makes these heuristics rigorous.

For any  $A<E<B$ we define
$$\kappa_E=\min(|E-A|,|E-B|)$$
the distance of $E$ from the edges of the support of the equilibrium measure.
Also, in this section, $a(N)\ll b(N)$ means $a(N)=\oo(b(N))$ as $N\to\infty$.
We will finally use the notations
\begin{align*}
\Sigma^{(N)}_{\rm Int}(u,\tau)&:=\left\{z=E+\ii\eta: A\leq E\leq B,\ N^{-1+u}\kappa_E^{-1/2}\leq \eta\leq \tau\right\},\\
\Sigma^{(N)}_{\rm Ext}(u,\tau)&:=\left\{z=E+\ii\eta: E\in[A-2N^{-2/3+u},A-N^{-2/3+u}],\ N^{-2/3+u}\leq \eta\leq \tau\right\}.
\end{align*}

\begin{lemma}\label{lem:Johansson}
Assume that
\begin{equation}\label{eqn:kNTo0}
\var_N(z)\ll \max(\kappa_E,\eta)
\end{equation}
as $N\to\infty$, uniformly in
$
\Sigma^{(N)}_{\rm Int}(u,\tau),
$
for some fixed $u>0$ and small $\tau>0$.
Then for any $\e>0$ there are constants $C$, $0<\delta<\tau$
such that for any
$z\in
\Sigma^{(N)}_{\rm Int}(u,\delta)
$ we have
$$
|m_N(z)-m(z)|\leq C \left(\frac{ N^\e}{N\eta}+
\frac{|\var_N(z)|}{\max(\sqrt{\kappa_E},\sqrt{\eta})}\right).
$$
The same statement holds when replacing $\Sigma^{(N)}_{\rm Int}$ everywhere by $\Sigma^{(N)}_{\rm Ext}$.
\end{lemma}

\begin{proof} 
We first note that
\begin{multline*}
\frac{1}{N}  | m_N'(z)| = \frac{1}{N^2 }  \left |
\E^{\mu}    \sum_j  \frac {1}  {(z-\lambda_j)^2} \right |\\
\le  \frac{1}{N \eta  } \Im    m_N(z)\leq \frac{1}{N \eta  } |m_N(z)-m(z)|
+\frac{1}{N\eta}|\Im m(z)|\leq
\frac{1}{N \eta  } |m_N(z)-m(z)|+\frac{C}{N\eta}\max(\sqrt{\kappa_E},\sqrt{\eta}),
\end{multline*}
where we used
$
\Im m(z)\le C\max \{\sqrt{\kappa_E},\sqrt{\eta}\},
$
an easy estimate due to the square root singularity 
of the equilibrium measure $\varrho$ on the edges.
Equation (\ref{eqn:loopModified}) therefore implies
\begin{align}\label{eqn:quadratic}
&(m_N(z)-m(z))^2+b(z)(m_M(z)-m(z))+c(z)=0,\\
&b(z):=2m(z)-V'(E)+\frac{ c_1(z,N)}{N\eta},\notag\\
&c(z):=\frac{ c_2(z,N)}{N\eta}\max(\sqrt{\kappa_E},\sqrt{\eta})+
 c_3(z,N) N^{\e-1}-\var_N(z),\notag
\end{align}
 where there is a constant $C>0$ such that for any for any $N$ and $z$, $|c_1(z,N)|, |c_2(z,N)|,|c_3(N,z)|<C$
(we used (\ref{eqn:VVanishes}) to bound the integral term in (\ref{eqn:loopModified})).

To solve the above 
quadratic equation (\ref{eqn:quadratic}), we need a priori estimates on the coefficients. As 
$\varrho$ has a square root singularity close to the edges, there is a constant $c>0$ such that
\begin{equation}\label{eqn:est2ndCoef}
c \max(\sqrt{\kappa_E},\sqrt{\eta})<|2m(z)-V'(E)|<c^{-1} \max(\sqrt{\kappa_E},\sqrt{\eta}).
\end{equation}
On the other hand, unifomly in $\Sigma_{{\rm Int}}^{(N)}(u,\tau)$ we have 
\begin{equation}\label{eqn:calcul}
\frac{1}{N\eta}\ll \max(\sqrt{\kappa_E},\sqrt{\eta}),
\end{equation}
so we obtain
\begin{equation}\label{eqn:coef1}
|b(z)|\gg \max(\sqrt{\kappa_E},\sqrt{\eta}).
\end{equation}
Moreover, from (\ref{eqn:kNTo0}) and (\ref{eqn:calcul}), the estimate
\begin{equation}\label{eqn:coef2}
c(z)\ll \max(\kappa_E,\eta)
\end{equation}
holds.
{F}rom the estimates
(\ref{eqn:coef1}) and (\ref{eqn:coef2}) we have  $b(z)^2\gg c(z)$, so 
the quadratic equation (\ref{eqn:quadratic}) yields
$$
m_N(z)-m(z)=\frac{-b(z)\pm\sqrt{b(z)^2-4c(z)}}{2}\underset{N\to\infty}{\sim}\frac{1}{2}\left(-b(z)\pm
b(z)\left(1-\frac{4c(z)}{2b(z)^2}\right)\right).
$$
 For $E$ in the bulk and $\eta\sim 1$ we know that $m_N(z)-m(z)\to 0$ and $b(z)\sim 1$,
so the appropriate asymptotics needs to be $m_N(z)-m(z)\sim -c(z)/b(z)$. By continuity, this holds in 
$\Sigma_{{\rm Int}}^{(N)}(u,\tau)$, concluding the proof.
In the case of the domain $\Sigma_{\rm Ext}^{(N)}(u,\delta)$, the proof is the same.
\end{proof}

The following lemma is similar to the previous one, but aims at controlling the extreme eigenvalues. For this, we introduce the notation
\begin{equation}\label{eqn:defOmega}
\Omega^{(N)}(d,s,\tau)=\left\{z=E+\ii \eta\mid \eta=N^{-\frac{2}{3}+s}, A-\tau\leq E\leq A-N^{-\frac{2}{3}+d}\right\}.
\end{equation}

\begin{lemma}\label{lem:Joh2}
Assume that for some $0<d,s\leq 2/3$, $\tau>0$,
\begin{equation}\label{eqn:aprioriBound}
\var_N(z)+\frac{1}{N}|m_N'(z)|\ll |z-A|
\end{equation}
uniformly on $\Omega^{(N)}(d,s,\tau)$.
Then for any $\e>0$ we have, uniformly on $\Omega^{(N)}(d,s,\tau)$,
we have
$$
|m_N(z)-m(z)|=\OO\left( |z-A|^{-1/2} 
\left(\var_N(z)+\frac{1}{N}|m_N'(z)| + N^{-1+\e} \right)\right).
$$
\end{lemma}

\begin{proof}
This lemma can be proved in a way perfectly analogous to Lemma \ref{lem:Johansson}:
we
solve the quadratic equation (\ref{eqn:loopModified}), after bounding 
its integral term by $N^{-1+\e}$. Two solutions are possible, 
which have asymptotics
(using (\ref{eqn:est2ndCoef}) and (\ref{eqn:aprioriBound}))
$$
m_N(z)-m(z)\sim\frac{\var_N(z)+\frac{c_1(z,N)}{N}|m_N'(z)|+c_2(z,N)N^{-1+\e}}{2m(z)-V'(E)}\quad \mbox{or}\quad
 m_N(z)-m(z)\sim -2m(z)+V'(E),
$$
where $|c_1(z,N)|, |c_2(z,N)|\leq C$ for some $C>0$ independent of $z$ and $N$. 
For $z=A-\tau+\ii N^{-\frac{2}{3}+s}$,
 we know that $m_N-m\to 0$ (this relies on the macroscopic 
convergence of the spectral measure and the large deviation 
estimate (\ref{eqn:largDev1})). This together with the continuity 
of $m_N-m$ and  (\ref{eqn:aprioriBound}), (\ref{eqn:est2ndCoef}), implies that the proper choice is
the first one uniformly in $\Omega^{(N)}(d,s,\tau)$.
\end{proof}

 The proofs of the following three technical lemmas are postponed to Appendix \ref{App:CalcLem}.

\begin{lemma}\label{lem:varianceBound}
Assume that rigidity at scale $a$ and concentration at scale $a/2$ hold. Then for any fixed $\tau>0,\varepsilon>0$,
uniformly on
$
\Sigma^{(N)}_{\rm Int}(3a/4+\varepsilon,\tau)
$
one has
$$\frac{1}{N^2}{\rm Var}\left(\sum\frac{1}{z-\lambda_i}\right)\ll \frac{N^{\frac{3a}{4}}}{N\eta}\max(\kappa_E^{1/2},\eta^{1/2}).$$
\end{lemma}

\begin{lemma}\label{lem:edgeVar2}
Assume that rigidity at scale $a$ holds, and moreover that the extra rigidity at scale $3a/4$ holds except for a few edge particles, in the following sense:
for
any $\e>0$, there are constants
$c, N_0>0$ such that for any
$N\geq N_0$ and $\hat k\geq N^{\frac{3a}{4}+\e}$  we have
$$
\P^\mu\left(|\lambda_k-\gamma_k|> N^{-\frac{2}{3}+\frac{3a}{4}+\e}(\hat k)^{-\frac{1}{3}}\right)\leq \ e^{- N^c}.
$$
Let $d> 2a/3$ and $\tau>0$ be small enough.
Then uniformly in $\Sigma^{(N)}_{\rm Ext}(d,\tau)$
one has
$$\frac{1}{N^2}{\rm Var}\left(\sum\frac{1}{z-\lambda_i}\right)\ll
\frac{1}{N\eta}\max(\kappa_E^{1/2},\eta^{1/2}).
$$
\end{lemma}

\begin{lemma}\label{lem:edgeVar}
Assume that rigidity at scale $a$ holds, and moreover that the extra rigidity at scale $3a/4$ holds except for a few edge particles, in the following sense:
for
any $\e>0$, there are constants
$c, N_0>0$ such that for any
$N\geq N_0$ and $\hat k\geq N^{\frac{3a}{4}+\e}$  we have
$$
\P^\mu\left(|\lambda_k-\gamma_k|> N^{-\frac{2}{3}+\frac{3a}{4}+\e}(\hat k)^{-\frac{1}{3}}\right)\leq e^{- N^c}.
$$
Let $a>d>s> a/2$.
Then uniformly in $\Omega^{(N)}(d,s,\tau)$ (defined in (\ref{eqn:defOmega}))
we have
\begin{align}
&\frac{1}{N}m_N'(z)=\OO\left( N^{-\frac{2}{3}+\frac{3a}{4}+\e-2s}\mathds{1}_{|z-A|<N^{-\frac{2}{3}+a+\e}}+N^{-1+\e}|z-A|^{-1/2}\right),\label{eqn:edgeEst1}\\
&\frac{1}{N^2}\var\left(\frac{1}{N}\sum_{i=1}^N\frac{1}{z-\lambda_i}\right)=\OO\left(
N^{-\frac{2}{3}-4s+2 a+\e}\mathds{1}_{|z-A|<N^{-\frac{2}{3}+a+\e}}+N^{-2+a+\e}|z-A|^{-2}\right)\label{eqn:edgeEst2}.
\end{align}
\end{lemma}

We will need to transfer information on the  Stieltjes
transform to the typical location of the points. The following result is similar to Lemma 2.3
in \cite{BouErdYau2012} for example, except that
this version will be suited to take into account the
weaker information on $m_N-m$ near the edges.

\begin{lemma}\label{lem:HS}

\noindent a) Let $\widetilde \varrho(s)\rd s$ be an arbitrary signed measure (depending on $N$) and let
 $$
 S(z):=\int \frac{\widetilde\varrho(s)}{z-s}\rd s
$$ be its Stieltjes transform.
Let $\tau>0$ be fixed, $\eta>N^{-1}$, $E\in[A,B]$ and $\eta_E=\kappa_E^{-\frac{1}{2}}\eta$.
Assume that for some (possibly $N$-dependent) $U$ we have
\begin{align}
&\left| S(x+\ii y)\right|\leq \frac{ U}{N y}\;\;  \mbox{for any $x\in[E,E+\eta_E]$ and}\;\; \eta_E <y<\tau,\label{eqn:cond1}\\
&\left| S(x+\ii y)\right|\leq \frac{ U}{N}\;\;  \mbox{for any $x\in\mathbb{R}$ and}\;\; \tau/2 <y<\tau,\label{eqn:cond2}\\
&\mbox{there is a constant $L>0$ such that for any $N$ and $|s|>L$, $| \widetilde\varrho(s)| \leq |s|^{-cN}$}\label{eqn:cond3}.
\end{align}

 Define a function $f=f_{E,\eta_E}$: $\R\to \R$
 such that $f(x) = 1$ for $x\in (-\infty, E]$, $f(x)$ vanishes
for  $x\in [E+\eta_E, \infty)$, moreover
 $|f'(x)|\leq c\,{\eta_E}^{-1}$ and $|f''(x)|\leq c\,{\eta_E}^{-2}$, for some constant $c$ independent of $N$.
Then for some  constant $C>0$, independent of $N$, we have
$$
\left|\int f(\lambda)\wt\varrho(\lambda)\rd\lambda \right|  \le
C\ \frac{U(\log N )}{N}.
$$
\noindent b) The same result holds for a specific value of $E$ below $A$, namely
for $E$ that is the unique solution of the equation
 $E=A-2\eta_E$.
\end{lemma}

\begin{proof}
We prove a), item b) is analogous. From  (B.13) in \cite{ERSY}:
\begin{align*}
\left|\int_{-\infty}^\infty f(\lambda)\wt\varrho(\lambda)\rd\lambda\right|
\leq&
C\left|\iint y f''(x)\chi(y)\Im S(x+\ii y)\rd x\rd y\right|\\
&+
C\iint\left(|f(x)|+|y||f'(x)|\right)|\chi'(y)|\left|S(x+\ii y)\right|\rd x\rd y,
\end{align*}
for some universal $C>0$,
where $\chi$ is a smooth cutoff function with support in $[-1, 1]$, with $\chi(y)=1$ for $|y| \leq \tau/2$
 and with bounded derivatives.
From (\ref{eqn:cond3}) all integrals can actually be restricted to a compact set. From (\ref{eqn:cond2}) the second integral is
$\OO(U/N)$.

Concerning the first integral, we split it into the domains $0<y<\eta_E$ and $\eta_E<y<1$.  By symmetry we only need to consider positive $y$.
The integral on the domain $\{0<y<\eta_E\}$ is easily bounded by
$$
\left|\iint_{0<y<\eta_E}y f''(x)\chi(y)\Im S(x+\ii y)\rd x\rd y\right|=\OO\left(\iint_{|x-E|<\eta_E, 0<y<\eta_E}y{\eta_E}^{-2}\frac{U}{Ny}\rd x\rd y\right)=\OO\left(\frac{U}{N}\right).
$$
On the domain $\{\eta_E<y<1\}$, we integrate by parts twice (first in $x$, then in $y$), and use the Cauchy-Riemann
equation ($\partial_x \Im S=-\partial_y\Re S$) to obtain:
\begin{align}
\iint_{y>\eta_E}y f''(x)\chi(y)\Im S(x+\ii y)\rd x\rd y=
&-\iint_{y>\eta_E}f'(x)\partial_y(y\chi(y))\Re S(x+\ii y)\rd x\rd y\label{eqn:1st}\\
&-\int f'(x)\eta_E\chi(\eta_E)\Re S(x+\ii\eta_E)\rd x.\label{eqn:2nd}
\end{align}
The first term (\ref{eqn:1st}) can be bounded by (\ref{eqn:cond1}), it is
$$\OO\left(\int_{\eta_{E}<y<\tau}\frac{U}{N y}\rd y\right)=\OO\Big(\frac{U(\log N)}{N}\Big).$$
The second term, (\ref{eqn:2nd}), can also be bounded thanks to (\ref{eqn:cond1}), by $\OO(\frac{U}{N})$,
concluding the proof.
\end{proof}

\begin{lemma}\label{lem:smallest}
Assume that rigidity at scale $a$ holds, and moreover that the extra rigidity at scale $3a/4$ holds except for a few edge particles, in the following sense:
for
any $\e>0$, there are constants
$c, N_0>0$ such that for any
$N\geq N_0$ and $\hat k\geq N^{\frac{3a}{4}+\e}$  we have
\begin{equation}\label{eqn:edgeEst3}
\P^\mu\left(|\lambda_k-\gamma_k|> N^{-\frac{2}{3}+\frac{3a}{4}+\e}(\hat k)^{-\frac{1}{3}}\right)\leq  e^{- N^c}.
\end{equation}
Then for any $d>\frac{25}{28}a$ and large enough $N$, we have
\begin{equation}\label{eqn:gamma1Est}
\gamma_1^{(N)}\geq A-N^{-\frac{2}{3}+d}.
\end{equation}
\end{lemma}

\begin{proof}
To prove (\ref{eqn:gamma1Est}) we will rely on lemmas \ref{lem:Joh2} and \ref{lem:edgeVar}.
From the hypothesis (\ref{eqn:edgeEst3}) the conclusions
(\ref{eqn:edgeEst1}) and (\ref{eqn:edgeEst2}) hold (our final choice for $s, d$
will satisfy the required bounds: $a>d>s>a/2$).
As a consequence, to check the a priori bound (\ref{eqn:aprioriBound})
 uniformly on $\Omega^{(N)}(d,s,\tau)$,
it is sufficient to prove that for $\e$ small enough,
$$
\left\{\begin{array}{l}
N^{-\frac{2}{3}+\frac{3a}{4}+\e-2s}+N^{-1+\e}|z-A|^{-\frac{1}{2}} +N^{-1+\e}=\oo(|z-A|)\\
N^{-\frac{2}{3}-4s+2a+\e}+N^{-2+a+\e}|z-A|^{-2}+N^{-1+\e}=\oo(|z-A|)
\end{array}
\right.,\ \;
\mbox{i.e.,} \;
d> \min\Big\{\frac{3a}{4}-2s,2a-4s,\frac{a}{3}, -\frac{1}{3}\Big\}.
$$
These conditions hold trivially when $s>a/2$ for example, which will be true
with our choice. As $s>a/2$, the first two terms in the $\min$ are harmless, and
$d>a/3$ will be satisfied in our final choice for $d$ (we will have $d>25a/28$).
 The last constraint for $d$  is trivial.

Assume that one can find arbitrarily large $N$ such that $\gamma_1^{(N)}\leq A-N^{-\frac{2}{3}+d}$.
We choose $z=\gamma_1^{(N)}+\ii N^{-\frac{2}{3}+s}=E+\ii\eta$.
We have, by Lemmas \ref{lem:Joh2} and \ref{lem:edgeVar}
\begin{multline}\label{eqn:2ndcondition}
|m_N(z)-m(z)|\leq C\ |z-A|^{-\frac{1}{2}}\left(\var_N(z)+\frac{1}{N}|m_N'(z)|
+N^{-1+\e}\right)\\
\leq C\ |z-A|^{-\frac{1}{2}}\left(
N^{-\frac{2}{3}+b+\e-2s}+N^{-1+\e}|z-A|^{-\frac{1}{2}}
+
N^{-\frac{2}{3}-4s+2a+\e}+N^{-2+a+\e}|z-A|^{-2}
+N^{-1+\e}
\right).
\end{multline}
On the other hand for any $b$ (we will choose $b$ greater and close to $3a/4$), we have (in the first inequality we use that the concentration scale of $\lambda_1$ around $\gamma_1^{(N)}$, $N^{-\frac{2}{3}+\frac{a}{2}}$, is much smaller than the $\eta$ scale $N^{-\frac{2}{3}+s}$),
\begin{align}
\frac{1}{N\eta}&\leq -\frac{C}{N}\E\left(\Im\left(\frac{1}{z-\lambda_1}\right)\right)\notag\\
&\leq -\frac{C}{N}\E\left(\Im\left(\frac{1}{z-\lambda_1}-\frac{1}{z-\gamma_1}\right)\right)+\OO\left(\frac{\eta}{N|z-A|^2}\right)\notag\\
&\leq -\frac{C}{N}\E\left(\Im\left(\frac{1}{z-\lambda_1}-\frac{1}{z-\gamma_1}\right)+\sum_{2\leq i\leq N^{b},\gamma_i^{(N)}\leq \gamma_i}\left(\Im\left(\frac{1}{z-\lambda_i}-\frac{1}{z-\gamma_i}\right)\mathds{1}_{\lambda_i\leq \gamma_i}\right)\right)\label{eqn:Neta}
+\OO\left(\frac{\eta}{N|z-A|^2}\right),
\end{align}
where for the last inequality we simply used that $-\Im(1/(z-\lambda_i)-1/(z-\gamma_i))\geq 0$
 whenever $|z-\lambda_i|\mathds{1}_{\lambda_i\leq \gamma_i}\leq
|z-\gamma_i|\mathds{1}_{\lambda_i\leq \gamma_i}$. The latter inequality holds with probability
 $1-\OO(e^{-N^{c}})$ since its complement is included in $|\lambda_i-\gamma_i^{(N)}|>|A-\gamma_1^{(N)}|=N^{-\frac{2}{3}+d}$,
 but $\lambda_i$ is concentrated at scale $a/2$ and $d>a/2$ in our final choice for $d$ ($d>25a/28$).

We now want to remove the assumption $\mathds{1}_{\lambda_i\leq\gamma_i}$ from (\ref{eqn:Neta}) and bound the associated error term.
For any $i\leq N^b$ such that $\lambda_i>\gamma_i $ we have
\begin{equation}\label{eqn:calc}
|\lambda_i-\gamma_i|\le |A-\gamma_{\lfloor N^b\rfloor}| + |\lambda_{\lfloor N^b\rfloor}-\gamma_{\lfloor N^b\rfloor}|
=\OO( N^{-\frac{2}{3}+\frac{2}{3}b} + N^{-\frac{2}{3} + \frac{3}{4}a - \frac{1}{3}b+\e})=\OO( N^{-\frac{2}{3} +\frac{2}{3} b})
\end{equation}
 where we used $b>3a/4$, and  chose $\e>0$ so small that $\e \le 3a/4-b$. We also used
that $\lambda_{\lfloor N^b\rfloor}$  is rigid at scale
 $3a/4$ and these bounds hold outside of a set of exponentially small probability.
Our final choice of $b$ and $d$ will satisfy $2b/3<d$ ($b=3a/4+\e$, $d=25a/28+\e$),
consequently for any $i\leq N^b$ such that $\lambda_i>\gamma_i $
 we have $|\lambda_i-\gamma_i|\ll N^{-\frac{2}{3}+d}$ and we can apply
\begin{equation}\label{eqn:diffIm}
\left|\Im\left(\frac{1}{z-\lambda}-\frac{1}{z-\gamma}\right)\right|=|\lambda-\gamma|\OO\left(\frac{\eta}{| z
-\gamma|^3}\right)
\end{equation}
that holds for  any real $\lambda,\gamma$,
and $z=E+\ii\eta$ such that $|\lambda-\gamma|\ll|E-\gamma|$.
 This condition is satisfied since
\be \label{conddd}
|\lambda_i-\gamma_i|\ll N^{-\frac{2}{3}+d}\le |E-A| \le |E-\gamma_i|.
\ee

Thus,
using (\ref{eqn:calc}) and $|z-A|\le |z-\gamma_i|$ we obtain
$$
\E\left(\Im\left(\frac{1}{z-\lambda_i}-\frac{1}{z-\gamma_i}\right)\mathds{1}_{\lambda_i> \gamma_i}\right)
\leq
\E(|\lambda_i-\gamma_i|\1_{\lambda_i>\gamma_i})\frac{\eta}{|z-A|^3}=\OO\left(\frac{N^{-\frac{2}{3}+\frac{2}{3}b}\eta}{|z-A|^3}\right).
$$
Consequently, from (\ref{eqn:Neta})
we obtain
\begin{multline}\nonumber
\frac{1}{N\eta}
\leq  -\frac{1}{N}\E\left(\Im\left(\frac{1}{z-\lambda_1}-\frac{1}{z-\gamma_1}\right)+\sum_{2\leq i\leq N^{b},\gamma_i^{(N)}\leq \gamma_i}\left(\Im\left(\frac{1}{z-\lambda_i}-\frac{1}{z-\gamma_i}\right)\right)\right)\\
+\OO\left(\frac{\eta}{N|z-A|^2}\right)+
\OO\left(\frac{N^{-\frac{2}{3}+\frac{2}{3}b+b}\eta}{|z-A|^3}\right),
\end{multline}
which implies
\begin{multline}\label{eqn:contradiction}
\frac{1}{N\eta}\leq |m_N(z)-m(z)|+\frac{1}{N}\sum_{i\geq N^{b}}\E\Im\left(\frac{1}{z-\lambda_i}-\frac{1}{z-\gamma_i}\right)
+\frac{1}{N}\sum_{ 2\le i\leq N^{b},\gamma_i^{(N)}>\gamma_i}\E\Im\left(\frac{1}{z-\lambda_i}-\frac{1}{z-\gamma_i}\right)\\
+\OO\left(\frac{\eta}{N|z-A|^2}\right)+
\OO\left(\frac{N^{-\frac{2}{3}+\frac{2}{3}b+b}\eta}{|z-A|^3}\right).
\end{multline}
Because of accuracy at scale $3a/4 (\le b)$ and concentration at scale $a/2$ for particles with index $i\geq N^b$,
 we also have (using (\ref{eqn:diffIm})  and \eqref{conddd})
\begin{equation}\label{eqn:3rdcondition}
\frac{1}{N}\E\left|\sum_{i\geq N^{b}}\Im\left(\frac{1}{z-\lambda_i}-\frac{1}{z-\gamma_i}\right)\right|
\leq
\frac{C}{N}\sum_{i\geq 1}\frac{N^{-\frac{2}{3}+b}i^{-\frac{1}{3}}\eta}{|z-\gamma_i|^3}\leq C\ N^{-\frac{1}{3}+b+s-2d}.
\end{equation}
Moreover, for $\gamma_i^{(N)}\geq \gamma_i$ and $i\leq N^b$, we have for any $\e>0$ and large enough $N$, $\E(|\lambda_i-\gamma_i|)\leq N^{-\frac{2}{3}+\frac{2}{3}b+\e}$,
so
\begin{equation}\label{eqn:4thcondition}
\frac{1}{N}\E\left|\sum_{ 2 \le i\leq N^{b},\gamma_i^{(N)}>\gamma_i}\Im\left(\frac{1}{z-\lambda_i}-\frac{1}{z-\gamma_i}\right)\right|
\leq
\frac{1}{N}
\sum_{i\leq N^b}\frac{N^{-\frac{2}{3}+\frac{2}{3}b+\e}\eta}{|z-\gamma_i|^3}\leq
 N^{b- 1}\frac{N^{-\frac{2}{3}+\frac{2}{3}b+\e}\eta}{|z-A|^3}.
\end{equation}
Consequently, when comparing the exponents of $N$ in equations (\ref{eqn:contradiction}), using the estimates (\ref{eqn:2ndcondition}),
(\ref{eqn:3rdcondition}) and (\ref{eqn:4thcondition}),
and using that $|z-A|\ge N^{-\frac{2}{3}+d}$ and $\eta = N^{-\frac{2}{3}+s}$,
one of the following inequalities holds:
$$
\left\{
\begin{array}{l}
-s\leq b-2s-\frac{d}{2}\\
-s\leq -d\\
-s\leq 2a-4s-\frac{d}{2}\\
-s\leq a-\frac{5}{2}d\\
-s\leq -\frac{1}{3}-\frac{d}{2} \\
-s\leq b+s-2d\\
-s\leq \frac{5}{3}b+s-3d.
\end{array}
\right.
$$
For the choice $b=\frac{3}{4}a+\e$, $d=\frac{25}{28} a+\e$, $s=\frac{29}{56}a$, and $\e>0$ small enough, one can check that none of these equations is satisfied (these optimal constants $25/28$ and $29/56$ are obtained when, for $b=3a/4$, the third and fifth equations are equal).
This is a contradiction concluding the proof.
\end{proof}

\begin{proof}[Proof of Proposition \ref{prop:ImprAcc}]
For simplicity we will improve accuracy only for particles close to the edge $A$,  $k\leq N/2$, the other edge being proved in a similar way.
We assume rigidity at scale $a$. By Proposition \ref{prop:ImprConc} concentration at scale $a/2$ holds. Therefore,
for any $\varepsilon>0$, by
Lemma \ref{lem:varianceBound},  uniformly on $\Sigma^{(N)}_{\rm Int}(3a/4+\varepsilon,\tau)$
we have 
$$
\var_N(z)\ll\frac{N^{\frac{3a}{4}}}{N\eta}\max(\kappa_E^{\frac{1}{2}},\eta^{\frac{1}{2}}).
$$
This easily implies that
\begin{equation}\label{eqn:boundVar}
\var_N(z)\ll \max(\kappa_E,\eta)
\end{equation} uniformly on $\Sigma^{(N)}_{\rm Int}(3a/4+\varepsilon,\tau)$.
To see this, as $\eta\geq N^{-1+\frac{3a}{4}}\kappa_E^{-1/2}$
we always have $\frac{N^{\frac{3a}{4}}}{N\eta}\kappa_E^{\frac{1}{2}}\leq \kappa_E$. Moreover,
if $\eta\geq \kappa_E$ we have $\eta\geq N^{-1+\frac{3a}{4}}\kappa_E^{-\frac{1}{2}}\geq
N^{-1+\frac{3a}{4}}\eta^{-\frac{1}{2}}$ , so
$\eta\geq N^{-\frac{2}{3}+\frac{a}{2}}$, so $\frac{N^{\frac{3a}{4}}}{N\eta}\eta^{\frac{1}{2}}\leq \eta$,
completing the proof of (\ref{eqn:boundVar}).

Consequently, the conclusion of Lemma \ref{lem:Johansson} holds:
uniformly on $\Sigma^{(N)}(3a/4+\varepsilon,\tau)$, we have
$$
|m_N(z)-m(z)|\leq c \left(\frac{ N^\e}{N\eta}+|\var_N(z)|
  \frac{1}{ \max( \sqrt{\kappa_E},\sqrt{\eta})}\right)\leq
c\frac{N^{\frac{3a}{4}}}{N\eta}.
$$
One can therefore apply  Lemma \ref{lem:HS} with
the choice $\widetilde\varrho = \varrho^{(N)}_1-\varrho$,  $\eta=N^{-1+\frac{3a}{4}}$
and  $U=N^{\frac{3a}{4}}$ (the extra assumption (\ref{eqn:cond2}) about the macroscopic behaviour of
$m_N-m$ holds thanks to Lemma \ref{lem:keyFar} and condition (\ref{eqn:cond3}) is satisfied thanks to (\ref{eqn:BPS2})): we proved that, for any $b>3a/4$, we have
\begin{equation}\label{eqn:HSConcl}
\Big|\int f(\lambda)(\varrho_1^{(N)}(\lambda)-\varrho(\lambda))\rd\lambda\Big|\leq C\ N^{-1+b},
\end{equation}
for any $E\in(A,B)$. Here $f=f_E = f_{E, \eta_E}$ as defined
in  Lemma \ref{lem:HS}.
We choose some $E\geq A+N^{-\frac{2}{3}+\frac{a}{2}}$, so that
$\eta_E =\kappa_E^{-1/2}\eta \le N^{-\frac{2}{3}+\frac{a}{2}}$, thus
$E-\eta_E\geq A$. We therefore have, using (\ref{eqn:HSConcl}),
\begin{align*}
&\int_{-\infty}^E\varrho_1^{(N)}\geq \int f_{E-\eta_E} \varrho_1^{(N)}=\int (\varrho_1^{(N)}-\varrho)f_{E-\eta_E}+\int \varrho f_{E-\eta_E}
=
\OO\left(N^{-1+b}\right)+\int_{-\infty}^{E}\varrho+\OO(\eta),\\
&\int_{-\infty}^E\varrho_1^{(N)}\leq \int f_{E} \varrho_1^{(N)}=\int (\varrho_1^{(N)}-\varrho)f_{E}+
\int \varrho f_{E}
=
\OO\left(N^{-1+b}\right)+\int_{-\infty}^E\varrho+\OO(\eta).
\end{align*}
The error $O(\eta)$ can be included into the first error term.
We first assume that $k\geq N^b$,
 and we choose $E=\gamma_k^{(N)}$ (as defined in (\ref{gammadefN}))
in the above equations, where the condition $E-\eta_E\geq A$ is satisfied when $k\geq N^{b}$. We get $|\int_{\gamma_{k}^{(N)}}^{\gamma_k}\varrho|=\OO\left(N^{-1+b}\right)$, hence
\begin{equation}\label{eqn:11}
\left|\left({\gamma_k^{(N)}}\right)^{3/2}-{\left(\gamma_k\right)^{3/2}}\right|=\OO(N^{-1+b}).
\end{equation}
This implies accuracy at scale $3a/4$: if $k\geq N^b$ we have $\gamma_k^{3/2}\geq c N^{-1+b}$,
so by linearizing (\ref{eqn:11}) we obtain
$$
\gamma_k^{(N)}=(\gamma_k^{3/2}+\OO(N^{-1+b}))^{2/3}=\gamma_k\left(1+\frac{N^{-1+b}}{\gamma_k^{3/2}}\right)^{2/3}=\gamma_k+\OO\left(N^{-1+b}/\gamma_k^{1/2}\right)
=
\gamma_k+\OO(N^{-\frac{2}{3}+b}k^{-\frac{1}{3}}).
$$
We proved that accuracy at scale $3a/4$ holds provided that
$\hat k\geq N^b$, for $b$ arbitrarily close to $3a/4$.
We know that, for such $k$, together with concentration at scale $a/2$ this implies rigidity at scale $3a/4$ (by the same reasoning as in the proof of Theorem \ref{thm:rigidity}).
This allows us first to use Lemma \ref{lem:smallest}
to obtain that, for any $\e>0$, for large enough $N$ we have
\begin{equation}\label{eqn:smallest}
\gamma_1^{(N)}\geq A-N^{-\frac{2}{3}+\frac{25}{28}a}.
\end{equation}
 It also allows us to use Lemmas \ref{lem:edgeVar2} and \ref{lem:Johansson}
together to conclude that for any $d>2a/3$ and $\tau>0$ small enough, we have, uniformly in $\Sigma_{\rm Ext}^{(N)}(d,\tau)$,
$$
|m_N(z)-m(z)|\leq \frac{1}{N\eta}.
$$
By part $b)$ of Lemma $\ref{lem:HS}$, with $\eta=N^{-1+\frac{3d}{2}}$, $E=A-2N^{-\frac{2}{3}+d}$,
 this implies that there is a function $f= 1$ on $(-\infty,A-2N^{-\frac{2}{3}+d}]$, $f= 0$ on
$[-N^{-\frac{2}{3}+d},\infty)$, such that
$$
\left|\int f(\lambda)\varrho^{(N)}_1(\lambda)\rd\lambda\right|\leq C\ \frac{\log N}{N},
$$
(since in this interval $\varrho=0$),
hence there is some $c>0$ such that for large enough $N$ we have
$$
\gamma^{(N)}_{\lfloor c\log N\rfloor}>A-N^{-\frac{2}{3}+d}.
$$
In particular, as $N^{-\frac{2}{3}+\frac{11}{12}a}j^{-\frac{1}{3}}=N^{-\frac{2}{3}+\frac{2}{3}a}$ when $j=N^{\frac{3}{4}a}$,
the previous equation proves accuracy at scale $11a/12$ for any $\lambda_i$ with $i\in\llbracket C\log N, N^{\frac{3a}{4}+\e}\rrbracket$.
For the remaining $i\in\llbracket 1, C\log N\rrbracket$, we use (\ref{eqn:smallest}), which also gives accuracy at scale
$11a/12$ because $25/28<11/12$.
\end{proof}

\subsection{Proof of Theorem \ref{thm:condRig}}\label{subsec:condRig}

This proof goes along the same lines as the one of Theorem \ref{thm:rigidity}
 up to two major differences that make it easier:
\begin{itemize}
\item For large enough $N$, the Hamitonian $\mathcal{H}_{\bf y}$
will be shown to be convex, so there is no need for introducing any convexified measure.
\item In the hypothesis of Theorem \ref{thm:condRig} about rigidity for local measures,
 the definition of the good set $\cR_K^*(\xi)$ already assumes a strong form of accuracy: $|\E^{\sigma_\by}(x_k)-\gamma_k|\leq N^{-\frac{2}{3}+\xi} k^{-\frac{1}{3}}$. Therefore there will be no need to prove an analogue of Proposition \ref{prop:ImprAcc}.
\end{itemize}

By the following easy lemma,
the first particle $x_1$ satisfies a strong form of rigidity concerning deviations on the left.

\begin{lemma}\label{lem:concX1}
There exists a constants $c,C>0$ depending only on $\beta, V,\xi$ such that for any $K$ and $\by\in\cR^*=\cR_K^*(\xi)$ we have,
for any $u>0$,
$$
\Prob^{\sigma_\by}\left(x_1\leq -u N^{-\frac{2}{3}+\xi}\right)\leq C\, e^{-c u^2}.
$$
\end{lemma}

\begin{proof}
We note
$$
\mathcal{H}_\by(\bx)=\frac{1}{2}\sum_I V_\by(x_i)-\frac{1}{N}\sum_{i<j}\log (x_j-x_i).
$$
Then
$$
Z_{\hat\sigma_\by}\leq
2Z_{\hat\sigma_\by}
\frac{1}{Z_{\hat\sigma_\by}}\int
e^{-\beta N \left(\mathcal{H}_\by(x)+\frac{1}{N}\sum_I\Theta(N^{\frac{2}{3}-\xi}x_i)
+\frac{1}{N}\sum_I\Theta(N^{\frac{2}{3}-\xi}x_i)
\right)}\rd x=2Z_{\sigma_\by}.
$$
In the above inequality we used $\P^{\hat\sigma_\by}(x_1\geq -N^{-\frac{2}{3}+\xi})\geq 1/2$ (because $\by\in\cR^*$) and
$\Theta(N^{\frac{2}{3}-\xi}x_i)=0$ when $x_i\geq -N^{-\frac{2}{3}+\xi}$.
We then easily get, for $u\geq 2$,
\begin{multline*}
\Prob^{\sigma_\by}(x_1\leq -u N^{-\frac{2}{3}+\xi})=
\frac{Z_{\hat\sigma_\by}}{Z_{\sigma_\by}}\frac{1}{Z_{\hat\sigma_\by}}
\int
e^{-\beta N \left(\hat{\mathcal{H}}^\sigma_\by+\frac{1}{N}\sum_I\Theta(N^{\frac{2}{3}-\xi}x_i)
\right)}
\mathds{1}\left(x_1\leq -u N^{-\frac{2}{3}+\xi}\right)\rd x\\
\leq 2 \E^{\hat\sigma_\by}\left(e^{-\beta\Theta(N^{\frac{2}{3}-\xi}x_1)}\mathds{1}\left(x_1\leq -u N^{-\frac{2}{3}+\xi}\right)\right)\leq
2e^{-\beta (u-1)^2}.
\end{multline*}
This concludes the proof (bounding the probability by 1 when $0<u<2$).
\end{proof}
The following notion of conditional rigidity at scale $M$ will be useful in our proof of optimal conditional rigidity,
i.e., Theorem \ref{thm:condRig}.
It is analogous to Definition 6.1 in \cite{EYsinglegap}, which was in the context of bulk eigenvalues.

\begin{definition}
Given $\xi>0$,  we will say that the measure $\sigma_\by$ satisfies {\it conditional rigidity at scale $M$}
if there exists $c>0$
such that  for large enough $N$ we have, for any $\by \in\mathcal{R}^*_{K}(\xi)$,
$\ell\in I$ and $u>0$,
$$
\Prob^{\sigma_{\bf y}}\left(|x_\ell-\gamma_\ell|>N^{-\frac{2}{3}+\xi}\ell^{-\frac{1}{3}}u
+
N^{-\frac{2}{3}}\ell^{-\frac{1}{3}}M
\right)
\leq
e^{-c u^2}.
$$
The parameter $\xi$ is considered fixed in this definition.
\end{definition}

Following ideas from Section 6.1 in  \cite{EYsinglegap}, we set $\eta=\xi/3$ and will consider a sequence
$N^\xi=M_1<\dots<M_A=C K N^{-2\eta}$ (for some large constant $C$) such that for any $j\in\llbracket 1,A-1\rrbracket$ we have
$M_{j+1}/M_{j}\sim N^\eta$ (meaning that $c N^\eta<M_{j+1}/M_{j}<C N^\eta$).
Here $A$ is a constant bounded by $\OO(\xi^{-1})$.
Our first task is to prove conditional rigidity at scale $M_A$ for $\sigma_\by$.\\

\noindent{\bf Step 1: conditional rigidity at a large scale.} The Hamiltonian $\mathcal{H}^\sigma_\by$ satisfies the following convexity bound: for any $\bv\in\bR^K$,
$$
\langle \bv,(\nabla^2 \mathcal{H}^\sigma_\by)\bv \rangle\geq c\sum_{i\in I}  \left(\sum_{j\not\in I}
\frac{1}{N(x_i-y_j)^2}-2W+N^{-1+\frac{4}{3}-2\xi}\Theta''(N^{\frac{2}{3}-\xi}x_i)\right)|v_i|^2.
$$
If $x_i\geq -N^{-\frac{2}{3}+\xi}$,  we get that $\sum_{j\not\in I}
\frac{1}{N(x_i-y_j)^2}\geq cN^{-1}\sum_{j=K}^{N/2}(j/N)^{-4/3}\geq c (N/K)^{1/3}$ (remember that the rigidity exponent $\xi$ is much smaller than the exponent $\delta$ in (\ref{NK}), so $|x_i-y_j|\leq C N^{-\frac{2}{3}+\xi}+|y_j|\leq C N^{-\frac{2}{3}+\xi}+C(j/N)^{\frac{2}{3}}+CN^{-\frac{2}{3}+\xi}j^{-\frac{1}{3}}\leq C(j/N)^{\frac{2}{3}}$ for $j\geq K$). If
$x_i\leq -N^{-\frac{2}{3}+\xi}$ then $N^{-1+\frac{4}{3}-2\xi}\Theta''(N^{\frac{2}{3}-\xi}x_i)\geq c N^{1/3-2\xi}\geq
c(N/K)^{\frac{1}{3}}
$, so in all cases we proved the inequality
$$
\nabla^2 \mathcal{H}^{\sigma}_\by\geq c\ (N/K)^{\frac{1}{3}}.
$$
The measure $\sigma_\by$ therefore satisfies a logarithmic Sobolev inequality with constant of order
$K^{-1/3}N^{4/3}$, so we have for any $\ell\in\llbracket  1,K- M_A\rrbracket$ that
$$
\Prob^{\sigma_\by}\left(|x_\ell^{[M_A]}-\E^{\sigma_\by}(x_\ell^{[M_A]})|>v\right)\leq \exp(-c M_A K^{-1/3}N^{4/3} v^2),\ v\geq 0.
$$
In particular,
\begin{equation}\label{eqn:concAv}
\Prob^{\sigma_\by}\left(|x_\ell^{[M_A]}-\E^{\sigma_\by}(x_\ell^{[M_A]})|>N^{-2/3+\xi}\ell^{-1/3}u\right)
\leq
e^{-c u^2},
\end{equation}
because $N^{2\xi}>K^{2\eta}$. Moreover, using the definition (\ref{R*def}), we know that
$$
\left|\E^{\sigma_\by}(x_\ell^{[M_A]})-\gamma_\ell^{[M_A]}\right|\leq C N^{-\frac{2}{3}+\xi}\ell^{-1/3}.
$$
We therefore proved that
$$
\Prob^{\sigma_\by}\left(|x_\ell^{[M_A]}-\gamma_\ell^{[M_A]}|>C N^{-\frac{2}{3}+\xi}\ell^{-\frac{1}{3}} +N^{-2/3+\xi}\ell^{-1/3}u\right)
\leq
e^{-c u^2}.
$$
Moreover, from easy ordering considerations we have for any $\ell\in\llbracket M_A, K-M_A\rrbracket$
\begin{align}
&x_\ell-\gamma_\ell\leq (x_{\ell}^{[M_A]}-\gamma_{\ell}^{[M_A]})+(\gamma_{\ell}^{[M_A]}-\gamma_{\ell})
\leq (x_{\ell}^{[M_A]}-\gamma_{\ell}^{[M_A]})+C M_A N^{-\frac{2}{3}}\ell^{-\frac{1}{3}},\label{eqn:right}\\
&x_\ell-\gamma_\ell\geq (x_{\ell-M_A}^{[M_A]}-\gamma_{\ell-M_A}^{[M_A]})+(\gamma_{\ell-M_A}^{[M_A]}-\gamma_{\ell})
\geq (x_{\ell-M_A}^{[M_A]}-\gamma_{\ell-M_A}^{[M_A]})-C M_A N^{-\frac{2}{3}}\ell^{-\frac{1}{3}}.\non
\end{align}
Thus
\begin{equation}\label{eqn:initConc}
\Prob^{\sigma_\by}\left(|x_\ell-\gamma_\ell|>C N^{-\frac{2}{3}+\xi}\ell^{-\frac{1}{3}}M_A +N^{-2/3+\xi}\ell^{-1/3}u\right)
\leq
e^{-c u^2}.
\end{equation}
If $\ell\in\llbracket 1, M_A\rrbracket$, then the bound (\ref{eqn:right}) still holds, and for concentration on the left we
simply use
$
x_\ell-\gamma_\ell\geq x_1-C \left({\ell}{ N^{-1}}\right)^{\frac{2}{3}}
$, which yields
$$
|x_\ell-\gamma_\ell|\leq |\min(x_1,0)|+
 |x_\ell^{[M_A]}-\gamma_{\ell}^{[M_A]}|+C M_A N^{-\frac{2}{3}}\ell^{-\frac{1}{3}}.
$$
Using Lemma \ref{lem:concX1} and (\ref{eqn:concAv}), this inequality proves that the desired rigidity (\ref{eqn:initConc})
also holds for $\ell\in\llbracket 1, M_A\rrbracket$. The case $\ell\in\llbracket K-M_A, K\rrbracket$ is more elementary,
the boundary on the right being fixed:
$
|x_\ell-\gamma_\ell|\leq |y_{K+1}-\gamma_{K+1}|+|\gamma_{K+1}-\gamma_{\ell}|+
|x_{\ell-M_A}^{[M_A]}-\gamma_{\ell-M_A}^{[M_A]}|+|\gamma_\ell-\gamma_{\ell-M_A}^{[M_A]}|
$,
and the desired result (\ref{eqn:initConc}) follows from the definition of $\cR$ to bound $|y_{K+1}-\gamma_{K+1}|$, (\ref{eqn:concAv}) and bounding of
$|\gamma_{K+1}-\gamma_\ell|$
and $|\gamma_\ell-\gamma_{\ell-M_A}^{[M_A]}|$ by $C M_A N^{-\frac{2}{3}}\ell^{-\frac{1}{3}}$.
This concludes the proof of conditional rigidity at scale $M_A$ for $\sigma_\by$.\\

\noindent{\bf Step 2: induction on the scales.}
In order to consider smaller scales, we will need in the following version of the locally constrained measure (\ref{eqn:omega}): for any $\ell\in
\llbracket 1,K-M\rrbracket$, define the probability measure
\begin{align}
&\rd{\omega_\by}^{(\ell,M)}(\bx)\sim\exp\left(-\beta\phi_{\rm loc}^{(\ell,M)}(\bx)\right)\rd\sigma_\by(\bx),\notag\\
&
\phi_{\rm loc}^{(\ell,M)}(\bx)=
\sum_{i<j,i,j\in I^{(\ell,M)}}\theta\left(\frac{N^{\frac{2}{3}}\ell^{\frac{1}{3}}}
{MN^{2\eta}}(x_i-x_j)\right),\label{eqn:phiLoc}
\end{align}
where $\theta(x)=(x-1)^2 \mathds{1}_{x>1}+(x+1)^2 \mathds{1}_{x<-1}$ and $I^{(\ell,M)}=\llbracket \ell,\ell+M-1\rrbracket$. Note that our definition of $\phi_{\rm loc}^{(\ell,M)}$
only differs from $\phi^{(\ell,M)}$ (see Definition \ref{def:locallyConstrained2}) concerning the extra factor $N^{2\eta}$.
We now present our induction on the scales: we will show that if the following three conditions hold for the
index $j$ then they are also true for $j-1$.
\begin{enumerate}[(i)]
\item There exists $c>0$ such that for large enough $N$,  for any $\ell\in\llbracket 1,K-M_j\rrbracket$ and $u>0$
$$
\Prob^{\sigma_\by}\left(|x_\ell^{[M_j]}-\gamma_\ell^{[M_j]}|>N^{-\frac{2}{3}+\xi}\ell^{-\frac{1}{3}}u\right)
\leq
e^{-c u^2}.
$$
\item The following conditional rigidity at scale $M_j$ holds: there exists $c,C>0$ such that for large enough $N$,  for any $\ell\in I$ and $u>0$, we have
$$
\Prob^{\sigma_\by}\left(|x_\ell-\gamma_\ell|>N^{-\frac{2}{3}+\xi}\ell^{-\frac{1}{3}}u+C\,N^{-\frac{2}{3}}\ell^{-\frac{1}{3}}M_j\right)
\leq
e^{-c u^2}.
$$
\item The following entropy bound holds, for large enough $N$ and any $\ell\in\llbracket 1,K-M_j\rrbracket$:
$$
\mbox{S}(\sigma_\by\mid\omega_\by^{(\ell,M_j)})\leq e^{-c M_j^2N^{-2\eta}}.
$$
\end{enumerate}

The initial step $j=A$ of the induction was just checked in Step1, concerning points (i) and (ii) (see equations (\ref{eqn:concAv}) and (\ref{eqn:initConc})). Concerning (iii), it follows easily from (ii): if $\phi_{\rm loc}^{(\ell,M_A)}(\bx)>0$ then
$x_1<-N^{-\frac{2}{3}}M_A N^{2\eta}$, which has $\sigma_\by$-probability bounded by $\exp(-c M_A^2N^{-2\eta})$
(by (ii)). The logarithmic Sobolev inequality for $\sigma_\by$ therefore allows us to conclude:
\begin{equation}\label{eqn:localEntropyBound}
\mbox{S}(\sigma_\by\mid\omega_\by^{(\ell,M_A)})\leq C N^C\E^{\sigma_\by}|\nabla \phi_{\rm loc}^{(\ell,M_A)}|^2\leq C N^C \exp(-c M_A^2N^{-2\eta})\leq \exp(-c' M_A^2N^{-2\eta}).
\end{equation}

We now prove that (i),(ii),(iii) with $M_{j}$ implies the same result with $M_{j-1}$.
That (i) implies (ii) is easy and follows from the exact same argument allowing to conclude about the initial local
rigidity (\ref{eqn:initConc}). To prove (iii) from (ii), note that if $\phi_{\rm loc}^{(\ell,M_j)}(\bx)>0$ then for some
$i\in I^{(\ell,M_j)}$ we have
$|x_i-\gamma_i|>c N^{-\frac{2}{3}}i^{-\frac{1}{3}}M_j N^{2\eta}$. From (ii) this has probability (for $\sigma_\by$) bounded by
 $e^{-c M_j^2N^{-2\eta}}$. One then concludes similarly to (\ref{eqn:localEntropyBound}).
We therefore now only need to prove (i) at scale $M_{j-1}$.
We have the following analogue of equation (6.12)  in \cite{EYsinglegap}:
for any choice $\ell_j\in\llbracket 1,K-M_j\rrbracket$ and
$\ell_{j-1}\in\llbracket 1,K-M_{j-1}\rrbracket$ such that
$\llbracket \ell_{j-1},\ell_{j-1}+M_{j-1}\rrbracket\subset\llbracket \ell_j,\ell_j+M_j\rrbracket$ we have
\begin{equation}\label{eqn:concDifferences}
\P^{\omega_\by^{(\ell,M_j)}}\left(\left|
x_{\ell_{j-1}}^{[M_{j-1}]}-x_{\ell_{j}}^{[M_{j}]}-\E^{\omega_\by^{(\ell,M_j)}}
\left(x_{\ell_{j-1}}^{[M_{j-1}]}-x_{\ell_{j}}^{[M_{j}]}\right)
\right|>N^{-\frac{2}{3}}\ell_j^{-\frac{1}{3}}N^{5\eta/2}u\right)\leq e^{-c u^2}.
\end{equation}
The proof of the above equation relies on Herbst's argument for concentration of measure from the logarithmic Sobolev inequality,  and Lemma 3.9 in \cite{BouErdYau2011} to obtain a local LSI. Note that the assumptions of this Lemma
are satisfied in our case: one can decompose $\mathcal{H}_\by^\sigma=\mathcal{H}_1+\mathcal{H}_2$ where $$\mathcal{H}_1(\bx)=\frac{1}{N}\phi_{\rm loc}^{(\ell,M_j)}(\bx)-\frac{1}{N}\sum_{s<t,s,t\in I^{\ell,M}}\log|x_s-x_t|$$
and $\mathcal{H}_2$ is convex, thanks to the confining term $\Theta$ which applies to all $x_i$'s, $i\in I$.
Compared to (6.12)  in \cite{EYsinglegap}, we obtained $N^{5\eta/2}$ instead of $K^{5\eta/2}$ due to
$\sqrt{M_j/M_{j-1}}=N^{\eta/2}$ and the factor $N^{2\eta}$ in (\ref{eqn:phiLoc}) instead of $K^{2\eta}$.

Moreover, using the boundedness of the $x_k$'s on the right and Lemma \ref{lem:concX1} on the left,
similarly to (\ref{eqn:localEntropyBound}) we easily obtain
$$
\left|\E^{\omega_\by^{(\ell,M_j)}}(x_i)-\E^{\sigma_\by}(x_i)\right|\leq
C\sqrt{\mbox{S}(\sigma_\by\mid\omega_\by^{(\ell,M_j)})}\leq \exp(-c M_j^2 N^{-2\eta}).
$$
We know from (\ref{R*def}) that $|\E^{\sigma_\by}x_i-\gamma_i|\leq N^{-\frac{2}{3}+\xi}i^{-\frac{1}{3}}$, so
$$
\left|\E^{\omega_\by^{(\ell,M_j)}}(x_i)-\gamma_i\right|\leq
C N^{-\frac{2}{3}+\xi}i^{-\frac{1}{3}}.
$$
Changing $\omega_\by^{(\ell,M_j)}$ into $\sigma_\by$ in the equation (\ref{eqn:concDifferences})
implies an error of order
$\sqrt{\mbox{S}(\sigma_\by\mid\omega_\by^{(\ell,M_j)})}$, which yields
$$
\P^{\sigma_\by}\left(\left|
x_{\ell_{j-1}}^{[M_{j-1}]}-x_{\ell_{j}}^{[M_{j}]}-
\left(\gamma_{\ell_{j-1}}^{[M_{j-1}]}-\gamma_{\ell_{j}}^{[M_{j}]}\right)
\right|>N^{-\frac{2}{3}}\ell_j^{-\frac{1}{3}}N^{5\eta/2}u+N^{-\frac{2}{3}+\xi}\ell_{j}^{-\frac{1}{3}}\right)\leq \exp(-c u^2)+\exp(-c M_j^2 N^{-2\eta}).
$$
Combining this with (i) and using $\xi=3\eta$ we get
$$
\P^{\sigma_\by}\left(\left|
x_{\ell_{j-1}}^{[M_{j-1}]}-
\gamma_{\ell_{j-1}}^{[M_{j-1}]}
\right|>N^{-\frac{2}{3}+\xi}\ell_j^{-\frac{1}{3}}u+N^{-\frac{2}{3}+\xi}\ell_{j}^{-\frac{1}{3}}\right)\leq \exp(-c u^2)+\exp(-c M_j^2 N^{-2\eta}).
$$
If  $\ell_{j-1}\leq K-M_{j}$ we can choose  $\ell_j=\ell_{j-1}$ in the above equation. If $\ell_{j-1}\in\llbracket K-M_j,K\rrbracket$, then we choose $\ell_j=\lfloor K-M_j\rfloor$ and we have $\ell_j\sim\ell_{j-1}$: in any case we therefore proved
$$
\P^{\sigma_\by}\left(\left|
x_{\ell_{j-1}}^{[M_{j-1}]}-
\gamma_{\ell_{j-1}}^{[M_{j-1}]}
\right|>C N^{-\frac{2}{3}+\xi}\ell_{j-1}^{-\frac{1}{3}}u+CN^{-\frac{2}{3}+\xi}\ell_{j-1}^{-\frac{1}{3}}\right)\leq \exp(-c u^2)+\exp(-c M_j^2 N^{-3\eta}).
$$
 We therefore proved (i) on scale $M_{j-1}$ provided that $u\leq c M_j N^{-\eta}$.

We now assume $u\geq c M_j N^{-\eta}$.
Note that $N^{-\frac{2}{3}+\xi}\ell_{j-1}^{-\frac{1}{3}}u\geq
c(N^{-\frac{2}{3}}\ell_{j-1}^{-\frac{1}{3}}M_j+N^{-\frac{2}{3}+\xi}\ell_{j-1}^{-\frac{1}{3}}u)$, which allows
the following bounds thanks to (ii):
\begin{multline*}
\P^{\sigma_\by}\left(\left|
x_{\ell_{j-1}}^{[M_{j-1}]}-
\gamma_{\ell_{j-1}}^{[M_{j-1}]}
\right|>N^{-\frac{2}{3}+\xi}\ell_{j-1}^{-\frac{1}{3}}u\right)
\leq
\P^{\sigma_\by}\left(\left|
x_{\ell_{j-1}}^{[M_{j-1}]}-
\gamma_{\ell_{j-1}}^{[M_{j-1}]}
\right|>c(N^{-\frac{2}{3}}\ell_{j-1}^{-\frac{1}{3}}M_j+N^{-\frac{2}{3}+\xi}\ell_{j-1}^{-\frac{1}{3}}u)\right)\\
\leq \sum_{\ell\in I^{(\ell_{j-1},M_{j-1})}}\P^{\sigma_\by}\left(\left|
x_{\ell}-\gamma_{\ell}
\right|>c(N^{-\frac{2}{3}}\ell_{j-1}^{-\frac{1}{3}}M_j+N^{-\frac{2}{3}+\xi}\ell_{j-1}^{-\frac{1}{3}}u)\right)\leq M_{j-1} e^{-c u^2}\leq e^{-c' u^2}.
\end{multline*}
This concludes the induction. Notice that
the constant $c$ in the Gaussian tail $\exp(c u^2)$ deteriorates at each step, but we perform only finitely many steps.
The result (ii) at the final scale $M_1=N^\xi$ finishes the proof of Theorem \ref{thm:condRig}.

\section{Analysis of the local Gibbs measure}\label{sec:AnLocGibbs}

Before studying $\sigma_\by$, we remind well-known properties of the equilibrium density, at the macroscopic level: 
$\varrho=\varrho_V$ can be obtained as
the unique solution to the variational problem
\be\label{varprin}
\inf\Big\{ \int_\bR V(t) \rd\varrho (t)- \int_\bR\int_\bR \log|t-s|\rd\varrho(t)\rd\varrho(s)
\; : \; 
 \mbox{$\varrho$ is a probability measure}\Big\},
\ee
and  it  satisfies the following equation 
\begin{equation}\label{equil}
   \frac{1}{2} V'(x) = \int \frac{\varrho(y)\rd y}{x-y}, \qquad x\in [A, B].
\end{equation}

\subsection{Rescaling}\label{sec:rescale}

We now switch to the microscopic coordinates with a scaling adapted  to the left edge
of the spectrum at $A=0$,
i.e., we consider the scaling transformation
$\lambda_j \to 3/2\, N^{2/3} \lambda_j$. In this new coordinate, the gaps of
the   points at the edge  are order one and
the gaps in the bulk are of order $N^{-1/3}$.  With a slight abuse of notation
we will still use the same
letters $x_j, y_j$ for the internal and external points,
 but {\it from now on they should be understood
in the microscopic coordinates except in the Appendix \ref{App:CalcLem}. }  This means that the classical location
of the $k$-th point and the $k$-th gap are
\be\label{orient1}
  \gamma_k = (\wh k)^{2/3} \big( 1+ O((k/N)^{2/3})\big), \qquad  \gamma_{k+1}- \gamma_k \sim  (\wh k)^{-1/3},
\ee
 for any $k\in  \llbracket 1, N\rrbracket$,
see \eqref{orient}
(here the constant is adjusted to be 1, from the choice of normalization \eqref{eqn:matchSing} and the scaling $\lambda_j\to3/2\, N^{2/3} \la_j$).
 Recall we partition the external and internal points as
$$
  (\bx, \by) = (x_1,x_2, \dots, x_K, y_{K+1}, y_{K+2}, \dots , y_N).
$$
Given a boundary condition $\by$, we again set $J_\by = (-\infty, y_{K+1}) =:(y_-, y_+)$
to be the configuration interval, and let $\alpha_j=\alpha_j(\by)$ be
 $(K+1)$-quantiles of the density in $J_\by$ exactly as in \eqref{aldef}:
\be
\label{aldefuj}
   \int_{0}^{\alpha_j}\varrho(s)\rd s = \frac{j}{K+1}   \int_{0}^{y_+}\varrho(s)\rd s,
  \qquad j\in I.
\ee

The  measure $\sigma$ from \eqref{sig}  in microscopic coordinate reads as
\be\label{71}
  \sigma(\rd\bx) : = \frac{Z}{Z_{\sigma}} e^{-2 \beta\sum_{i\in I}  \Theta (N^{-\xi}x_i)}\mu(\rd\bx),
\ee
and the local measures $\sigma_\by$ are defined analogously:
$$
  \sigma_\by(\rd\bx)
  =\frac{1}{Z_{\by, \sigma}}  e^{-\beta N \cH_\by^\sigma (\bx)}\rd\bx,
$$
with Hamiltonian
\begin{align}
   \cH_\by^\sigma  (\bx): = & \frac{2}{N}\sum_{i\in I}  \Theta (N^{-\xi}x_i)
 + \sum_{i\in I} \frac{1}{2}   V_\by   (x_i) -\frac{1}{N}
   \sum_{i,j\in I\atop i<j} \log |x_j-x_i|, \non \\
\label{Vyext}
   V_\by (x) := &  V(xN^{-2/3}) - \frac{2}{N}\sum_{j\not\in I} \log |x-y_j|.
\end{align}
Here  $  V_\by (x) $ can be viewed as the external potential of
the log-gas.

Recall the rigidity bound \eqref{rigcons} for $\sigma$.
The definitions of the good boundary conditions \eqref{Rdef1},
\eqref{R*def} and \eqref{wtRdef} are also rescaled:
\begin{align}
\non
  \cR=\cR_{ K}(\xi) : = & \{ \by\; : \; |y_k-\gamma_k|\le N^{ \xi} (\wh k)^{-1/3},
 \; k\not\in I  \},
\\ \non
  \cR^*=\cR^*_{ K}(\xi) :
= & \{ \by \in\cR_{K}(\xi) \; : \;\left| \E^{\sigma_{\bf y}}x_k-\gamma_k\right|\leq N^{\xi}(\wh k)^{-\frac{1}{3}},\;\;
\P^{\wh \sigma}(x_1\ge \gamma_1 - N^\xi)\ge 1/2 \;\;\forall k\in I \},
\\ \non
   \cR^\#= \cR^\#_{K}(\xi) : = &\{ \by\in \cR_{ K}( \xi/3 ) \; : \;
  |y_{K+1}-y_{K+2 }|\ge N^{-\xi}K^{-1/3}
 \}.
\end{align}
In the new coordinates, the lower bound \eqref{Hyconv} on the Hessian of
 $\cH^\sigma_\by $  reads as
\be\label{sihess}
   (\cH^\sigma_\by)''   \ge
cK^{-1/3}N^{-1}, \qquad \by\in \cR.
\ee
 We also have the rescaled form of \eqref{35}
that for any $\by\in \cR$
\be\label{351}
    |\al_j-\gamma_j|\le  C N^{  \xi} \,  \frac {  j^{2/3}} K
\le CN^\xi \, j^{-1/3}, \qquad j\in I.
\ee

The main universality result on the local measures is the following theorem,
 which is essentially the rescaled version of Theorem~\ref{thm:local}.
We will first complete the proof of Theorem~\ref{thm:local}, then the rest
of the paper is devoted to the proof of Theorem~\ref{thm:cor} which
will be completed at the end of Section~\ref{sec:energ}.

\begin{theorem}[Edge universality for local measures]\label{thm:cor}
We assume the conditions of Theorem~\ref{thm:local}, in particular
that the parameters $\xi,\delta, \zeta$ and $K$ satisfy \eqref{con1} and \eqref{Kcon1}.
Let $\by\in \cR^\#_{K,V,\beta}(\xi)\cap \cR^*_{K,V,\beta}(\xi)$ and
 $\widetilde\by\in \cR^\#_{K,\widetilde V,\beta}(\xi)\cap \cR^*_{K,\widetilde V,\beta}(\xi)$ be
two different boundary conditions
satisfying
\be\label{bc}
 y_{K+1} = \wt  y_{K+1}.
\ee
In particular, we know that
\be
\label{Exm}
  |\E^{\sigma_\by} x_j -\al_j| +  |\E^{\wt\sigma_{\wt\by}} x_j -\wt\al_j| \le CN^\xi  j^{-1/3}, \qquad j\in I,
\ee
and
\be\label{714}
 \P^{\hat \sigma} ( x_1 \ge \gamma_1 - N^{\xi}/2 )  \ge 1/2,
 \qquad \P^{  \hat {\tilde \sigma}} ( x_1 \ge \gamma_1 - N^{ \xi}/2 )  \ge 1/2.
\ee
Fix $m\in \N$.
Then
there is a  small $\chi>0$ such that
for any $\Lambda\subset \llbracket 1, K^\zeta\rrbracket$,
$|\Lambda|=m$, and  any smooth, compactly supported
observable $O:\bR^m\to \bR$, we have
\be\label{eq:compy}
  \Bigg| \E^{\sigma_\by} O\Bigg( \Big(   j^{1/3}(x_j -\al_j)
\Big)_{j\in \Lambda}\Bigg)
  -   \E^{\wt \sigma_{\wt\by}} O\Bigg( \Big (   j^{1/3}(x_j -\wt \al_j)\Big)_{j\in \Lambda}\Bigg)\Bigg| \le
CN^{-\chi}.
\ee
\end{theorem}

The main tool for proving  Theorem~\ref{thm:cor}
is the interpolating measure between $\mu_\by$ and $\wt\mu_{\wt\by}$
which will be defined in Section~\ref{sec:int}.

\subsection{Proof of Theorem~\ref{thm:local} from  Theorem~\ref{thm:cor}}

In order to prove  Theorem~\ref{thm:local}, we
 will need a slight extension of Theorem~\ref{thm:cor} result  we formulate now. We claim
that  Theorem~\ref{thm:cor} also
holds  if the measures are rescaled by an
$N$-dependent factor,
provided the rescaling factor is very close to one. More precisely,
fix a small $\ell= \ell_N = O(K^{-1})$  and define the rescaled potential
$$
    V^*  (x) = V \Big( \frac{x}{1+\ell}\Big)
$$
and the cutoff potential $\sum_{i\in I}\Theta^*(x_i)$, where
$$
   \Theta^*(x)= \Theta\Big( N^{-  \xi}x/(1 + \ell)  \Big).
$$
From $ V^*$ and $\Theta^*$, we define the rescaled measure  $ \sigma^*$
by the formula \eqref{eqn:measure} and  \eqref{71}.
For any observable $Q$  we clearly have the relation
$$
\E^{\sigma^*}Q(\bla ) =  \E^{\sigma} Q( (1+ \ell)  \bla ).
$$
Furthermore, the  equilibrium  density  $ \varrho^*$ for the measure $\sigma^*$
(defined by the variational principle \eqref{varprin};
 notice that it is independent of the cutoff $\Theta$)
satisfies
$$
\varrho^*(x) = \frac{1}{1+\ell} \varrho \Big( \frac{x}{1+\ell}\Big).
$$
Fix a boundary condition  $\by = (y_{K+1}, \dots,y_N)\in \cR^\#_{ K}\cap \cR^* $
 and define the rescaled boundary condition by
 $y^*_{j} = (1+ \ell) y_{j}$ for all $j \ge K+1$.
The  conditional measure  $ \sigma^*_{\by^*}$ thus satisfies the relation
\be\label{res}
\E^{ \sigma^*_{\by^*}} Q(x_1, \dots , x_K) = \E^{ \sigma_{\by}} Q((1+ \ell)x_1, \dots, (1+ \ell) x_K)
\ee
and we also have $\alpha_j^* = (1+ \ell) \alpha_j$.
Now we will compare the measure $\wt\sigma_{\wt\by}$ with the rescaled conditional measure $\sigma^*_{\by^*}$
assuming that they have the same configurational interval $\wt J = J^*$, i.e.,
that $y_{K+1}^* = \wt y_{K+1}$ (in applications, we will choose $\ell$
in order to match these boundary conditions).
Therefore, we would like to extend the validity of
  \eqref{eq:compy} to the rescaled measures, i.e., to conclude that
\be\label{eq:compy*}
  \Bigg| \E^{\sigma^*_{\by^*}} O\Bigg( \Big( j^{1/3}(x_j -\al^*_j)
\Big)_{j\in \Lambda}\Bigg)
  -   \E^{\wt \sigma_{\wt\by}} O\Bigg( \Big( j^{1/3}(x_j -\wt \al_j)
\Big)_{j\in \Lambda}
\Bigg)\Bigg| \le N^{-\chi}.
\ee

Notice that if $\by \in \cR^\#_{K} $ w.r.t. the measure $\sigma$
then  $\by^* \in \cR^\#_{K} $ w.r.t. the measure $\sigma^*$
by simple scaling. Again by scaling,  we have
$$
|\E^{\sigma^*_{\by^*}} x_j -\al^*_j|  =  (1+ \ell) | \E^{ \sigma_{\by}} x_j - \al_j|  \le  CN^{ \xi}   j^{-1/3}
$$
and thus \eqref{Exm} holds  w.r.t. the measure $\sigma^*_{\by^*}$.
 Furthermore, we can check \eqref{714} holds with $N^{\xi}$ replaced by
$N^{ \xi} (1 + O(K^{-1}))$. 
Instead of  \eqref{eqn:matchSing}, we now have
\be\label{rhoasympt2}
\varrho^*(t) =  \frac {  \varrho( t/(1+ \ell)  )} {1+ \ell } =  \frac {
\sqrt { t/(1+ \ell)}[ 1 + O( t/(1+ \ell))] } { 1 + \ell}
=    \sqrt { t }\Big[ 1 -\ell+ O\big( t(1 + O(\ell)\big)\Big].
\ee
In order to prove \eqref{eq:compy*},  we need to check that
 the following  proof of \eqref{eq:compy} holds with
\eqref{eqn:matchSing} replaced by \eqref{rhoasympt2} and the very minor
change of \eqref{714} just mentioned.    The task  is straightforward and we  will
only remark  on a  small  change  in the proof  near the equation \eqref{724}.

\medskip

We make another small observation. Similarly to the remark after Theorem~\ref{thm:local},
we can replace $\alpha_j$ by  $\gamma_j = j^{2/3}(1+ O\big[ (j/N)^{2/3}\big]\big)$
or simply by $j^{2/3}$ for the purpose of proving Theorem~\ref{thm:cor} as long as $j\le K^\zeta$.
This follows from the smoothness of $O$, from \eqref{351} and from
\eqref{rhoasympt2} that implies
 $\gamma_j = j^{2/3}(1+ O\big[ (j/N)^{2/3}\big)\big]= j^{2/3} + o(j^{-1/3})$.
If we are dealing with the measure $\sigma^*_{\by^*}$, then for $j \in \Lambda$ there is $\chi> 0$ such that
\be\label{724}
\alpha_j^*  = (1+ \ell) \alpha_j =  j^{2/3}(1+ O\big[ (j/N)^{2/3}\big)\big] +O\Big ( N^{\xi} \,  \frac {  j^{2/3}} K \Big ) + O(j^{2/3} K^{-1})  =  j^{2/3} + O(j^{-1/3}) N^{-\chi} .
\ee
Here we have used \eqref{351},  $\ell= O(K^{-1})$ and $\zeta$ in the definition of the set $\Lambda$
 satisfying $\zeta< 1$.

\begin{proof}[ Proof of  Theorem~\ref{thm:local}] Under the condition \eqref{bc},
 Theorem~\ref{thm:local} would directly follow from Theorem~\ref{thm:cor}.
We now prove Theorem~\ref{thm:local} in the general case.
 Suppose
that  $\by, \wt \by  \in \cR^\#_{K}\cap \cR^*_K $  but the boundary condition \eqref{bc}
  is not satisfied.
Given these two boundary conditions,
 we define $\ell=\ell(\by,\wt\by) $ by the formula
\be\label{ell}
  \wt y_{K+1} =  (1+\ell) y_{K+1}.
\ee
Using this $\ell$, we define the rescaled boundary conditions  $y^*_{j} = (1+ \ell) y_{j}$.
Now we will compare the measure $\wt\sigma_{\wt\by}$ with the rescaled conditional measure
$\sigma^*_{\by^*}$
which now has the same configurational interval $\wt J = J^*$.
 With the choice of $\ell$ in \eqref{ell}
and  the rigidity estimate \eqref{rig} , we can estimate $\ell$ by
$$
|\ell| \le N^\xi K^{-1} ,
$$
where we have also used  $\gamma_k \sim  k^{2/3}$.
From the rescaling identity \eqref{res} applied to an observable
$Q$ of special form,
we have
\be \label{eq:compyy}
\E^{\sigma^*_{\by^*}} O\Bigg( \Big( j^{1/3}(x_j -\al^*_j)\Big)_{j\in \Lambda}\Bigg)
  =    \E^{\sigma_{\by}} O\Bigg( \Big (  j^{1/3} (1+ \ell) \Big [ x_j -  \al_j \Big ] \Big)_{j\in \Lambda}\Bigg).
\ee
From the rigidity estimate \eqref{32the} (notice we need to change to the microscopic coordinates), we have
\be
\Big |     j^{1/3}  \ell (x_j - \al_j)
\Big | \le N^{2 \xi } K^{-1} \le N^{-\chi}.
\ee
since $2 \xi + \chi < \delta$, see \eqref{con1}.
We can use the smoothness of $O$ to remove the $(1+\ell)$
factor on the right hand side of \eqref{eq:compyy} by Taylor expansion
at a negligible error.
Using \eqref{eq:compy*}, we have proved that
\be
  \Bigg| \E^{\sigma_{\by}} O\Bigg( \Big( j^{1/3}(x_j -\al_j)\Big)_{j\in \Lambda}\Bigg)
  -   \E^{\wt \sigma_{\wt\by}} O\Bigg( \Big( j^{1/3}(x_j -\wt \al_j)\Big)_{j\in \Lambda}
\Bigg)\Bigg| \le N^{-\chi}
\ee
and this  proves  Theorem~\ref{thm:local}. \end{proof}

\subsection{Outline of the proof of Theorem~\ref{thm:cor}}\label{sec:int}

The basic idea to prove \eqref{eq:compy} is to introduce
a one-parameter family of interpolating measures between
any two  measures $\sigma_\by$  and $\wt \sigma_{\wt \by}$  with potentials $V_\by$ and $\wt V_{\wt \by}$
with fixed boundary conditions $\by$ and $\wt\by$
and possible two different external potentials $V$ and $\wt V$.
These measures are defined for any $0\le r\le 1$ by
\be\label{defom}
  \om= \om_{\by,\wt\by}^r \sim e^{-\beta N\cH_{\by,\wt\by}^r}, \qquad
   \cH_{\by,\wt\by}^r (\bx): =  \frac{2}{N}\sum_{i\in I}  \Theta \big( N^{-\xi} x_i\big)  +
 \sum_{i\in I} \frac{1}{2} V^r_{\by,\wt\by} (x_i) -\frac{1}{N}
   \sum_{i,j\in I\atop i<j} \log |x_j-x_i|,
\ee
with
\be\label{defVom}
   V^r_{\by,\wt\by}(x): = (1-r) V_\by (x)+r\wt V_{\wt\by}(x).
\ee
Notice that $\om_{\by,\wt\by}^{r=0} = \sigma_\by$ and  $\om_{\by,\wt\by}^{r=1} = \wt\sigma_{\wt \by}$.
Basic properties of the measure $\om$ will be established in Section~\ref{sec:inter}.
Now we outline our main steps to prove  \eqref{eq:compy}.

\begin{itemize}
\item [Step 1.]  {\it Interpolation.} For any observable $Q(\bx)$, we rewrite the
 difference of the expectations of $Q$
w.r.t. the two different local measures by
$$ \E^{\sigma_\by} Q(\bx) -  \E^{\wt\sigma_{\wt\by}} Q(\bx) = \int_0^1 \frac{\rd}{\rd r}  \E^{\om^r_{\by,\wt \by}} Q(\bx)\rd r
  = \beta \int_0^1  \langle Q; h_0\rangle_{\om^r} \rd r,
$$
with
\begin{align}\label{h0def}
    h_0(\bx): &  = N\sum_{j\in I} [V_\by(x_j) - \wt V_{\wt\by}(x_j)] \nonumber  \\
&    =   \sum_{j\in I}  \left   [ N \Big (V(x_j N^{-2/3}) - \wt V(x_j N^{-2/3}) \Big ) -  2 \sum_{k\not\in I}
   \Big (  \log |x_j-y_k| - \log |x_j- \wt y_k| \Big ) \right  ].
\end{align}

So the main goal is to show that for any $\om=\om^r_{\by,\wt\by}$ with good boundary conditions we have
\be\label{h0corr}
 |\langle  Q; h_0\rangle_{\om}| \le N^{-\chi}.
\ee
This will hold for a certain class of observables $Q$ that depend on a few coordinates
near the left edge. The class of observables we are interested in have the form
\be\label{O}
  Q(\bx): =  O\Bigg( \Big( j^{1/3}(x_j - j^{2/3})\Big)_{j\in \Lambda}\Bigg).
\ee

\item [Step 2.] {\it Random walk representation.}
For any smooth  observables    $F(\bx)$ and  $Q(\bx)$ and   any time $T > 0$
 we have the following representation formula for the time dependent correlation function
(see \eqref{repp} for the precise statement):
\be\label{RW}
\E^\om  Q(\bx) \, F(\bx)    - \E^\om   Q (\bx(0))  F(\bx (T))
= \frac{1}{2}\int_0^{T}   \rd \Si \; \E^\om
    \sum_{b=1}^K \pt_b Q(\bx(0)) \langle \nabla F(\bx(S)) , \bv^b(S, \bx(\cdot)) \rangle.
\ee
Here the  path $\bx(\cdot)$ is the solution of the reversible stochastic dynamics
with equilibrium measure  $\om$ (see  \eqref{SDE} later).
We  use the notation $\E^\om$ also for the expectation with respect to the path  measure
starting from the initial distribution $\om$
and $\langle \cdot, \cdot \rangle$ denotes the inner product in $\R^K$. Furthermore,
for any $b\in I$ and for any fixed path $\bx(\cdot)$,
 the vector  $\bv^b(t)= \bv^b(t, \bx(\cdot))\in \R^K$ is the solution to the
equation
$$
   \partial_t \bv^{b} (t) =  - \cA(t) \bv^b (t), \quad t\ge 0, \qquad v_j^b(0)=\delta_{bj}.
$$
The matrix $\cA(t)$  depends on time through the path $\bx(t)$, i.e., it is
 of the form $\cA(t)=\wt \cA(\bx(t))$.
It  will be defined in \eqref{A=B+W} and
it is related to the Hessian of the Hamiltonian
$\cH_{\by,\wt\by}^r$  of the measure $\om$.
Using rigidity estimates on the path $\bx(\cdot)$, we will show that
with very high probability the matrix elements of $\cA(t)$
 satisfy the time-independent lower bound
\be\label{AAbound}
   \cA(t)_{ij}  \ge \frac { 1  } { (i^{2/3}-j^{2/3})^2}
+   \delta_{ij} \frac{ K^{2/3} }{ K^{2/3} - j^{2/3}  }  ,
\ee
up to irrelevant factors (see \eqref{Be}, \eqref{We}).

 We apply the random walk representation \eqref{RW} for $T\sim K^{1/3}$ and $F=h_0$. This is sufficient
since the time to equilibrium for the $\bx(t)$ process is of order $K^{1/3}$, which
will be guaranteed by convexity properties of the Hamiltonian of the measure $\om$
(Lemma~\ref{omegatime}).

\item[Step 3.]  If the coefficient matrix $\cA(t)$ satisfies  \eqref{AAbound}, then
the semigroup associated with  the equation
\be
\pt_t\bu(t) = -\cA(t)\bu(t)
\label{heat}
\ee
 has
good $L^p\to L^q$  decay estimates (Proposition~\ref{prop:disp}) that follow from energy method
and a new Sobolev inequality (Proposition~\ref{prop:sob}).
Rigidity estimates w.r.t. $\omega$  (Lemma~\ref{lm:inter}) will ensure that
the bound \eqref{AAbound} holds with very high probability. The $L^p\to L^q$  decay estimates
together with the bound
$$
   |\partial_j h_0(\bx)| \lesssim  \frac {K^{1/3}}{K+1-j}, \qquad j\in I,
$$
that also follows from rigidity,
will allow us to reduce the upper limit in the time integration in \eqref{RW}
 from $T\sim K^{1/3}$ to $\wt T\sim K^{1/6}$
in \eqref{RW}. The necessary rigidity estimate w.r.t. $\omega$ is obtained by
interpolating between the rigidity estimates for $\sigma_\by$ and $\wt\sigma_{\wt\by}$.

\item[Step 4.] Finally, we also have a time dependent version of the
 $L^\alpha\to L^\infty, \alpha > 1,$ decay estimate that follows from
a different Sobolev inequality (see Theorem~\ref{thm:2ndSob}). More precisely, in Lemma~\ref{cor:disp2}
we will show that if the matrix elements $\cA_{ij}(t)$ satisfy \eqref{AAbound}, then
for the $M$-th coordinate of the solution to \eqref{heat} we have for any $\alpha > 1$
$$
 \int_0^t |u_M(s)|^\alpha \rd s \le C_\alpha
   M^{-2/3}(t+1) \| \bu(0)\|_\alpha^\alpha, \qquad M\in I, \quad t>0,
$$
(up to irrelevant factors). We will apply this bound with $\alpha = 1 + \e$ to control the remaining
time integration from $0$ to $\wt T$ in \eqref{RW}.

\end{itemize}

\newcommand{\cB}{{\mathcal B}}

\section{Properties of the interpolating measure}\label{sec:inter}

In this section we establish the necessary apriori results for
$\om$, defined in \eqref{defom}. We start with its speed to equilibrium
from  a convexity bound on the Hessian.
The measure $\om$  defines a Dirichlet form  $D^{\om}$ and
its generator $\cL^\om$ in the usual way:
$$
  - \langle f, \cL^{\om} f\rangle_\om = -\int f\cL^\om f\rd\om
=  D^{\om} (f) = \frac{1}{2}\int |\nabla f|^2\rd\om,
$$
where
$$
\cL^\om = \frac{1}{2}\sum_{i\in I}\Big[ \pt_i^2 +\beta
\Big\{    - 4 N^{-\xi}  \Theta' \big( N^{-\xi} x_i\big)  
  - N( V^r_{\by,\wt\by})'(x_i) + \sum_{j\ne i}\frac{1}{x_i-x_j}\Big\}
\pt_i \Big].
$$
Note that in the context of studying the dynamics near the edge
in the microscopic coordinates, the natural Dirichlet form is defined without
the $1/N$ prefactor in contrast to \eqref{Ddef} and \eqref{L}, where
the scaling was dictated by the bulk.

Finally, let $\bx(t)$ denote the corresponding stochastic process
(local Dyson Brownian motion), given by
\be
  \rd x_i = \rd B_i + \beta \Big[     -  2 N^{1-\xi}  \Theta' \big( N^{-\xi} x_i\big)
    - \frac{N}{2}  (V_{\by,\wt\by}^{r})'   ( x_i ) +
 \frac{1}{2}\sum_{j\ne i} \frac{1}{x_i-x_j}\Big] \rd t, \qquad i\in I,
\label{SDE}
\ee
where $(B_1, \dots, B_{K})$ is a family of independent standard Brownian motions.
 With a slight abuse of notations,  when we talk about the process,
we will use $\P^\om$ and $\E^\om$ to denote the probability
and expectation w.r.t. this dynamics with initial data $\om$, i.e., in equilibrium.
This dynamical point of view gives rise to a representation for
the correlation functions  in terms random walks in random environment.
Note that $\beta\ge 1$ is needed for the well-posedness of \eqref{SDE}.
From the Hessian bound \eqref{sihess} and the Bakry-\'Emery criterion
  we have proved the following result:
\begin{lemma}\label{omegatime} Let $\xi$ be any
fixed positive constant and assume $  K$ satisfies  \eqref{NK}.
Let $\by, \wt\by \in \cR=\cR_{K}(\xi) $, $r\in [0,1]$  and set $\om = \om_{\by,\wt\by}^r$.
Then the measure $\om=\om_{\by, \wt\by}^r$
satisfies the     logarithmic Sobolev
inequality
$$
S(g  \om    |  \om )  \le  CK^{1/3}   D^{ \om } ( \sqrt {g} )
$$
and the time to equilibrium for the dynamics $\cL^\om$ is at most of order $K^{1/3}$. \qed
\end{lemma}

Next, we formulate the rigidity and level repulsion bounds for $\om$.

\begin{lemma}[Rigidity and level repulsion for $\om$]   \label{lm:inter}
Let $\xi$ be any
fixed positive constant and assume $  K$ satisfies  \eqref{NK}.
Let $\by, \wt\by \in \cR=\cR_{K}(\xi) $, $r\in [0,1]$  and set $\om = \om_{\by,\wt\by}^r$.
Recall also the definition of $\al_i$ from \eqref{aldefuj}. Then the following holds:

(i) [Rigidity] There is a constant $c> 0$ such that
\be\label{omrig}
  \P^\om\big( |x_i-\al_i| \ge N^{C_2\xi}  i^{-1/3}u \big)\le e^{-cu^2},
\qquad i\in I, \quad u>0.
\ee

(ii) [Level repulsion] For any $s>0$ we have
\begin{align}\label{omk5n}
\P^{ \om}[ y_{K+1} -x_{K} \le s K^{-1/3}  ] & \le
  C \left (  K^2s \right ) ^{\beta + 1}, \\ \label{omk5nlow}
\P^{ \om}[ y_{K+1}-x_K \le s K^{-1/3}  ] & \le
  C \left ( N^{C\xi} s \right ) ^{\beta + 1} + e^{-N^c}.
\end{align}

(iii) We also have
\begin{align}\label{omlog}
 \E^\om |\log (y_{K+1}-x_K)|  & \le CN^{C\xi},\\
\label{omresolv}
   \E^\om \frac{1}{|y_{K+1}-x_K|^q} & \le C_qN^{C\xi} K^{q/3}, \qquad q <\beta+1.
\end{align}
\end{lemma}

The key to translate the  rigidity   estimate of the measures $\sigma_\by$ and
$\sigma_{\wt \by}$ to  the measure $\om=\om_{\by, {\wt \by}}^r$
is the following lemma.

\begin{lemma}\label{lm:91}
Let $K$ satisfy \eqref{NK} and
$\by, \wt \by  \in \cR_{K}(\xi) $. Consider
the local equilibrium measure $\sigma_\by$ defined in \eqref{Vyext}
  and  assume that \eqref{Exm} is satisfied.
Let  $ \om_{\by, {\wt \by}}^r$ be the measure   defined in \eqref{defom}.
Recall that $\al_k$ denote the equidistant
points in $J$, see \eqref{aldefuj}.
Then there exists a constant $C$, independent of  $\xi$, such that
\be\label{Exm3}
    \E^{ \om_{\by, {\wt \by}}^r }   \left | x_j -  \alpha_j \right |
   \le C  N^{ C  \xi}.
\ee
\end{lemma}

\begin{proof}[Proof of Lemma~\ref{lm:91}]
We first recall the following estimate on the entropy from Lemma 6.9 of \cite{EYsinglegap}.

\begin{lemma}\label{ent}
Suppose $\mu_1$ is a probability measure and $\om = Z^{-1} e^g d \mu_1$
for some function  $g \in L^1(\rd\mu_1)$ with $e^g\in L^1(\rd\mu_1)$  and normalization $Z$.
Then we can bound the entropy by
$$
S:= S(\om| \mu_1) = \E^\om g - \log \E^{\mu_1} e^g  \le \E^\om g - \E^{\mu_1} g.
$$
Consider two probability  measures $\rd \mu_i = Z_i^{-1} e^{- H_i} \rd \bx $, $i=1,2$.
Denote by $g$ the function
$$
g = r (  H_1 - H_2), \quad 0 < r < 1,
$$
and set $\om= Z^{-1} e^g d \mu_1$ as above. Then we can bound the entropy by
$$
\min (S(\om| \mu_1), S(\om| \mu_2)) \le\Big [ \E^{\mu_2} - \E^{\mu_1}   \Big ]  (H_1-H_2).
$$
\end{lemma}

We now apply this lemma with  $\mu_2 = \wt \sigma_{\wt \by}$ and  $\mu_1 = \sigma_{ \by}$ to prove that
\be\label{85}
\min[ S(\om_{\by, {\wt \by}}^r|\sigma_\by), S(\om_{\by, {\wt \by}}^r|\wt \sigma_{\wt \by}) ]  \le
  N^{ C \xi}.
\ee
To see this, by definition  of $g$  and the rigidity estimate \eqref{rig},  we have
\begin{align}\label{842}
\E^{\mu_2} g- \E^{\mu_1} g  &  =  \frac{r}{2}\Big [ \E^{\mu_2} -  \E^{\mu_1} \Big ] \sum_{i \in I}
\Big [ V_\by(x_i) - \wt V_{\wt\by}(x_i)
 \Big ]  \non \\
 &  =  \frac{r}{2}\Big [ \E^{\mu_2} -  \E^{\mu_1} \Big ] \sum_{i \in I}
\Big [ V_\by(x_i) - \wt V_{\wt\by}(x_i)
 - \big (V_\by(\al_i) - \wt V_{\wt\by}(\al_i)   \big ) \Big ]  \non \\
& =   \frac{r}{2}\Big [ \E^{\mu_2} -  \E^{\mu_1} \Big ] \sum_{i \in I} \int_0^1\rd s\Big [ V_\by'(s\al_i+(1-s)x_i ) - \wt V_{\wt\by}'(s\al_i+(1-s)x_i)
 \Big ] (x_i-\al_i)  \non \\
& = \Big [ \E^{\mu_2} +  \E^{\mu_1} \Big ] O\Big(  \sum_{i \in I}
 \sup_{s\in [0,1]}\frac{K^\xi}{ |s\al_i+(1-s)x_i - y_{K+1}| }|x_i-\al_i|\Big) \le  N^{C \xi}.
\end{align}
In the last step we used the rigidity \eqref{32the} to see
that with a very high $\mu_1$- or $\mu_2$-probability the numbers
 $s\al_i+(1-s)x_i\sim \al_i$
are equidistant up to an additive error
$K^\xi$ if $i$ is  away from the boundary, i.e., $ i\le  K-K^{C\xi}$,
see \eqref{32the}.  For indices near the boundary, $ i \ge K-K^{C\xi}$,  we used
$|s\al_i+(1-s)x_i|\ge c\min\{1, |x_{K}-y_{K+1}|\}$ and the rigidity $|x_i-\alpha_i|\le N^{C\xi}K^{-1/3}$.
The bound \eqref{omresolv}
guarantees that the short distance singularity  $|x_{K}-y_{K+1}|^{-1}$
 has an  $\E^{\mu_{1,2}}$ expectation that is bounded by $CN^{C\xi}K^{1/3}$, which gives \eqref{842}.

We now assume that \eqref{85} holds with the
 choice of $S(\om_{\by, {\wt \by}}^r|\sigma_\by)$ for simplicity
of notation.
By the entropy inequality, we have
$$
\E^{\om_{\by, {\wt \by}}^r} |x_i - \gamma_i| \le N^{\xi + \e}  \log   \E^{\sigma_\by}    e^{  N^{-\xi - \e}   | x_i - \gamma_i | } + N^{ C \xi } N^{\xi + \e}
\le N^{ C \xi }.
$$
This proves Lemma \ref{lm:91}.
\end{proof}

\begin{proof}[Proof of Lemma~\ref{lm:inter}]
Given \eqref{Exm3},
 the proof of \eqref{omrig} follows from  the argument in the proof of
Theorem \ref{thm:condRig}.
Once the rigidity bound \eqref{omrig} is proved,
we can follow the proof of Theorem~\ref{lr2} to obtain the
repulsion estimates \eqref{omk5n}-\eqref{omk5nlow}.
 The only modification is that we use
the potential $V_{\by, \wt\by}^r$ of the measure $\om=\om_{\by,\wt\by}^r$
(see \eqref{defVom}) instead of $V_\by$. The analogue of $V_\by^*$
(see \eqref{V*defn}) can be directly defined for  $V_{\by, \wt\by}^r$  as
\be\label{Vyy*}
   [V_{\by, \wt\by}^r]^*(x) = (1-r) V_\by^*(x) + r \wt V_{\wt\by}^*(x).
\ee
Formula \eqref{ch21n} will be slightly modified, e.g.
the factor $(y_-+(1-\varphi)(w_j-y_-)-y_k)^\beta$
will be replaced with $(y_-+(1-\varphi)(w_j-y_-)-y_k)^{(1-r)\beta}
(y_-+(1-\varphi)(w_j-y_-)-\wt y_k)^{r\beta}$, but it does not change the
estimates. Similarly, the necessary bound \eqref{V*dern} for
the potential $[V_{\by, \wt\by}^r]^*$
easily follows from \eqref{Vyy*} and the same bounds on $V_\by^*$
and $V_{\wt\by}^*$. Finally, \eqref{omlog} and \eqref{omresolv} are trivial
consequences of \eqref{omk5n} and \eqref{omk5nlow}.
\end{proof}

\section{Random walk representation for the correlation function}

  The first step to  prove \eqref{h0corr} is to  use the random walk
 representation formula from  Proposition  7.1 of \cite{EYsinglegap}
which we restate in Proposition~\ref{prop:repp} below.
 This formula in a lattice  setting  was given  in  Proposition 2.2 of
 \cite{DeuGiaIof2000} (see also Proposition 3.1 in \cite{GOS}). The random walk representation
already appeared  in the earlier paper of Naddaf and Spencer \cite{NS},  which  was a
probabilistic formulation  of the idea
 of Helffer and Sj\"ostrand \cite{HS}.

Fix $\Si>0$ and  $\bx\in  J^K=J_\by^I$.
Let $\bx(s)$ be the solution to \eqref{SDE}
with initial condition $\bx(0)=\bx$. Let $\E^\bx$ denote the
expectation with respect to this path measure.
With a slight abuse of notations, we will use
 $\P^\om$ and $\E^\om$ to denote the probability and expectation
with respect to the path measure of the solution to \eqref{SDE}
with initial condition $\bx$ distributed by $\om$.

For any fixed path
$\bx(\cdot):=\{\bx(s)\; : \; s\in [0,\Si]\}$ we
define the following operator ($K\times K$ matrix) acting on
$K$-vectors $\bu\in \R^K$ indexed by the set $I$;
\be\label{A=B+W}
  \cA(s) = \wt\cA (\bx(s)), \qquad  \wt\cA =\wt\cB +\wt \cW,
\ee
with actions
$$
  [\wt \cB(\bx)\bu]_i : = \frac{1}{2} \sum_{j\in I}
\frac{1}{(x_i-x_j)^2}(u_i-u_j), \qquad
    [\wt \cW(\bx)\bu]_i = \cW_i u_i  \qquad i\in I,
$$
where we defined
\be\label{def:Wtilde}
  \wt \cW_i(\bx) =   2N^{1-2\xi} \Theta'' (N^{-\xi} x_i)
 + \frac { N^{-1/3}}  2   \big[(1-r)V''(x_iN^{-2/3})+
r\wt V''(x_iN^{-2/3})\big]+  \frac{1}{2}  \sum_{k\not\in I}
\Big[ \frac{1-r}{(y_k-x_i)^2} + \frac{r}{(\wt y_k-x_i)^2} \Big].
\ee
(Notice that $W_i(\bx)$ depends only on $x_i$).

\begin{proposition}\label{prop:repp} For any smooth functions  $F: J^K \to \R$
and  $Q: J^K \to \R$  and   any time $T > 0$
 we have
\begin{align}\label{repp}
\E^\om Q(\bx)\, F(\bx)   - \E^\om  Q (\bx(0))  F(\bx (T))
= \frac{1}{2}\int_0^{T}   \rd \Si \int\omega (\rd \bx)
    \sum_{a,b =1}^K \pt_b Q(\bx)
    \E^\bx \pt_a  F (\bx(\Si))   v^b_a(S, \bx(\cdot)).
\end{align}
 Here for any $\Si>0$ and for any path $\{ \bx(s)\in   J^K \; : \; s\in [0,\Si]\}$, we
define  $\bv^b(t)= \bv^b(t, \bx(\cdot))$ as the solution to the equation
\be\label{vb}
   \partial_t \bv^{b} (t) =  - \cA(t) \bv^b (t), \quad t\in [0, S], \qquad v_a^b(0)=\delta_{ba}.
 \ee
The dependence of $\bv^b$ on the path $\bx(\cdot)$ is present via the dependence
$\cA(t) =  \wt \cA( \bx (t))$. In other words, $v^b_a(t)$ is the fundamental solution
of the heat semigroup $\pt_s+ \cA(s)$.
\end{proposition}

\section{Proof of Theorem~\ref{thm:cor}} \label{sec:prthm}

{F}rom now on we assume  the conditions of Theorem~\ref{thm:cor}.
In particular we are given some $\xi>0$  and we assume that the
boundary conditions satisfy
 $\by, \wt \by  \in \cR^\#(\xi)$
and \eqref{Exm}.

\subsection{First time cutoff}

We now start to estimate the correlation function in \eqref{h0corr}.
We first  apply the formula  \eqref{repp} with $F$ replaced by $h_0$ defined in \eqref{h0def}
so that
$$
 \E^\om Q(\bx) h_0 (\bx)    - \E^\om  Q(\bx(0))  h_0(\bx (T))
  =  \frac{1}{2}\int_0^{T} \rd\Si \int\om(\rd\bx)
\sum_{a,b=1}^K \pt_b Q(\bx)
    \E^\bx \partial_a  h_0 (\bx(\Si))   v^b_a(S, \bx(\cdot)).
$$
 We collect  information on $h_0$ in the following lemma:
\begin{lemma}\label{lm:init} Let $K$ satisfy
\be\label{NKn}
N^\delta \le K \le N^{2/5-\delta}
\ee
for some small $\delta>0$ and let $\by, \wt\by\in  \cR^\#(\xi)$.
Then for any $\kappa < \beta+1$ we have
\be
  \E^\om |h_0(\bx)|^\kappa \le C_\kappa K^2N.
\label{Eh0}
\ee
Furthermore, if $\bx$ satisfies
\be\label{nabhcond}
   \max_{ j\in I } j^{1/3} |x_j-\alpha_j| \le N^{C_3\xi},
\ee
then
\be\label{nabh}
   |\partial_j h_0(\bx)| \le  \frac {CN^{C\xi} K^{1/3}}{K+1-j}, \qquad j\in I.
\ee
In particular, we have
 the $L^1$-bound
\be\label{h0l1}
   \sum_j  |\partial_j h_0(\bx)| \le  CN^{C\xi} K^{1/3}.
\ee
\end{lemma}

\begin{proof}  The bound \eqref{Eh0} follows from  \eqref{omk5n}, while
\eqref{nabh}  will be
proven in Appendix~\ref{app:V*}.  \end{proof}

 Since the time to equilibrium of the $\cL^\om$ dynamics is of order $K^{1/3}$
(see Lemma~\ref{omegatime}), by
choosing
$$
T: = C K^{1/3}  \log N
$$
with a large constant $C$,
we have
\be\label{ht}
 \big | \E^\om  Q(\bx(0))  h_0(\bx (T)) - \E^\om  Q(\bx)\;   \E^\om  h_0(\bx)
  \big | \le N^{-C}.
\ee
 In proving this relation, we use a cutoff argument. Although $h_0$ is
singular and it is not in $L^2(\om)$, we can write
$h_0 = h_< + h_>$, $h_<(x): = h_0(x)\mathds{1}(h_0(x)\le N^C)$.
By \eqref{Eh0} the probability $\P^\om( h_0 \ge N^C)$ and
hence  the contribution of $h_>$ to \eqref{ht} are negligible.
The function $h_<$ is in $L^2(\rd\om)$, so we can use the spectral gap
of order $CK^{1/3}$ (Lemma~\ref{omegatime}) to show that the
contribution of the $h_<$ part to \eqref{ht} is also negligible.
We can thus represent  the correlation function as
\be\label{repp1}
  \langle Q ; h_0  \rangle_\om  = \frac{1}{2} \int_0^{T} \rd\Si \int\om(\rd\bx)
\sum_{a,b=1}^K \pt_b Q(\bx)
    \E^\bx \partial_a  h_0 (\bx(\Si))   v^b_a(S, \bx(\cdot))
 + O(N^{-C}) .
\ee

\subsection{Set of good paths}

We have a good control on the solution to \eqref{vb} if the coordinates of the trajectory $\bx(\cdot )$
remain close to the classical locations $(\alpha_1, \dots, \alpha_K)$.
Setting a constant $C_3>C_2$  ($C_2$ is the constant in \eqref{omrig}),
for any $T$ we thus define the
set of ``good'' path as:
\be\label{K1}
\cG_T :
=  \Big\{  \sup_{0 \le s \le T} \;\max_{ j\in I }  j^{1/3} |x_j(s)-\alpha_j| \le N^{C_3\xi}   \Big\},
 \ee
where $\alpha_j$ is given by \eqref{aldefuj}.

\begin{lemma}\label{lm:GQ} Assume that the rigidity estimate \eqref{omrig}
holds for the measure $\om$.
For the cutoff time $T= CK^{1/3}\log N$,
there exists a positive constant $\theta$,  depending on $\xi$,
  such that
\be\label{K11}
\P^\om ( \cG^c_T ) \le   e^{ - N^{\theta }}.
\ee
\end{lemma}

\begin{proof}
 We first recall  the following result of Kipnis-Varadhan \cite{KipVar1986}:

\begin{lemma}
For any process with a reversible measure $\om$ and Dirichlet form $D^\om(f) =  \frac{1}{2}
 \int |\nabla f|^2\rd\om$, we have
\be\label{913}
\P^\om ( \sup_{0 \le s \le T} |f(\bx(s))| \ge \ell) \le \frac 1 \ell \sqrt { \| f \|_2^2 + T D^\om(f) }.
\ee
\end{lemma}

To apply this lemma, let $f(\bx) = g(x_j)$ with
$$
g(x) =   e^{ N^{- C \xi } (x_j - \alpha_j) j^{1/3}}.
$$
From the  rigidity estimate \eqref{omrig}
$$
 \| f \|_2^2 + T D^\om(f)  \le  \Big[
1+ T\big(N^{-C\xi}j^{1/3}\big)^2\Big] \| f \|_2^2  \le CK(\log N) \int_\R    e^{   |u| } e^{ - c u^2} \rd u
 \le C K^{2}  .
$$
From \eqref{913}, we have for any $c > 0$
$$
\P ( \sup_{0 \le s \le T}  N^{- C \xi } (x_j(s)- \alpha_j) j^{1/3}   \ge N^c    ) \le
\P ( \sup_{0 \le s \le T}  |g(x_j (s))|   \ge e^{N^c}     ) \le C K^2  e^{- N^c}.
$$
Similarly, we can prove
$$
\P ( \sup_{0 \le s \le T}  N^{- C \xi } (\alpha_j - x_j) j^{1/3}   \ge N^c    ) \le  C  e^{- N^c}.
$$
This proves Lemma~\ref{lm:GQ}.
\end{proof}

\subsection{Restriction to the set $\cG_T$}

 Now we show that
the expectation \eqref{repp1} can be restricted to the good set $\cG:=\cG_T$
with a small error.  With a slight abuse of notations
we use $\cG$ also to denote the characteristic function of the set $\cG$.
For a fixed $\Si$  and for a fixed $b\in I$ we can estimate the contribution of
the $\cG^c$ by
$$
\frac{1}{2}\int  \big |  \pt_b Q(\bx)\big | \sum_{a=1}^K
\E^\bx  \Big[  \cG^c  |\pt_a h_0(\bx(S))|  v_a^b (\Si, \bx(\cdot) ) \Big] \omega (\rd \bx).
$$
Since $\cA\ge 0$ as a $K\times K$ matrix,   the  equation \eqref{vb}
 is contraction in $L^2$.
 Clearly $\cA$ is a contraction in $L^1$ as well, hence it is a contraction in any $L^q$,
$1\le q\le 2$, by interpolation.
By the H\"older inequality and the  $L^q$-contraction for some $1<q<2$,
we have $\sum_a | v_a^b (\Si, \bx(\cdot) )|^q \le \sum_a | v_a^b (0, \bx(\cdot) )|^q=1$, so we get
\begin{align}\non
   \E^\om    \cG^c  \Big| \sum_{a=1}^K |\pt_a h_0(\bx(S))|  v_a^b (\Si, \bx(\cdot) )\Big|
 & \le \big[   \E^\om    \cG^c\big]^{q/(q-1)}
   \Big[\E^\om   \Big| \sum_{a=1}^K |\pt_a h_0(\bx(S))|  v_a^b (\Si, \bx(\cdot) )\Big|^q \Big]^{1/q} \\ \non
  &    \le   \big[   \P^\om    \cG^c\big]^{q/(q-1)}
   \Big[\E^\om  \sum_{a=1}^K  |   \pt_a  h_0  ( \bx( \Si) )|^q \Big]^{1/q} \\ \non
& \le C e^{- c N^{\theta}} K\cdot \max_a \Big[ \E^\om |\pt_a h_0(\bx)|^q\Big]^{1/q} \le  e^{- c N^{\theta_4}}
\end{align}
with some $\theta_4>0$.  Here  we used \eqref{K11} for the first factor.  In the second
factor, after the invariance of the dynamics, we used  the explicit form of $h_0$ \eqref{h0def} and
 the level repulsion bound \eqref{omresolv}:
$$
  \E^\om |\pt_a h_0|^q \le CN^{q/3} + C\max_a \E^\om \Big[ \sum_{k\not\in I}\frac{1}{|x_a-y_k|^q}\Big]
  \le CN.
$$
Therefore,  from \eqref{repp1} we conclude that
\be\label{finrep}
  \langle Q ; h_0  \rangle_\om  =  \frac{1}{2}\int_0^{T} \rd\Si \int\om(\rd\bx)
\sum_{a,b=1}^K \pt_b Q(\bx)
    \E^\bx \Big[ \cG \;  \partial_a  h_0 (\bx(\Si))  \; v^b_a(S, \bx(\cdot))\Big]
 + O(N^{-C}) .
\ee

In the next step we will reduce the upper limit of the time integration from
$T\sim K^{1/3}$ to $\wt T\sim K^{1/6}$. This reduction uses effective $L^p\to L^q$
bounds on the solution to \eqref{vb} that we will obtain with
energy method and Nash-type argument.

\subsection{Energy method and the evolution equation on the good set $\cG$} \label{sec:energ}

In order to study the evolution equation \eqref{vb} with $\bx(\cdot)$ in the good set $\cG$,
we consider the following general
 evolution equation
\be\label{ve}
\partial_s  \bu (s) =  - \cA(s)  \bu (s), \qquad \bu(s)\in \R^I=\R^K,
\qquad \bu(0)=\bu_0.
\ee
Here $\cA$ and $\cB$ are  time dependent matrices of the form
\be\label{ABW}
\cA(s)=\cB(s)+ \cW(s), \qquad \mbox{with}\quad
   [\cB(s)\bu]_i = \sum_j B_{ij}(s) (u_i-u_j), \qquad [\cW(s)\bu]_i = W_i(s)u_i.
\ee
For $x_i, x_j$ satisfying  the rigidity bound defined in the good path \eqref{K1}
 we have
\be\label{Be}
 \wt  \cB(\bx)_{ij}: =
\frac{1}{(x_i-x_j)^2} \ge \frac { N^{-C \xi}  } { (i^{2/3}-j^{2/3})^2}
\ee
for some constant $C$. Similarly, for $\by \in \cR$ and $x_i$ satisfying
  the rigidity bound defined in the good path \eqref{K1}
we have
\be\label{We}
 \wt \cW_i(\bx)\ge
 \sum_{ \ell  > K } \frac { 1 } { (x_i-y_\ell )^2}  \ge  \frac{ K^{1/3} N^{-C\xi} }{ d_i  } , \qquad
  d_j : = (K+1)^{2/3} - j^{2/3}, \qquad j\in I,
\ee
where we have used the definition of $\wt W$ in \eqref{def:Wtilde} and $\Theta'' \ge   0$.

Denote  the  $L^p$-norm  of a vector $\bu = \{ u_j \; : \; j\in I\}$ by
$$
\| \bu \|_p  =  \Big(\sum_{j\in I}  |u_j|^p \Big)^{1/p}.
$$
We have the following decay estimate.

\begin{proposition}\label{prop:disp}
Let $\cA$ be given in \eqref{ABW} and
consider the evolution equation  \eqref{ve}.
Fix $\Si>0$.  Suppose that for some
constant $b$ the coefficients of $\cA$ satisfy
\be\label{B}
  B_{jk}(s) \ge   \frac b   { (j^{2/3}-k^{2/3})^2}, \quad 0 \le s \le \Si, \quad j\ne k \in I,
\ee
and
\be\label{W1}
 W_j (s)  \ge  \frac{b K^{1/3}}{ d_j } ,   \qquad d_j := (K+1)^{2/3}-j^{2/3}, \quad j\in I,
\quad 0 \le s \le \Si .
\ee
Then  for any $1\le p\le q\le \infty$ and for any small $\eta>0$  we have the decay estimate
\be\label{decay}
\| \bu(s) \|_q  \le C(p,q,\eta)\Big[  ( K^{-\frac{2}{3}\eta}   sb)^{-(  \frac{3}{p}  - \frac 3 q)}\Big]^{1-6\eta}
    \| \bu(0) \|_{ p }, \qquad 0<s\le\Si.
\ee
\end{proposition}

\begin{proof}
We consider only the case $b = 1$, the
 general case follows from scaling.  We  follow  the idea of Nash and
start from the $L^2$-identity
$$
\partial_s  \| \bu(s) \|_2^2 =   - 2 \fa(s)[\bu(s), \bu(s)],
$$
where
$\fa(s)[ \bu, \bv] : = \sum_i u_i [\cA(s) \bv]_i$ is the quadratic form of $\cA(s)$.
For each $s$ we can extend $\bu(s):I\to \bR^{K}$ to a function
 $\wt \bu(s): $  on $\ZZ_+$ by defining  $\wt u_j(s) = u_j(s)$ for $j\le K$
and $\wt u_j(s) = 0$ for $j > K$.
 Dropping the time argument, we have, by  the estimates \eqref{B} and  \eqref{W1} with $b=1$,
\be\label{1022}
 2 \fa [\bu, \bu] \ge c
\sum_{i, j \in \ZZ_+} \frac { (\wt u_i- \wt u_j)^2 } { (i^{2/3}-j^{2/3})^{2}}
\ge  K^{-\frac{2}{3}\eta}
\sum_{i, j \in \ZZ_+} \frac { (\wt u_i- \wt u_j)^2 } { |i^{2/3}-j^{2/3}|^{2
-\eta }}
 \ge c_\eta K^{-\frac{2}{3}\eta} \|\wt\bu\|_p^2=c_\eta K^{-\frac{2}{3}\eta} \|\bu\|_p^2 , \quad p:=\frac{3}{1+\eta},
\ee
with some positive constant $c_\eta$.  In the first inequality, to 
estimate the $W$ term, we have used that
$$
\sum_{i > K } \frac { 1 } { (i^{2/3}-j^{2/3})^2}  \le   \frac{C K^{1/3}}{ d_j }  \le C W_j,  \quad   j \le K ,
$$
to estimate the summation   in \eqref{1022} when one of the indices $i,j$ is
bigger than $K$.
In the second inequality we used that
$$ 
   |i^{2/3} - j^{2/3}|^\eta \le K^{\frac{2}{3}\eta}
$$
for any $i, j\le K$ which is the support of $\wt u$. 
In the third inequality 
we used the discrete version of the following  Sobolev type inequality
that  will be proved in Appendix~\ref{sec:appsob}.

 \begin{proposition}\label{prop:sob}  We will formulate our result both in the continuous and
in the discrete setting.
 
 \begin{enumerate}
\item Continuous version. For any small $\eta>0$ there exists $c_\eta>0$ such that
 for any real function $f$ defined on $\bR_+$, we have
\be\label{conti1}
\int_0^\infty \int_0^\infty\frac{(f(x)-f(y))^2}{|x^{2/3}-y^{2/3}|^{2-\eta}} \rd x \rd y
 \ge c_\eta\Big(\int_0^\infty |f(x)|^p\rd x\Big)^{2/p},
 \qquad p := \frac{3}{1+\eta}.
\ee

\item Discrete version.
For any small $\eta>0$ there exists $c_\eta>0$ such that
 for any sequence $\bu = (u_1, u_2, \dots )$ we have
\be\label{discr1}
  \sum_{i\ne j\in \bZ_+} \frac{(u_i-u_j)^2}{|i^{2/3}-j^{2/3}|^{2- \eta} }   \ge c_\eta
  \Big(\sum_{i\in \bZ_+} |u_i|^p \Big)^{2/p}
  =  c_\eta\|\bu\|_p^2.
\ee
\end{enumerate}
\end{proposition}

We now return to the proof of Proposition \ref{prop:disp}.   
Combining \eqref{nash}, \eqref{1022} with the simple H\"older estimate
$$
   \| \bu\|_p^2 \ge \|\bu\|_2^{\frac{8-4\eta}{3}} \|\bu\|_1^{-\frac{2-4\eta}{3}},
$$
we have
$$
   \partial_s \|\bu\|_2^2 \le -  c_\eta  K^{-\frac{2}{3}\eta} 
  \|\bu\|_2^{\frac{8-4\eta}{3}} \|\bu\|_1^{-\frac{2-4\eta}{3}},
$$
i.e.,
$$
   \partial_s \|\bu\|_2 \le
- c_\eta  K^{-\frac{2}{3}\eta}  \|\bu\|_2^{\frac{5-4\eta}{3}} \|\bu\|_1^{-\frac{2-4\eta}{3}},\ {\rm so}\
   -\frac{1}{\| \bu(t)\|_2^{\frac{2-4\eta}{3}}}  \le - c_\eta K^{-\frac{2}{3}\eta} 
 t \|\bu\|_1^{-\frac{2-4\eta}{3}}
$$
since $\|\bu\|_1$ is decreasing. Thus
$$
\| \bu(t)\|_2 \le \Big(\frac{1}{(c_\eta K^{-\frac{2}{3}\eta} t)^{3/2}}\Big)^{1-6\eta}  \|\bu_0\|_1,
$$
and by duality
$$
  \| \bu(t)\|_\infty \le \Big(\frac{1}{(c_\eta K^{-\frac{2}{3}\eta}  t)^3}\Big)^{1-6\eta}   \|\bu_0\|_1.
$$
Thus, after interpolation we have proved \eqref{decay}.
\end{proof}

Now we apply Proposition \ref{prop:disp} to our case.

\begin{corollary}\label{cor:disp} Fix $\Si\le T$
and set $\cA(s) = \wt \cA(\bx(s))$ as defined in \eqref{A=B+W}.
 On the set $\cG$,
the coefficients of $\cA(s)=\cB(s)+\cW(s)$
 satisfy \eqref{B} and \eqref{W1}
with the constant $b=cN^{-2\xi}$. Consequently, the solution to
$$
   \pt_t \bu(t)  = -\cA(t)\bu(t)
$$
satisfies
\be\label{decayb}
\| \bu(s) \|_q  \le C(p,q,\eta) \Big(\frac{ N^{2\xi   +\frac{2}{3}\eta  }}{s}\Big)^{(  \frac{3}{p}
  - \frac 3 q)(1-6\eta) }
    \| \bu(0) \|_{ p }, \qquad 0<s\le\Si, \quad 1\le p\le q\le \infty.
\ee
\end{corollary}

\begin{proof}
From the estimates on $B$ and $W$ proved in (\ref{Be}, \ref{We}),  we have  proved the  estimates on the kernel elements in
Lemma~\ref{prop:disp}   with $b= N^{-C \xi}$.
Thus \eqref{decayb} directly follows from \eqref{decay}.
\end{proof}

\subsection{Second time cutoff}

Now we specialize the observable $Q(\bx)$  to be of  the form
\eqref{O}.
Thus $Q$   depends only on  variables with indices in $\Lambda\subset \llbracket 1, K^\zeta\rrbracket$
and $|\Lambda|=m$ with  $m$ a finite fixed number. Its derivative is bounded by
\be\label{Obound}
|\partial_j Q(\bx)|=
\left \| j^{1/3}(\partial_j O)\Bigg( \Big( i^{1/3}(x_i -\al_i)\Big)_{i\in \Lambda}\Bigg) \right \|_\infty \le K^{\zeta/3} , \qquad j\in \Lambda.
\ee
With the help of Corollary \ref{cor:disp},  we can reduce the upper limit of the time integration
in \eqref{finrep}
from $T\sim K^{1/3}$ to $\wt T\sim K^{1/6}$.
More precisely, using the $L^1\to L^\infty$  bound of \eqref{decayb}  with the choice $\eta=\xi$,
 the integration from $\wt T$ to $T$ in \eqref{finrep} is bounded by
\begin{align}\label{TTcut}
\frac{1}{2} \int_{\wt T}^{T} \rd\Si \int\om(\rd\bx)
\sum_{a,b=1}^K \pt_b Q(\bx)
    & \E^\bx \Big[ \cG \;  \partial_a  h_0 (\bx(\Si))| \; v^b_a(S, \bx(\cdot))\Big] \nonumber \\
 & \le CN^\xi|\Lambda| K^{\zeta/3} \max_{b\le K^{\zeta}}
\int_{\wt T}^T  \sum_{a=1}^K
\E^\om   \Big[ \cG \;  |\partial_a  h_0 (\bx(\Si))|  \; v^b_a(S, \bx(\cdot))\Big] \rd\Si \nonumber \\
& \le CN^{C\xi} K^{(1+\zeta)/3} \max_{b}  \int_{\wt T}^T \| \bv^b(S, \bx(\cdot))\|_\infty \rd S \nonumber\\
& \le CN^{C\xi} K^{(1+\zeta)/3}   \int_{\wt T}^T \Si^{-3(1-6\eta)} \rd \Si \nonumber \\
& \le CN^{C\xi}  K^{(1+\zeta)/3} \wt T^{-2},
\end{align}
where we also used \eqref{Obound} and \eqref{nabh} together with the fact that,  on the set $\cG$,
 $\bx(S)$ satisfies \eqref{nabhcond}.

Choosing
\be\label{deftt}
\wt T = CN^{C_5\xi} K^{(1+   \zeta)/6}
\ee
with a sufficiently large constant $C_5$, we conclude from \eqref{finrep} and \eqref{TTcut} that
\begin{align}\label{finrepp}
\left |  \langle Q; h_0  \rangle_{\om }  \right |  \le CN^\xi K^\zeta  \max_{b\le K^\zeta}
\int_0^{\wt T}  \sum_{a=1}^K
\E^\om   \Big[ \cG \;  |\partial_a  h_0 (\bx(\Si))|  \; v^b_a(S, \bx(\cdot))\Big]\rd S
 +O(N^{-C\xi})
\end{align}
with the special choice of $Q$ from \eqref{O}.

\subsection{A space-time decay estimate and completion of the proof
of Theorem~\ref{thm:cor}}

Using \eqref{nabh},
 the first term in \eqref{finrepp} is estimated by
\be\label{10}
 CN^\xi K^\zeta \sum_a
\int_0^{\wt T}
\E^\om  \cG |\partial_a h_0(\bx(\Si))|  v_{a}^b(\Si, \bx(\cdot))\rd\Si
  \le
\sum_a \frac{N^{C\xi}K^{1/3}}{K+1-a} \int_0^{\wt T}
\E^\om  \big[ \cG  v_{a}^b(\Si)\big] \rd\Si.
\ee
The last term can be estimated using a new space-time decay estimate
  for the equation  \eqref{ve}. Roughly speaking,
the energy method asserts that the total dissipation is bounded by the initial $L^2$ norm.
 We will apply this idea to the vector $\{ v_i^{\al/2}\}$,  see \eqref{energy}, and combine
it with a new Sobolev inequality to obtain a 
a control on the time integral of
a weighted $L^\infty$ norm in terms of the  $L^\alpha$ norm (the weight comes from
the fact that the dissipative term is inhomogenous in space). 
More precisely, we have the following estimate.

\begin{lemma}\label{cor:disp2}
Consider $\cA(s) = \wt \cA(\bx(s))$ as defined in \eqref{A=B+W}.
Suppose that the coefficients $\cA(s)=\cB(s)+\cW(s)$
 satisfy \eqref{B} and \eqref{W1}
with a constant $b$. Then for any exponent $\alpha>1$
there is a constant $C_\al$ such that
the solution to
$$
   \pt_t \bv(t)  = -\cA(t)\bv(t)
$$
satisfies, for any integer $1\le M\le K$ and for any positive time $t>0$,
\be\label{linfb}
 \int_0^t |v_M(s)|^\alpha\rd s \le
   C_\alpha M^{-2/3} C^{\sqrt{\log M}}b^{-1}(t+1) \| \bv(0)\|_\alpha^\alpha.
\ee
\end{lemma}

\begin{proof} With some positive constant $c_\al>0$ we have the following estimate for the
solution $\bv=\bv(t)$:
\begin{align}
   \pt_t \| \bv\|_\al^\al = \al\sum_i |v_i|^{\al-1}(\sgn v_i) \pt_t v_i
 & \le - \al\sum_{i,j} |v_i|^{\al-1}(\sgn v_i) \cB_{ij}[ v_i-v_j] \nonumber \\
 & = -\frac{\al}{2}\sum_{i,j}\big[ |v_i|^{\al-1}(\sgn v_i) - |v_j|^{\al-1}(\sgn v_j)\big]
 \cB_{ij}[ v_i-v_j] , \nonumber \\
& \le -c_\al\sum_{i,j} \cB_{ij}\big[ |v_i|^{\al/2}-|v_j|^{\al/2}\big]^2, \label{nash}
\end{align}
where we dropped the potential term $\al\sum_i |v_i|^{\al-1}(\sgn v_i) W_iv_i\ge 0$
and used the symmetry of $\cB_{ij}$ in the first step. In the second step we used
$\cB_{ij}\ge 0$ and
the straighforward calculus inequality
$$
   \big[ |x|^{\al-1}\sgn(x) -|y|^{\al-1}\sgn(y)\big]  (x-y) \ge
 c_\al' \big[ |x|^{\al/2}-|y|^{\al/2}\big]^2, \qquad x,y \in \R,
$$
with some  $c_\al'>0$.
Integrating \eqref{nash} from 0 to any $t>0$ we thus have
\be\label{energy}
  \int_0^t \sum_{i,j} \cB_{ij}(s)\big[ |v_i(s)|^{\al/2}-|v_j(s)|^{\al/2}\big]^2 \le
  \| \bv(0)\|_{\al}^\al - \|\bv(t)\|_\al^\al \le  \| \bv(0)\|_{\al}^\al.
\ee
Using  the lower bound on the coefficients of $\cB(s)$,
we get
$$
  \int_0^t \sum_{i\ne j}\frac{ \big[ |v_i(s)|^{\al/2}-|v_j(s)|^{\al/2}\big]^2}{(i^{2/3}-j^{2/3})^2}\rd s
 \le Cb^{-1} \| \bv(0)\|_\al^\al.
$$
Now we formulate another Sobolev-type inequality  which will be proved  in
Appendix~\ref{sec:appsob}.

\begin{theorem}\label{thm:2ndSob} There is a constant $C>0$ such that
 for any  $M \in I$  and $\bu\in \C^M$ we have
$$
    |u_M|^2 \le M^{-2/3}C^{\sqrt{\log M}}\Bigg[ \sum_{i\ne j =1}^M \frac{ (u_i-u_j)^2}{(i^{2/3}- j^{2/3})^2}
   + \sum_{i=1}^M |u_i|^2\Bigg].
$$
\end{theorem}
The factor $C^{\sqrt{\log M}}$ is probably an artifact of our proof. The factor $M^{-2/3}$ is optimal
as we can   take $u_M=1$ and $u_j=0$ for all other $j\ne M$.

Using  Theorem \ref{thm:2ndSob} with the choice $u_i= |v_i|^{\al/2}$, we have for any $M\le K$,
$$
   \int_0^t |v_M(s)|^\al \rd s \le M^{-2/3} C^{\sqrt{\log M}} b^{-1}\Big[ \| \bv(0)\|_\al^\al
  +  \int_0^t \| \bv(s)\|^\al_\al \rd s \Big] \le  M^{-2/3} C^{\sqrt{\log M}}b^{-1}(t+1) \| \bv(0)\|_\al^\al,
$$
where in the last step we used that the $L^\al$ norm  does not increase
in time  by \eqref{nash}.
This completes the proof of   Lemma~\ref{cor:disp2}. \end{proof}

\begin{proof}[Proof of Theorem \ref{thm:cor}]
On the set $\cG$, the coefficients of $\cA(s)=\cB(s)+\cW(s)$
satisfy the bounds  \eqref{B} and \eqref{W1} with the constant $b=cN^{-2\xi}$.
Using a H\"older inequality
\be\label{hoeq}
\int_0^{\wt T} |v_M(s)|\rd s  \le \wt T^{1-\frac{1}{\al}} \Big( \int_0^{\wt T}
 |v_M(s)|^\al \rd s\Big)^{1/\al},
\ee
and then  \eqref{linfb},
 with the choice of $t=\wt T$ from \eqref{deftt},  we can complete the bound \eqref{10}:
\be\label{100}
\sum_a \frac{N^{C\xi}K^{1/3}}{K+1-a} \int_0^{\wt T}
\E^\om  \big[ \cG  v_{a}^b(\Si)\big] \rd\Si  \le C_\al \sum_a
\frac{N^{C\xi}K^{1/3} K^{(1+\zeta)/6}}{K+1-a}  \; a^{-\frac{2}{3\al}}
\le  C_\al N^{C\xi} K^{-\frac{4-(3+\zeta)\al}{6\al}}.
\ee
Combining \eqref{finrepp}, \eqref{10} and \eqref{100} we get
$$
\left |  \langle O; h_0 \rangle_{\om }  \right |
 \le  C_\al N^{C\xi}  K^{-\frac{4-(3+\zeta)\al}{6\al}}
 +O(N^{-C\xi})
$$
with the special choice of $Q$ from \eqref{O}.
For any $\zeta<1$ there exists an $\alpha>1$
such that the exponent of $K$ is negative. Then, with
 a sufficiently small $\xi$ (depending on $\zeta$, $\alpha$
and $\delta$) we obtain \eqref{h0corr}  and this completes the proof of Theorem~\ref{thm:cor}.
\end{proof}

\begin{proof}[Proof of  Theorem~\ref{thm:cordec}]
We follow
the proof of  Theorem~\ref{thm:cor}, but instead of $Q(\bx)$ and $h_0(\bx)$ we use
the simple observables $q(x_i)$, $f(x_j)$  depending on a single coordinate.
Then the analogue of  \eqref{TTcut} gives a bound
$N^{C\xi}\wt T^{-2}\| q'\|_\infty \| f' \|_\infty$ and
\eqref{finrepp} reads
$$
  \langle q(x_i); f(x_j)\rangle \le N^{C\xi}\| q'\|_\infty \| f' \|_\infty
\int_0^{\wt T} \E^\om \big[ \cG\, v_j^i(S)\big] \,\rd S
 \le  C_\al N^{C\xi}  \| q'\|_\infty \| f' \|_\infty
 \Big [ j^{-2/3\al}\wt T^{1-\frac{1}{\al}}  \wt T^{\frac{1}{\al}}   + \wt T^{-2} \Big ] ,
$$
where in the last step we used \eqref{linfb} with \eqref{hoeq} as above  and an inequality similar
to  \eqref{TTcut} with $\partial_a h_0 = \delta_{0i}$ (notice that  the factor $K^{(1+\zeta)/3}$ is
not needed now.)
Choosing $\alpha$ very close to 1, we can replace $j^{-2/3\alpha}$ with $j^{-2/3}$
at the expense of increasing the constant $C$ in the exponent of  $N^{C\xi}$.
Optimizing these two estimates
 yields the choice $\wt T= j^{ 2/9}$ and thus
$$
   \langle q(x_i); f(x_j)\rangle \le  C_\al N^{C\xi}j^{ -4/9}\| q'\|_\infty \| f' \|_\infty.
$$
Taking into account the rescaling
explained in Section~\ref{sec:rescale}, which results in the additional factor $N^{4/3}$
due to the derivatives,
this proves \eqref{qfdec} in Theorem~\ref{thm:cordec}.
\end{proof}

\appendix

\section{Proof of lemmas \ref{lem:varianceBound}, \ref{lem:edgeVar2}
 and \ref{lem:edgeVar}}\label{App:CalcLem}

\subsection{Proof of Lemma \ref{lem:varianceBound}.}

Let $\varepsilon>0$ be fixed and arbitrarily small as in the statement
of the lemma and  in
the definition of $\Omega^{(N)}_{\rm Int}(3a/4+\e,\tau)$.
In this proof, the notation $A(N,k,j,z)\lesssim B(N,k,j,z)$ means that there is an absolute constant $c>0$ depending only on $V$
such that for any element  $z\in\Omega^{(N)}_{\rm Int}(a,\tau)$ one has $|A(N,k,j,z)|\leq c\ |B(N,k,j,z)|$.
 In the same way, $A\sim B$ means
$c^{-1} |B|\leq |A|\leq c |B|$ for some $c>0$ depending only on $V$.

 For any fixed $E$ with $A\le E\le B$, we
 define the index $j$ such that $\gamma_j=\min\{\gamma_i:\gamma_i\geq E\}$.
For notational simplicity, we assume without loss of generality that $j\leq N/2$ and $A=0$.
Note that
\be\label{Ej}
E\sim (j/N)^{2/3}
\ee
when $E\geq N^{-2/3}$, from the definition  of $j$,
and $|E-j| \le C N^{-2/3}j^{-1/3}$.
Moreover, we will often use the fact that, as
a consequence of $z\in \Omega^{(N)}_{\rm Int}(3a/4+\e,\tau)$,
we have
\be\label{Sig}
  \eta \ge N^{-1+\frac{3a}{4} +\e} E^{-1/2} \sim N^{-\frac{2}{3}+\frac{3a}{4} +\e} j^{-\frac{1}{3}}.
\ee
We define
\be\label{Deltadef}
   \Delta_k:=\frac{1}{z-\lambda_k}-\frac{1}{z-\E(\lambda_k)} = \frac{\lambda_k-\E(\lambda_k)}{(z-\E(\lambda_k))(z-\lambda_k)}.
\ee
Then for any fixed $b$ (we will choose $b$ close to $a$, $a<b<a+\varepsilon/10$)
 we have
$$
\frac{1}{N^2}{\rm Var}\left(\sum_{k=1}^N\frac{1}{z-\lambda_k}\right)\lesssim \Sigma_{\rm Int}+\Sigma_{\rm ExtRight}+\Sigma_{\rm ExtLeft}
$$
where
$$
\Sigma_{\rm Int}=\frac{1}{N^2}\E\left(\sum_{|k-j|\leq N^{b}}\Delta_k\right)^2,\
\Sigma_{\rm ExtRight}=\frac{1}{N^2}\E\left(\sum_{k>j+ N^{b}}\Delta_k\right)^2,\
\Sigma_{\rm ExtLeft}=\frac{1}{N^2}\E\left(\sum_{k<j-N^{b}}\Delta_k\right)^2,
$$
(some of these summations may be empty).
We use the improved concentration result, $|\lambda_k -\E(\lambda_k)| \le N^{-\frac{2}{3} + \frac{a}{2}+\e'}k^{-\frac{1}{3}}$
with very high probability
for any $\e'$ (Proposition \ref{prop:ImprConc})  to bound the numerator in \eqref{Deltadef},
and the rigidity at scale $a$
to bound the denominator,  which allows us to replace $\lambda_k$ and $\alpha_k$ with $\gamma_k$,
and $z=E+\ii\eta$ with $\gamma_j +\ii\eta$ whenever $|j-k|\ge N^b$,
 since in this regime the error
is much smaller than $|\gamma_k-\gamma_j|$.
To see this statement more precisely, first observe that
we can assume that $k< 2N/3$, i.e., $\wh k\sim k$; the large $k$ regime is
trivial since $j\le N/2$.
 We make a distinction between two cases.
\begin{itemize}
\item
We first assume that $k\geq j/2$. Notice that
\be\label{replaa}
  |E-\lambda_k|\ge |\gamma_j-\lambda_k| - CN^{-2/3}
  \ge  c|\gamma_j-\gamma_k| - CN^{-2/3} - CN^{-\frac{2}{3}+a+\e'}(\wh k)^{-1/3}
\ee
from the choice of $j$ and from the rigidity bound for $\lambda_k$ with any  $\e'>0$.
Since $|j-k|\ge N^b$, we have $|\gamma_j-\gamma_k|\ge c N^{-2/3+b} [\max(j, k)]^{-1/3}$.
Using that $b>a$, one can choose $\e'>0$ such that $ |\gamma_j-\gamma_k|$ in \eqref{replaa} dominates
the two error terms, for $ k\ge j/2$.
\item
Suppose now that $ k\le j/2$, then
\eqref{replaa} can be improved by  noticing that
$$
\lambda_k\le \lambda_{j/2}\le \gamma_{j/2} +  CN^{-\frac{2}{3}+a+\e'}j^{-1/3}
$$
(using rigidity for $\lambda_{j/2}$),
thus we can use
\be\label{replaa1}
|E-\lambda_k|\ge
   \gamma_j-\gamma_{j/2} -  CN^{-2/3} -  CN^{-\frac{2}{3}+a+\e'}j^{-1/3}
\ge c|\gamma_j-\gamma_k| -  CN^{-2/3} -  CN^{-\frac{2}{3}+a+\e'}j^{-1/3}
\ee
instead of \eqref{replaa}.
Since $k\le j/2$, we have $|\gamma_j-\gamma_k|\ge  c N^{-2/3+b} j^{-1/3}$, which
is larger than the error term in \eqref{replaa1}.
\end{itemize}

To summarize, we proved the following estimates:
\begin{align*}
&\Sigma_{\rm Int}\lesssim\frac{1}{N^2}\left(\sum_{|k-j|\leq N^{b}}\frac{N^{-\frac{2}{3}+\frac{b}{2}}k^{-\frac{1}{3}}}{\eta^2}\right)^2,\\
&\Sigma_{\rm ExtRight}\lesssim\frac{1}{N^2}\left(\sum_{k\geq j+ N^{b}}\frac{N^{-\frac{2}{3}+\frac{b}{2}}k^{-\frac{1}{3}}}{\eta^2+(\gamma_k-\gamma_j)^2}\right)^2,\\
&\Sigma_{\rm ExtLeft}\lesssim\frac{1}{N^2}\left(\sum_{1\le k\leq j- N^{b}}\frac{N^{-\frac{2}{3}+\frac{b}{2}}k^{-\frac{1}{3}}}{\eta^2+(\gamma_k-\gamma_j)^2}
\right)^2 .
\end{align*}

{\bf Sum over internal points.} We first consider $\Sigma_{\rm Int}$. This is smaller than
$$
\frac{N^{-\frac{4}{3}+b}}{N^2\eta^4}\left(\sum_{\ell=\max(1,j-N^{b})}^{j+N^{b}}\ell^{-\frac{1}{3}}\right)^2
\lesssim
N^{-\frac{10}{3}+b}\eta^{-4}(N^{2b}j^{-2/3}\mathds{1}_{j\geq N^{b}}+N^{\frac{4b}{3}}\mathds{1}_{j\leq N^{b}})
\lesssim
N^{-\frac{10}{3}+3b}\eta^{-4}j^{-2/3}.
$$
This last term is, as expected, smaller than $N^{-1+\frac{3a}{4}}\eta^{-1}\max (E^{\frac{1}{2}},\eta^{\frac{1}{2}})$
which holds for the following reasons.
\begin{itemize}
\item Case $\eta\leq E$.  The desired inequality is $N^{-\frac{7}{3} + 3b - \frac{3a}{4}} \lesssim \sqrt{E} j^{\frac{2}{3}} \eta^3$.

As $E\sim(j/N)^{2/3}$, the desired inequality is
$\eta\gg  N^{-\frac{2}{3}}j^{-\frac{1}{3}}N^{b-\frac{a}{4}}$.
 This holds because $z\in\Omega^{(N)}_{\rm Int}(3a/4+\e,\tau)$, hence
$\eta \gtrsim N^{-1+\frac{3a}{4}+\e} E^{-\frac{1}{2}}\sim N^{-\frac{2}{3}+\frac{3a}{4}+\e} j^{-\frac{1}{3}}$, and $b\leq a+\varepsilon/10$.
\item  Case $\eta\geq E$. The desired inequality is $\eta^{7/2}\gg N^{-\frac{7}{3}+3b-\frac{3a}{4}}j^{-2/3}$.
We distinguish two cases. For large $j$, namely for  $j\gg N^{b-\frac{a}{4}}$,
from $\eta\geq E$ we have
$$
\eta^{\frac{7}{2}}\geq E^{\frac{7}{2}}=\left(\frac{j}{N}\right)^\frac{7}{3}\gg N^{-\frac{7}{3}+3b-\frac{3a}{4}}j^{-2/3}.
$$
On the other hand, from (\ref{Sig})  we have
$$\eta^{\frac{7}{2}}\geq N^{-\frac{7}{3}}j^{-\frac{7}{6}}N^{\frac{7}{2}(\frac{3a}{4}+\varepsilon)}\gg
N^{-\frac{7}{3}+3b-\frac{3a}{4}}j^{-2/3}$$
whenever $j\ll N^{(\frac{27}{4}a-6b)+7\varepsilon}$. As $(\frac{27}{4}a-6b)+7\varepsilon>b-\frac{a}{4}$,
we have either $j\gg N^{b-\frac{a}{4}}$ or $j\ll N^{(\frac{27}{4}a-6b)+7\varepsilon}$, so in any case we have proved the expected result.
\end{itemize}

{\bf Sum over external points on the right.} We now consider $\Sigma_{\rm ExtRight}$. Note that when $\ell\geq 0$,
$\gamma_{j+\ell}-\gamma_j=\ell N^{-\frac{2}{3}}j^{-\frac{1}{3}}\mathds{1}_{\ell\leq j}+\left(\frac{\ell}{N}\right)^{\frac{2}{3}}\mathds{1}_{\ell>j}$,
 consequently
$$
\Sigma_{\rm ExtRight}\leq \Sigma_1+\Sigma_2,\
\Sigma_1=\frac{1}{N^2}\left(\sum_{N^b\leq \ell\leq j}\frac{N^{-\frac{2}{3}+\frac{b}{2}}j^{-\frac{1}{3}}}{\eta^2+\ell^2N^{-\frac{4}{3}}j^{-\frac{2}{3}}}\right)^2,\
\Sigma_2=
\frac{1}{N^2}\left(\sum_{\ell\geq j}\frac{N^{-\frac{2}{3}+\frac{b}{2}}\ell^{-\frac{1}{3}}}{\eta^2+\left(\frac{\ell}{N}\right)^{\frac{4}{3}}}\right)^2.
$$
We first consider $\Sigma_1$. This summation is non-empty if $j\ge N^b \geq N^a$.
\begin{itemize}
\item In the case $E\leq \eta\leq \tau$, we have
$$
\Sigma_1\leq
\frac{1}{N^2}\left(\sum_{N^b\leq \ell \leq j}\frac{N^{-\frac{2}{3}+\frac{b}{2}}j^{-\frac{1}{3}}}{\eta^2}\right)^2
=
\frac{N^{-\frac{10}{3}}}{\eta^4}N^b j^{\frac{4}{3}}\ll \frac{N^{\frac{3a}{4}}}{N\eta}\sqrt{\eta},
$$
where the last step holds because $\eta^{\frac{7}{2}}\geq E^{\frac{7}{2}}\gg \left(\frac{j}{N}\right)^{\frac{7}{3}}N^{b-\frac{3a}{4}}j^{-1}$,
where the last inequality follows from \eqref{Ej} and the fact that $j\geq N^b$.
\item If $\eta\leq E$, we first consider the case $\eta\leq N^{-\frac{2}{3}+a}j^{-\frac{1}{3}}$.
 The following holds  (using  \eqref{Ej})
$$
\Sigma_1\leq \frac{1}{N^2}\left(\sum_{N^b\leq \ell\leq j}\frac{N^{-2/3+\frac{b}{2}}j^{-\frac{1}{3}}}{\ell^2 N^{-\frac{4}{3}}j^{-\frac{2}{3}}}\right)^2
= \left(\frac{j}{N}\right)^{\frac{2}{3}}N^{-b}\ll \frac{N^{\frac{3a}{4}}}{N\eta}\sqrt{E},
$$
because $\eta\leq N^{-\frac{2}{3}+a}j^{-\frac{1}{3}}\ll N^{-\frac{2}{3}+\frac{3a}{4}+b}j^{-\frac{1}{3}}$.
\item In the last possible case $N^{-\frac{2}{3}+a}j^{-\frac{1}{3}}\leq \eta\leq E$, we have
$$
\Sigma_1\leq
\frac{1}{N^2}\left(\sum_{N^b\leq \ell \leq\eta N^{\frac{2}{3}}j^{-\frac{1}{3}}}\frac{N^{-\frac{2}{3}+\frac{b}{2}}j^{-\frac{1}{3}}}{\eta^2}\right)^2
+
\frac{1}{N^2}\left(\sum_{\eta N^{\frac{2}{3}}j^{-\frac{1}{3}}\leq \ell\leq j}\frac{N^{-\frac{2}{3}+\frac{b}{2}}j^{-\frac{1}{3}}}{\ell^2N^{-\frac{4}{3}}j^{-\frac{2}{3}}}\right)^2
=
\frac{N^b}{N^2\eta^2}\ll \frac{N^{\frac{3a}{4}}}{N\eta}\sqrt{E}
$$
because $\eta \geq N^{-\frac{2}{3}+a}j^{-\frac{1}{3}}\gg N^{-\frac{2}{3}+b-\frac{3a}{4}}j^{-\frac{1}{3}}$ and we used $\eqref{Ej}$.
\end{itemize}

We now consider the term $\Sigma_2$.
\begin{itemize}
\item If $\eta\leq E$, we have
$$
\Sigma_2\leq \frac{1}{N^2}\left(\sum_{\ell\geq j}\frac{N^{-\frac{2}{3}+\frac{b}{2}}\ell^{-\frac{1}{3}}}{(\ell/N)^{4/3}}\right)^2=N^{-\frac{2}{3}+b}j^{-\frac{4}{3}}\ll \frac{N^{\frac{3a}{4}}}{N\eta}\left(\frac{j}{N}\right)^{\frac{1}{3}} = \frac{N^{\frac{3a}{4}}}{N\eta}\sqrt{E},
$$
where in the last  inequality  we used $\eta\leq E \sim (j/N)^{2/3} $
 and $j\gg N^{b-\frac{3a}{4}}$, this last relation holds because on
$\{\eta\leq E\}\cap\Omega^{(N)}_{\rm Int}(3a/4+\e,\tau)$ we have $j\geq N^{\frac{3a}{4}+\varepsilon}$.
\item If $\eta\geq E$, we have
$$
\Sigma_2\leq
\frac{1}{N^2}\left(\sum_{j\leq \ell \leq N\eta^{\frac{3}{2}}}\frac{N^{-\frac{2}{3}+\frac{b}{2}}\ell^{-\frac{1}{3}}}{\eta^2}\right)^2
+
\frac{1}{N^2}\left(\sum_{N \eta^{\frac{3}{2}}\leq \ell\leq N}\frac{N^{-\frac{2}{3}+\frac{b}{2}}\ell^{-\frac{1}{3}}}{(\ell/N)^{\frac{4}{3}}}\right)^2=\frac{N^b}{N^2\eta^2}\ll\frac{N^{\frac{3a}{4}}}{N\eta}\sqrt{\eta},
$$
where in the last step we used that on the domain $\{\eta\geq E\}\cap\Omega^{(N)}_{\rm Int}$, we have $\eta \geq N^{-\frac{2}{3}+\frac{a}{2}+\frac{2}{3}\varepsilon}$.
\end{itemize}

{\bf Sum over external points on the left.} We now consider $\Sigma_{\rm ExtLeft}$, which is non-trivial
only for $j\geq N^b \ge  N^a$.
Beginning similarly to the previous paragraph, we can write
$$
\Sigma_{\rm ExtLeft}\leq \widetilde\Sigma_1+\widetilde\Sigma_2,\
\widetilde\Sigma_1=\frac{1}{N^2}\left(\sum_{1\leq k\leq \frac{j}{2}}\frac{N^{-\frac{2}{3}}k^{-\frac{1}{3}}N^{\frac{b}{2}}}{\eta^2+\left(\frac{j}{N}\right)^{\frac{4}{3}}}\right)^2,\
\widetilde\Sigma_2=
\frac{1}{N^2}\left(\sum_{\frac{j}{2}\leq k\leq j-N^b}\frac{N^{-\frac{2}{3}}j^{-\frac{1}{3}}N^{\frac{b}{2}}}{\eta^2+(j-k)^2N^{-\frac{4}{3}}j^{-\frac{2}{3}}}\right)^2.
$$
A calculation yields
$$
\widetilde\Sigma_1=N^{-2+b}E^2\min(\eta^{-4},E^{-4})\ll \frac{N^{\frac{3a}{4}}}{N\eta}\max(E^{\frac{1}{2}},\eta^{\frac{1}{2}}),
$$
 where in the last step we used the following.
\begin{itemize}
\item If $\eta\geq E$, then the  desired inequality is $N^{-1+b -\frac{3a}{4}} E^2 \le \eta^{7/2}$
which  follows from $N^{-1+b -\frac{3a}{4}}  \le E^{3/2} \sim j/N$, which
holds since $j\ge N^a \ge N^{b -\frac{3a}{4}}$.

\item If $\eta\leq E$, the desired relation is $N^{-2+b} E^{-2} \le  \frac{N^{\frac{3a}{4}}}{N\eta} \sqrt{E}$
which again follows from   $N^{-1+b -\frac{3a}{4}}  \le E^{3/2} \sim j/N$ as before,
since  $j\ge N^a \ge N^{b -\frac{3a}{4}}$.
\end{itemize}

We now consider the $\widetilde\Sigma_2$ term.
\begin{itemize}
\item If $\eta \leq E$ and  $N^{-\frac{2}{3}+\frac{3a}{4}+\varepsilon}j^{-\frac{1}{3}}\leq\eta\leq N^{-\frac{2}{3}+a}j^{-\frac{1}{3}}$, we have
$
\widetilde\Sigma_2=N^{-\frac{2}{3}}j^{\frac{2}{3}}N^{-b}
$, so the desired result $\widetilde\Sigma_2\ll \frac{N^{\frac{3a}{4}}}{N\eta} \sqrt{E}\sim
 \frac{N^{\frac{3a}{4}}}{N\eta}\left(\frac{j}{N}\right)^{\frac{1}{3}}$ is equivalent to
$N^{-\frac{2}{3}}j^{\frac{2}{3}}N^{-b}\ll \frac{N^{\frac{3a}{4}}}{N\eta}\left(\frac{j}{N}\right)^{\frac{1}{3}}$,  i.e.,
$\eta\ll N^{b+\frac{3a}{4}+\varepsilon}N^{-2/3}j^{-1/3}$, which obviously holds by the assumption $\eta\leq N^{-\frac{2}{3}+a}j^{-\frac{1}{3}}$.
\item If $\eta\leq E$ and $N^{-\frac{2}{3}+a}j^{-\frac{1}{3}}\leq \eta$, $\widetilde\Sigma_2$ is bounded by
$$
\frac{1}{N^2}\left(\sum_{\frac{j}{2}\leq \ell\leq j-\eta N^{2/3}j^{1/3}}\frac{N^{-\frac{2}{3}}j^{-\frac{1}{3}}N^{\frac{b}{2}}}{(j-\ell)^2N^{-\frac{4}{3}}j^{-\frac{2}{3}}}\right)^2
+
\frac{1}{N^2}\left(\sum_{ j-\eta N^{2/3}j^{1/3}\leq \ell\leq j-N^b}\frac{N^{-\frac{2}{3}}j^{-\frac{1}{3}}N^{\frac{b}{2}}}{\eta^2}\right)^2=\frac{N^b}{N^2\eta^2}\ll\frac{N^{\frac{3a}{4}}}{N\eta}\sqrt{E},
$$
where  in the last step we used \eqref{Ej} and that $\eta\gg N^{-\frac{2}{3}}j^{-\frac{1}{3}}N^{b-\frac{3a}{4}}$ from \eqref{Sig}.
\item
If $E\leq\eta\leq \tau$, we also have $\widetilde\Sigma_2\leq \frac{N^b}{N^2\eta^2}$
which is properly bounded, exactly as we proved it for the proof of $\Sigma_2$ on the domain $\{\eta\geq E\}$.
\end{itemize}

\subsection{Proof of Lemma \ref{lem:edgeVar2}}

Let $d>2a/3$ and $b>3a/4$. On $\Omega_{\rm Ext}^{(N)}(d,\tau)$, we have $\eta\geq c \kappa_E$, so we want to prove that
uniformly in $z\in\Omega_{\rm Ext}^{(N)}(d,\tau)$ we have
$$
\frac{1}{N^2}\var\left(\sum_{i=1}^N\frac{1}{z-\lambda_i}\right)\ll\frac{1}{N\eta}\eta^{1/2}.
$$
We know that
\begin{align*}
&\frac{1}{N^2}\left|\var\left(\sum_{i=1}^N\frac{1}{z-\lambda_i}\right)\right|
\lesssim \Sigma_{\rm Int}+\Sigma_{\rm Ext},\\
&\Sigma_{\rm Int}=\frac{1}{N^{ 2 }}\E^\mu\left|\sum_{i\leq N^b}\left(\frac{1}{z-\lambda_i}-\frac{1}{z-{\E^\mu(\lambda_i)}}\right)\right|^2\lesssim\frac{1}{N^2}\left(
\sum_{i\leq N^b}\frac{N^{-\frac{2}{3}+\frac{a}{2}}i^{-\frac{1}{3}}}{\eta^2}\right)^2,\\
&\Sigma_{\rm Ext}=\frac{1}{N^{ 2}}\E^\mu\left|\sum_{i> N^b}\left(\frac{1}{z-\lambda_i}-\frac{1}{z-
\E^\mu(\lambda_i)}\right)\right|^2
\lesssim
\frac{1}{N^2}\left(
\sum_{i> N^b}\frac{N^{-\frac{2}{3}+\frac{a}{2}}i^{-\frac{1}{3}}}{\eta^2+(|E-A|+\left(\frac{i}{N}\right)^{\frac{2}{3}})^{2}}\right)^2,\\
\end{align*}
where we used concentration at scale $a/2$
that follows from the conditions of Lemma \ref{lem:edgeVar2}
using Proposition~\ref{prop:ImprConc}.
In the last equation, we additionally used  accuracy at scale $3a/4$ for $i\geq N^b$
to obtain that $|z-\lambda_i|\sim|z-\E(\lambda_i)|\sim |z-\gamma_i|$.
The term $\Sigma_{\rm Int}$ is therefore easily bounded by $N^{-\frac{10}{3}+a+\frac{4}{3}b}  \eta^{-4}$, and
$$
\Sigma_{\rm Ext}\leq N^{-\frac{10}{3}+a}
\left(
\sum_{i>1}\frac{i^{-\frac{1}{3}}}{\eta^2+\left(\frac{i}{N}\right)^{\frac{4}{3}}}\right)^2\lesssim\frac{N^{-2+a}}{\eta^2}.
$$
This concludes the proof because, for $\eta\geq N^{-\frac{2}{3}+d}$, $d>2a/3$, we have both
$$
\frac{N^{-\frac{10}{3}+2a}}{\eta^4}\leq \frac{1}{N\eta}\eta^{1/2}
\quad \mbox{and}\quad
\frac{N^{-2+a}}{\eta^2}\leq \frac{1}{N\eta}\eta^{1/2}.
$$.

\subsection{Proof of Lemma \ref{lem:edgeVar}}

Let $b>3a/4$. We begin with the bound on $m_N'$:
$$
\left|\frac{1}{N}m_N'(z)\right|
\leq
\frac{1}{N^2}\sum_{i\leq N^b}\E\left(\frac{1}{|z-\lambda_i|^2}\right)+
\frac{1}{N^2}\sum_{i> N^b}\E\left(\frac{1}{|z-\lambda_i|^2}\right).\\
$$
If $N^{-\frac{2}{3}+a+\e}>|z-A|>N^{-\frac{2}{3}+d}$
and $z\in \Omega^{(N)}(d,s,\tau)$, then the contributions
of both terms are easily bounded by
\be\label{secc}
\frac{1}{N^2}N^{b}(N^{-\frac{2}{3}+s})^{-2}+\frac{1}{N^2}\sum_{i\geq 1}\frac{1}{\left(|z-A|+\left(\frac{i}{N}\right)^{\frac{2}{3}}\right)^2}
\leq
N^{-\frac{2}{3}+b-2s}+N^{-1}|z-A|^{-\frac{1}{2}}.
\ee
If $|z-A|>N^{-\frac{2}{3}+a+\e}$, then by rigidity at scale $a$
we can use the second term  in \eqref{secc} to estimate
all indices $i\ge 1$. By choosing $b =3a/4 +\e$, this
gives the expected result \eqref{eqn:edgeEst1}.

We now bound the variance term, in the same way as in the previous subsection:
\begin{align*}
&\frac{1}{N^2}\left|\var\left(\sum_{i=1}^N\frac{1}{z-\lambda_i}\right)\right|
\lesssim \Sigma_{\rm Int}\mathds{1}_{N^{-\frac{2}{3}+a+\e}>|z-A|>N^{-\frac{2}{3}+d}}+\Sigma_{\rm Ext},\\
&\Sigma_{\rm Int}=\frac{1}{N^{2}}\E\left|\sum_{i\leq N^b}\left(\frac{1}{z-\lambda_i}-\frac{1}{z-\E^\mu(\lambda_i)}\right)\right|^2\lesssim\frac{1}{N^2}\left(
\sum_{i\leq N^b}\frac{N^{-\frac{2}{3}+\frac{a}{2}}i^{-\frac{1}{3}}}{\eta^2}\right)^2,\\
&\Sigma_{\rm Ext}=\frac{1}{N^{2}}\E\left|\sum_{i> N^b}\left(\frac{1}{z-\lambda_i}-\frac{1}{z-\E^\mu(\lambda_i)}\right)\right|^2
\lesssim
\frac{1}{N^2}\left(
\sum_{i> 1}\frac{N^{-\frac{2}{3}+\frac{a}{2}}i^{-\frac{1}{3}}}{(|z-A|+\left(\frac{i}{N}\right)^{\frac{2}{3}})^{2}}\right)^2.\\
\end{align*}
The announced bounds then follow by a computation of the above terms.

\section{Two Sobolev-type inequalities}\label{sec:appsob}

In this section we prove two Sobolev type inequalities. The first one
has a discrete and continuous version, the second one is valid
only in the discrete setup.

\begin{proof}[Proof of  Proposition \ref {prop:sob}]
We start with the proof of \eqref{conti1}. We recall the representation formula
for fractional powers of the Laplacian: for any $0< \al <2$ function $f$ on $\bR$ we have
\be
  \langle f, |p|^\al f\rangle = C(\alpha)\int_\bR\int_\bR\frac{ (f(x)-f(y))^2}{|x-y|^{1+\al}}
   \rd x\rd y
\label{pal}
\ee
with some explicit constant $C(\al)$, where $|p|: = \sqrt{-\Delta}$.

In order to bring the left hand side of \eqref{conti1} into  the form similar
to \eqref{pal}, we estimate,
 for $0<x<y$,
$$
   y^{2/3}-x^{2/3} =(3/2)\int_x^y s^{-1/3}\rd s \le C(y-x)(xy)^{-1/6}
$$
(for $y-x\le x$ we have $x\sim y$ and it follows directly, for $y-x\ge x$,
i.e., $y\ge 2x$, and $y-x\sim y$
we get $\int_x^y s^{-1/3}ds\sim y^{2/3} \le (y-x)(xy)^{-1/6}$).
Thus to prove \eqref{conti1}, it is sufficient to show that
\be\label{suff122}
 \int_0^\infty  \int_0^\infty \frac{(f(x)-f(y))^2}{|x-y|^{2-\eta}}
 (xy)^{q}
 \rd x \rd y
\ge c_\eta\Big(\int_0^\infty |f(x)|^p\rd x\Big)^{2/p}, \qquad p=\frac{3}{1+\eta},
  \quad  q:=\frac{1}{3}-\frac{\eta}{6},
\ee
holds for any function supported on $[0,\infty]$.

Now we symmetrize $f$, i.e., define ${\wt f}$ on $\R$ such that ${\wt f}(x)= f(x)$ for $x>0$
and ${\wt f}(x)=f(-x)$ for $x<0$.
Then
$$
2 \int_0^\infty  \int_0^\infty \frac{(f(x)-f(y))^2}{|x-y|^{2-\eta}} |xy|^q
\rd x \rd y
= \int_\R\int_\R \frac{({\wt f}(x)-{\wt f}(y))^2}{|x-y|^{2-\eta}} |xy|^q  \rd x \rd y
  - 2 \int_0^\infty  \int_0^\infty \frac{(f(x)-f(y))^2}{|x+y|^{2-\eta}} |xy|^q
  \rd x \rd y
$$
$$
  \ge \int_\R\int_\R \frac{({\wt f}(x)-{\wt f}(y))^2}{|x-y|^{2-\eta}}  |xy|^q \rd x \rd y
  - 2 \int_0^\infty  \int_0^\infty \frac{(f(x)-f(y))^2}{|x-y|^{2-\eta}} |xy|^q
  \rd x \rd y,
$$
where we used that  $ |x+y| \ge |x-y|$ for positive numbers. Thus
$$
  \int_0^\infty  \int_0^\infty \frac{(f(x)-f(y))^2}{|x-y|^{2-\eta}} (xy)^q
\rd x \rd y
\ge \frac{1}{4} \int_\R\int_\R \frac{({\wt f}(x)-{\wt f}(y))^2}{|x-y|^{2-\eta}}  |xy|^q\rd x \rd y.
$$
Since
$$
  \int_0^\infty |f(x)|^pdx =\frac{1}{2}\int_\R |{\wt f}(x)|^p\rd x,
$$
the estimate \eqref{suff122} would follow from
\be\label{suff2}
\int_\R\int_\R \frac{({\wt f}(x)-{\wt f}(y))^2}{|x-y|^{2-\eta}} |xy|^q \rd x \rd y
 \ge c_\eta'\Big(\int_\R |{\wt f}(x)|^p\rd x\Big)^{2/p},
 \qquad p := \frac{3}{1+\eta}.
\ee
Setting
$$
\phi(x): = |x|^q,
$$
 \eqref{suff2} is equivalent to
\be\label{suff1}
 \int_{-\infty}^\infty  \int_{-\infty}^\infty
\frac{(f(x)-f(y))^2}{|x-y|^{2-\eta}} \phi(x)\phi(y) \rd x \rd y
\ge c_\eta\Big(\int_\R |f(x)|^p\rd x\Big)^{2/p}
\ee
for any function $f$ on $\bR$ (for simplicitly we dropped the tilde in $f$
and the prime in $c_\eta$).

We have
$$
 \int_\R \int_\R \frac{(f(x)-f(y))^2}{|x-y|^{2-\eta}}\phi(x)\phi(y)  \rd x \rd y
  =\lim_{\e\to0}  \int_\R \int_\R \frac{(f(x)-f(y))^2}{|x-y|^{2-\eta}+\e}\phi(x)\phi(y)
 \rd x \rd y
$$
$$
   = \lim_{\e\to0}\Bigg[
 2  \int_\R \int_\R \frac{f(x)^2}{|x-y|^{2-\eta}+\e}  \phi(x)\phi(y) \rd x \rd y
 - 2 \int_\R \int_\R \frac{\phi(x) f(x) \phi(y) f(y)}{|x-y|^{2-\eta}+\e} \rd x \rd y
\Bigg]
$$
$$
   = \lim_{\e\to0}\Bigg[
 2 \int_\R \int_\R \frac{(\phi(x)f(x))^2}{|x-y|^{2-\eta}+\e}\rd x \rd y
 - 2 \int_\R \int_\R \frac{\phi(x) f(x) \phi(y) f(y)}{|x-y|^{2-\eta}+\e}\rd x \rd y
\Bigg]
$$
$$
  +  \lim_{\e\to0} 2\int_\R \int_\R
  \frac{f(x)^2}{|x-y|^{2-\eta}+\e} ( \phi(y) - \phi(x) )\phi(x)\rd x \rd y
$$
$$
  = \lim_{\e\to0} \int_\R \int_\R
\frac{(\phi(x)f(x)-\phi(y)f(y))^2}{|x-y|^{2-\eta}+\e} \rd x \rd y
+  \lim_{\e\to0} 2
 \int_\R |f(x)|^2 \Big[
 \int_\R \frac{  \phi(y) - \phi(x)}{|x-y|^{2-\eta}+\e} dy\Big]\phi(x) \rd x.
$$
The first term is 
$$
    (\phi f, |p|^{1-\eta} \phi f).
$$
Since $f$ is symmetric, we can assume $x>0$ in computing the
second term:
$$
  \lim_{\e\to0}  \int_\R \frac{  \phi(y) - x^q}{|x-y|^{2-\eta}+\e}\rd y
= \lim_{\e\to0} \int_0^\infty \frac{ y^q - x^q}{|x-y|^{2-\eta}+\e} \rd y
 +\lim_{\e\to0} \int_{-\infty}^0 \frac{ (-y)^q- x^q}{|x-y|^{2-\eta}+\e} \rd y
$$
$$
= \lim_{\e\to0} \int_0^\infty \frac{ y^q - x^q}{|x-y|^{2-\eta}+\e} \rd y
 +\lim_{\e\to0} \int_0^{\infty} \frac{ y^q- x^q}{|x+y|^{2-\eta}+\e} \rd y
$$
$$
   =  x^{q-1+\eta}\Big[
\lim_{\e\to0} \int_0^\infty \frac{ u^q - 1}{|u-1|^{2-\eta}+\e} \rd u
 +\lim_{\e\to0} \int_0^{\infty} \frac{ u^q- 1}{|u+1|^{2-\eta}+\e} \rd u\Big]
  = C_0(\eta)  x^{q-1+\eta}.
$$
 We need that $C_0(\eta)>0$ for small $\eta$. Since $C_0$
is clearly continuous, it is sufficient to show that
 $C_0(0)> 0$. 
This can be seen by the $v=1/u$ substitution for $u\ge1$
$$
    \int_0^\infty \frac{ u^q - 1}{|u\pm 1|^{2}} \rd u
  =\int_0^1 \frac{ u^q - 1}{|u\pm 1|^{2}} \rd u
 + \int_0^1 \frac{ (1/v)^q - 1}{|(1/v)\pm 1|^{2}} \frac{\rd v}{v^2}
  = \int_0^1 \frac{ u^q + u^{-q} - 2}{|u\pm 1|^{2}} \rd u > 0
$$
since $u^q+ u^{-q}\ge 2$. (What we really used about the weight function $\phi$ is that
$\frac{1}{2} (\phi(a) +\phi(1/a))\ge \phi(1)$ for any $a>0$.)
Once $C_0(0)>0$, we can choose a sufficiently small $\eta>0$ so that $C_0(\eta)>0$ as well.
{F}rom now on we fix such a small $\eta$.

In summary, we have
$$
 \int_\R  \int_\R \frac{(f(x)-f(y))^2}{|x-y|^{2-\eta}} \phi(x)\phi(y) \rd x \rd y
  =  \langle\phi f, |p|^{1-\eta} \phi f\rangle
+C_0(\eta) \int_\R \frac{|f(x)|^2}{|x|^{1-2q-\eta}}\rd x
$$
$$
= \langle\phi f, |p|^{1-\eta} \phi f\rangle
+C_0(\eta) \int_\R\frac{|\phi(x)f(x)|^2}{|x|^{1-\eta}}\rd x.
$$
So the positive term can be dropped and
in order to prove \eqref{suff1}, we need to prove
$$
\langle f\phi, |p|^{1-\eta}\phi f\rangle\ge
c_\eta\Big( \int_\R |f|^p\Big)^{2/p}.
$$

Denote $g=|p|^{\frac{1}{2}(1-\eta)}|x|^q f$, (recall $q=\frac{1}{3}-\frac{\eta}{6}$), we need to prove that
$$
  \| g\|_2\ge c_\eta\big\| |x|^{-q}|p|^{-\frac{1}{2}(1-\eta)} g\big\|_p.
$$
Recall the weighted Hardy-Littlewood-Sobolev inequality \cite{SteWei1958} in $n$-dimensions
$$
\Big \| |x|^{-q}  \int  |x-y|^{-a}  g(y) \rd y  \Big \|_p \le C \| g \|_r,  \quad  \frac 1 r + \frac { a+q} n =  1 + \frac 1 p, \quad 0 \le q < n/p, \quad
0< a < n .
$$
In our case, $a= (1 + \eta) /2, r = 2, n = 1$,  and all conditions are satisfied if we take
$ 0 < \eta < 1$.
This completes the proof of the continuous part of
 Proposition \ref{prop:sob}.  Part (ii), the discrete version \eqref{discr1},
follows from \eqref{conti1} by  linear interpolation 
exactly as in the proof of Proposition B.2
in \cite{EYsinglegap}.
\end{proof}

\begin{proof}[Proof of Theorem \ref{thm:2ndSob}]
Take $1\le \ell\le j\le M$ and estimate
\begin{align*}
   |u_j|^2 \le & 2\Bigg|\frac{1}{\ell}\sum_{i=j-\ell}^{j-1} (u_j-u_i)  \Bigg|^2 + 2 \Bigg|
  \frac{1}{\ell}\sum_{i=j-\ell}^{j-1} u_i  \Bigg|^2 \\ \nonumber
  \le & \frac{2}{\ell^2} \Bigg(\sum_{i=j-\ell}^{j-1} \frac{(u_j-u_i)^2}{(j^{2/3}-i^{2/3})^2}  \Bigg)
  \Bigg( \sum_{i=j-\ell}^{j-1}(j^{2/3}-i^{2/3})^2\Bigg) + \frac{2}{\ell} \sum_{i=j-\ell}^{j-1} |u_i|^2 \\ \nonumber
  \le & C \ell j^{-2/3} \sum_{i=j-\ell}^{j-1} \frac{(u_j-u_i)^2}{(j^{2/3}-i^{2/3})^2}
  + \frac{2}{\ell} \sum_{i=j-\ell}^{j-1} |u_i|^2,
\end{align*}
where we performed the summation
$$
   \sum_{i=j-\ell}^{j-1}(j^{2/3}-i^{2/3})^2 \le C\ell^3j^{-2/3}.
$$
We will apply this whenever $j\ge M/2$, so $j^{-2/3}$ will be replaced by $CM^{-2/3}$.
So for any $1\le \ell \le j\le M$, $j\ge M/2$, we have
\be\label{oneterm}
  |u_j|^2 \le C_0 \ell M^{-2/3} \sum_{i=j-\ell}^{j-1} \frac{(u_j-u_i)^2}{(j^{2/3}-i^{2/3})^2}
  + \frac{2}{\ell} \sum_{i=j-\ell}^{j-1} |u_i|^2,
\ee
with some fixed constant $C_0$.

Choose an increasing sequence $\ell_1\le \ell_2 \le\dots \le \ell_{n+1}$
such that $\ell_{j+1}\ge 2\ell_j$ and
 $n$ such that $\ell_{n+1}\le M/2$.
We
use \eqref{oneterm} for $j=M$ and $\ell=\ell_1$:
\be\label{um}
  |u_M|^2 \le  C_0 \ell_1 M^{-2/3} \sum_{i=M-\ell_1}^{M-1} \frac{(u_M-u_i)^2}{(M^{2/3}-i^{2/3})^2}
  + \frac{2}{\ell_1} \sum_{j=M-\ell_1}^{m-1} |u_j|^2
  \le  C_0 \ell_1 M^{-2/3} D +  \frac{2}{\ell_1} \sum_{j=M-\ell_1}^{M-1} |u_j|^2,
\ee
where we denote
$$
  D:=  \sum_{i\ne j =1}^M \frac{ (u_i-u_j)^2}{(i^{2/3}- j^{2/3})^2}.
$$
Now for each $u_j$ in the last sum we can use \eqref{oneterm} again, but now with $\ell=\ell_2$
\begin{align*}
\frac{2}{\ell_1} \sum_{j=M-\ell_1}^{M-1} |u_j|^2
 & \le \frac{2C_0\ell_2}{\ell_1} M^{-2/3} \sum_{j=M-\ell_1}^{M-1} \sum_{i=j-\ell_2}^{j-1}
 \frac{(u_j-u_i)^2}{(j^{2/3}-i^{2/3})^2}
  + \frac{4}{\ell_2\ell_1} \sum_{j=M-\ell_1}^{M-1} \sum_{i=j-\ell_2}^{j-1} |u_i|^2 \\
& \le  C_0 M^{-2/3}D \frac{2\ell_2}{\ell_1}
  + \frac{4}{\ell_2} \sum_{i=M-\ell_1-\ell_2}^{M-1} |u_i|^2.
\end{align*}
Combining with \eqref{um} we get
$$
   |u_M|^2 \le    C_0 M^{-2/3}D \Big(\ell_1+ \frac{2\ell_2}{\ell_1} \Big)
  + \frac{4}{\ell_2} \sum_{j=M-\ell_1-\ell_2}^{M-1} |u_j|^2.
$$
Continuing this procedure, after $n$ steps we get
$$
   |u_M|^2 \le    C_0 M^{-2/3}D \Big(\ell_1+ \frac{2\ell_2}{\ell_1} + \frac{4\ell_3}{\ell_2} +
  \dots +  \frac{2^n\ell_{n+1}}{\ell_n} \Big)
  + \frac{2^{n+1}}{\ell_{n+1}} \sum_{j=M-\ell_1-\dots -\ell_{n+1}}^{M-1} |u_j|^2
$$
and the recursion works since up to
the last step the running index $j$ satisfied
$j\ge M-\ell_1 - \dots -\ell_n \ge M - \ell_{n+1}\ge M/2$ by the choice of $n$.
Optimizing the choices we have
$$
   |u_M|^2 \le    2^n C_0 M^{-2/3}D \ell_{n+1}^{\frac{1}{n}} +  \frac{2^{n+1}}{\ell_{n+1}}\sum_j |u_j|^2.
$$
We can choose $\ell_{n+1} = M^{2/3}$ and $n=\sqrt{\log M}$, then
$$
   |u_M|^2 \le    C M^{-2/3}\Big( M^{\frac{2}{3n}} + 2^n\Big)(D + \|u\|_2)
  \le   M^{-2/3} C^{\sqrt{\log M}} \Big(D + \sum_{i=1}^M |u_i|^2\Big),
$$
which completes the proof.
\end{proof}

\section{Proof of Lemma~\ref{lm:init} }\label{app:V*}

Let $M= N^{C\xi}$ with a  constant $C> C_3$ (from the definition of $\cG$).
We define
\be\label{V*defnn}
    V^*_\by(x) : = V\big(xN^{-2/3}\big)
-\frac{2}{N} \sum_{k>K+M} \log|x-y_k|, \qquad x \in J_\by = (-\infty, y_+],
\ee
i.e., we write
$$
   V_\by(x) =   V^*_\by(x)
-\frac{2}{N} \sum_{k=K+1}^{K+M} \log|x-y_k|,
$$
where we split the external points into two sets. The nearby
external points (with indices  $K+1 \le k \le K+M$) are kept explicitly, while the
 far away points, $y_k$, $k>K+M$
 are kept together with the potential in $V_\by^*$ because there is a cancellation
between them to explore. The proof of the following lemma on the derivative
of $V_\by^*$ is postponed to the end of this section.

\begin{lemma}\label{lm:V*dern}
For any $\by\in \cR_K(\xi)$ and $M=N^{C\xi}$ with a large constant $C$, we have
\be\label{Vprime}
  [V_\by^*]'(x)= N^{-2/3} \int_{0}^{[(K+M)/N]^{2/3}}
 \frac{\varrho(y)\rd y}{xN^{-2/3}-y}
 + O \Big(\frac{N^{-1+\xi}}{|(K+M)^{2/3}-x|}\Big), \qquad  x\in [0, y_+]
\ee
and
\be\label{V*dern}
 \big|  [V_\by^*]'(x) \big|
  \le CN^{-1+2\xi}K^{1/3}, \qquad   x\in [-N^{C\xi}, y_+].
\ee
Here we assume that the density $\varrho$ satisfies \eqref{eqn:matchSing}.
\end{lemma}

 From the definitions \eqref{h0def},
\eqref{V*defnn}  we  claim that
\begin{align}\label{u0est}
   |\partial_j h_0(\bx)| &  \le  \sum_{k=K+2}^{K+M}
 \Big[ \frac{1}{|x_j-y_k|} -  \frac{1}{|x_j-\wt y_k|}\Big]
+ N\big| [V_\by^*]'(x_j)-[\wt V^*_{\wt\by}]'(x_j)\big|  \\
 &\le \frac{CN^{C\xi} K^{1/3}}{K+1-j} + N\big| [V_\by^*]'(x_j)-[\wt V^*_{\wt\by}]'(x_j)\big|. \nonumber
\end{align}
Notice that the summation over  $k$ starts from $k=K+2$, this is because
the boundary terms $|x_j-y_{K+1}|^{-1} = |x_j-\wt y_{K+1}|^{-1}$,  present
 both in $V_\by(x_j)$ and $\wt V_{\wt\by}(x_j)$,
cancel out. Using $|x_j-y_k|\ge |y_{K+2}- y_{K+1}|\ge N^{-\xi} K^{1/3}$
by the definition of $\by\in \cR^\#$,
each term in the summation is bounded by $N^\xi K^{1/3}$. So its
contribution is at most $CMN^\xi K^{1/3}\le CN^{C\xi}K^{1/3}$.
We will use this bound for $j\ge K-N^{C\xi}$. For $j\le K-N^{C\xi}$, we use
$|x_j-y_k|\ge |x_j-y_{K+1}|\ge c|\gamma_j-\gamma_K|\ge cK^{-1/3}|K-j|$ and this
gives the estimate on the first term in \eqref{u0est}.

For the second term in \eqref{u0est} we use  \eqref{Vprime} to have
\be\label{V-V}
   N\big| [V_\by^*]'(x_j)-[\wt V^*_{\wt\by}]'(x_j)\big|
\le N^{1/3} \int_{0}^{[(K+M)/N]^{2/3}}
 \frac{\big[\varrho(y)-\wt\varrho(y)\big]\rd y}{x_jN^{-2/3}-y}
 + O \Big(\frac{N^{\xi}}{|(K+M)^{2/3}-x_j|}\Big).
\ee
Notice that $x_j\sim j^{2/3}$ with a precision smaller than $N^{C_3\xi}j^{-1/3}$ since
$$
  | x_j- j^{2/3}| \le |x_j-\gamma_j| +|\gamma_j - j^{2/3}|
  \le N^{C_3\xi} j^{-1/3} + j^{4/3}N^{-2/3} \le N^{C_3\xi}j^{-1/3}, \qquad j\le K,
$$
by the  definition of $\cG$, by \eqref{orient1}
and  \eqref{NKn}.
Thus we have
$$
 |(K+M)^{2/3}-x_j|\ge |(K+M)^{2/3}- j^{2/3}| - N^{C_3\xi}j^{-1/3} \ge
 K^{-1/3}|K+M-j| -  N^{C_3\xi}j^{-1/3} \ge c K^{-1/3}|K+M-j|
$$
using that $M = N^{C\xi}\gg N^{C_3\xi}$.
 Thus the error term in \eqref{V-V}
is bounded by the r.h.s. of \eqref{nabh}.

Finally, in the main term of \eqref{V-V} we use the asymptotics \eqref{rhoasympt2}.
The density $\varrho(y)-\wt\varrho(y)$ is a $C^1$-function of size of order $y^{3/2}\le C(K/N) $
on the integration domain. Thus a simple analysis, similar to the proof of \eqref{cll}
in the Appendix shows that
$$
 N^{1/3} \int_{0}^{[(K+M)/N]^{2/3}}
 \frac{\big[\varrho(y)-\wt\varrho(y)\big]\rd y}{x_jN^{-2/3}-y}
   \le CN^{1/3}(K/N)(\log N)
$$
which is smaller than the r.h.s. of \eqref{nabh} by \eqref{NKn}. This proves \eqref{nabh}.
The proof of \eqref{h0l1} trivially follows from \eqref{nabh}. This completes the
proof of Lemma~\ref{lm:init}.

\begin{proof}[Proof of Lemma~\ref{lm:V*dern}]
For any  fixed $x\in [0, y_+]$ we have
\begin{align}\label{useeq}
  \frac{1}{2}\big( V_\by^*(x)\big)'& =   \frac{1}{2} N^{-2/3}V'(xN^{-2/3})
 - \frac{1}{N} \sum_{k>K+M}\frac{1}{x-y_k}  \\
& = N^{-2/3}\int \frac{\varrho(y)\rd y}{xN^{-2/3}-y}-
\frac{1}{N}   \sum_{k>K+M}\frac{1}{x-y_k}, \qquad
 x\in \big[0, y_+\big],\non
\end{align}
where we have used the equation \eqref{equil}.
Thanks to rigidity, $\by\in \cR_{K}(\xi)$, we can replace $y_k$'s with $\gamma_k$'s at
an error
\begin{align}  \nonumber
  |\cE_1|:= &\Bigg| \frac{1}{N}
  \sum_{k>K+M}\Big[ \frac{1}{x-y_k} - \frac{1}{x-\gamma_k}\Big]
\Bigg|   \le  \frac{1}{N}   \sum_{k>K+M} \frac{|y_k-\gamma_k|}{(x-\gamma_k)^2}\\
&  \le \frac{CN^\xi}{N}
   \sum_{k>K+M}\frac{1}{\wh k^{1/3} (x-\wh k^{2/3})^2}
 \le \frac{CN^{-1+\xi}}{|(K+M)^{2/3}-x|}
\label{cE1n}
\end{align}
where we also used that for any $x\le y_+ \le \gamma_{K+1} + CN^\xi (K+1)^{-1/3}$
and $k\ge K+M$  we have
$\gamma_k-x\ge c(\gamma_k-\gamma_{K+1})$ since $\gamma_k-\gamma_{K+1}\ge cMK^{-1/3}$ is
 larger than the
rigidity error $CN^\xi K^{-1/3}$.
For the purpose of the estimates, we can thus replace
$\gamma_{k}$ with $\wh k^{2/3}$.
 The last step in \eqref{cE1n} is a simple estimate.

After replacement, we have to control
$$
\frac{1}{N}  \sum_{k>K+M}\frac{1}{x-\gamma_k}
 = N^{-2/3} \frac{1}{N}  \sum_{k>K+M}\frac{1}{ xN^{-2/3}-\Gamma_k}
 = N^{-2/3}\int_{Q}
 \frac{\varrho(y)\rd y}{xN^{-2/3}-y} +\cE_2
$$
with  $Q := [N^{-2/3} \gamma_{K+M+1}, B]$.
The error $\cE_2$
can be written as

\begin{align*}
  \cE_2= & N^{-2/3}  \sum_{k>K+M}\int_{\gamma_k /N^{2/3} }^{ \gamma_{k+1}/N^{2/3}}
\Bigg[\frac{1}{xN^{-2/3}-y} -  \frac{1}{xN^{-2/3}- \gamma_kN^{-2/3}}
 \Bigg]\varrho(y)\rd y
\end{align*}
using
$$
 \int_{ \gamma_k /N^{2/3}}^{\gamma_{k+1} /N^{2/3}} \varrho=1/N.
$$
 Thus for
 $x\in [0, y_+]$ the error is bounded
by
$$
 |\cE_2|\le  \frac{CN^\xi}{N^{5/3}} \sum_{k>K+M}
 \Big(\frac{1}{\wh k^{1/3}N^{2/3}}\Big)
 \frac{1}{[ xN^{-2/3}-(\wh k/N)^{2/3}]^2}
  \le
  \frac{CN^{-1+\xi}}{|(K+M)^{2/3}-x|}
$$
using $|y-N^{-2/3} \gamma_k|\le CN^{\xi}\wh k^{-1/3}N^{-2/3}$ for $y\in
[ \gamma_k /N^{2/3},  \gamma_{k+1}/ N^{2/3}]$
and that $\gamma_k/N^{2/3} \sim (k/N)^{2/3}$. The calculation in the last line
is the same as in \eqref{cE1n}.
Thus
\be\label{singint}
 \frac{1}{2}\big( V_\by^*(x)\big)'
= N^{-2/3}\int_{0}^{\gamma_{K+M+1}/N^{2/3}}
 \frac{\varrho(y)\rd y}{xN^{-2/3}-y} + O \Big(\frac{N^{-1+\xi}}{|(K+M)^{2/3}-x|}
 \Big), \qquad x\in [0, y_+].
\ee
Moreover, we have
$$
   N^{-2/3}\Bigg|\int_{[(K+M)/N]^{2/3}}^{\gamma_{K+M+1}/N^{2/3}}
 \frac{\varrho(y)\rd y}{xN^{-2/3}-y}\Bigg|
 \le  \frac{C}{|(K+M)^{2/3}-x|} \int_{[(K+M)/N]^{2/3}}^{\gamma_{K+M+1}/N^{2/3}} \varrho(y)\rd y
 =  O \Big(\frac{N^{-1+\xi}}{|(K+M)^{2/3}-x|}\Big).
$$
Here we used that $|x-N^{2/3}y|$ is comparable with $|x-(K+M)^{2/3}|$ and
$\varrho(y)\le C\sqrt{y}\le C(K/N)^{1/3}$
on the integration domain
and that $| \gamma_{K+M+1}/N^{2/3} - [(K+M)/N]^{2/3}|\le C[K/N]^{4/3}$, see \eqref{orient1}.
Finally we used \eqref{NKn}.
Thus from \eqref{singint} we obtained \eqref{Vprime}.

To obtain the  bound in \eqref{V*dern}, we first notice in the error term in \eqref{Vprime} we have
$|(K+M)^{2/3}-x|\ge |(K+M)^{2/3}-y_+| \ge |(K+M)^{2/3}-(K+1)^{2/3}| - CN^{\xi}K^{-1/3} \ge cK^{-1/3}M$
since $M\ge CN^\xi$. So this can be bounded by the r.h.s. of  \eqref{V*dern}.

The singular
 integral in \eqref{Vprime},  up to logarithmic factors,
 is  bounded by the size of $\varrho$ on this interval, which is
at most $ C(K/N)^{1/3}$, so this term is also bounded by the r.h.s. of  \eqref{V*dern}.
   More precisely,  for any $0\le b < u<a$ and  for density $\varrho$ satisfying  \eqref{eqn:matchSing}, we
claim  that
\be\label{cll}
  \Big |  \int_b^a \frac{\varrho(y)\rd y}{u-y}  \Big | \le C |u|^{1/2}\max\{ \log |u-b|, \log |a-u|\}.
\ee
Since in our case, by the choice of $M$,
 $u:=xN^{-2/3}$ and $a:=[(K+M)/N]^{2/3}$
 are separated by at least $N^{-1}$,
we indeed get (with $b=0$)
$$
\Bigg|N^{-2/3} \int_0^{[(K+M)/N]^{2/3}}
 \frac{\varrho(y)\rd y}{xN^{-2/3}-y}\Bigg| \le CN^{-2/3}(K/N)^{1/3}(\log N).
$$

Finally, we need to consider the case $x<0$. The only difference from
the  proof for $x>0$  is that the equilibrium relation \eqref{equil} holds with
an error term:
$$
   \frac{1}{2} V'(x) = \int \frac{\varrho(y)\rd y}{x-y} + O(N^{-1/3+C\xi} ),
 \qquad x\in [-N^{-\frac{2}{3}+C\xi}, 0],
$$
that can be easily seen by comparing it with the $x=0$ case and using that $V$ is smooth and
$$
   \Big|  \int \frac{\varrho(y)\rd y}{x-y} -  \int \frac{\varrho(y)\rd y}{-y}\Big|
  \le |x|\int  \frac{\varrho(y)\rd y}{(|x|+y)y} \le C|x|^{1/2}\le CN^{-1/3+C\xi}
$$
by $\varrho(y)\le C\sqrt{y}$. This error term in \eqref{useeq} yields an error of
size $N^{-1+C\xi}$ in the final result, which is smaller than the r.h.s. of \eqref{V*dern}.
We thus  proved Lemma \ref{lm:V*dern}. \end{proof}

\section{Level repulsion for the local measure: Proof of Theorem~\ref{lr2}}\label{app:levelRep}

The proof in this section uses ideas similar to those in \cite{BouErdYau2011, EYsinglegap}.
Before we start the actual proof of Theorem~\ref{lr2}, we need some Lemmas.
 We first introduce  an auxiliary measure which
is a  slightly modified  version of the local equilibrium measures:
$$
 \sigma_0 :=
 Z^*   (y_{K+1}-x_{K})^{-\beta} \sigma_{\by},
$$
where  $Z^*$  is chosen for normalization.
In other words, we drop the term $ (y_{K+1}-x_{K})^\beta$ from
the measure $ \sigma_\by$ in $\sigma_0$.
We  first prove    estimates
weaker than  \eqref{k5n}-\eqref{k5nlow}  for $\sigma_{\by}$ and  $\sigma_{0}$.

\begin{lemma}\label{43n}
Let  $\by \in \cR=\cR_{K}(\xi) $.
We have for any $s>0$
\begin{align}
\label{k3n}
\P^{ \sigma_{\by}}  ( y_{K+1}-x_{K} \le sK^{-1/3}) & \le   C K^2  s, \\
\label{k34}
\P^{ \sigma_{\by}}  ( y_{K+1}-x_{K} \le sK^{-1/3}) & \le   C N^{C\xi}  s + e^{-N^c}.
\end{align}
The very same estimates hold if $\sigma_\by$ is replaced with $\sigma_0$.
\end{lemma}

\begin{proof} We set $y_+:= y_{K+1}$ and
$y_-:= y_+- a$ with $a:= N^\xi K^{-1/3}$. By $\by\in \cR_K$ we know that
$$
  y_{K+1}\ge \gamma_{K+1}- N^\xi (K+1)^{-1/3} \ge c K^{2/3},
$$
thus
$$
   0 < y_- < y_+, \qquad y_+, y_-\sim K^{2/3}.
$$
We decompose the configurational space according to the number
of the particles in $[y_-, y_+]$, which we denote by $n$.
For any  $0\le \varphi\le c$ (with a small constant smaller than 1/2) we consider
\begin{align}
Z_\varphi :=  & \sum_{n=0}^K \int\!\!\dots\!\!\int_{-\infty}^{y_-}\Big( \prod_{j=1}^{K-n} \rd x_j\Big)
  \int\!\!\ldots \!\!\int_{y_-}^{y_+-a\varphi} \Big( \prod_{j=K-n+1}^{K} \rd x_j\Big)
\Bigg[ \prod_{i,j\in I\atop i < j} (x_j-x_i)^\beta\Bigg]
e^{- N\frac{\beta}{2} \sum_{j\in I} V_\by (x_j) -2\beta \sum_{j\in I} \Theta(N^{-\xi}x_j) } \nonumber \\
  = &  \sum_{n=0}^K (1-\varphi)^{n+\beta n(n-1)/2}
\int\!\!\ldots\!\!\int_{-\infty}^{y_-}\Big( \prod_{j=1}^{K-n} \rd w_j\Big)
  \int\!\!\ldots \!\!\int_{y_-}^{y_+} \Big( \prod_{j=K-n+1}^{K} \rd w_j\Big) \nonumber\\
& \times \Bigg[ \prod_{i < j\le K-n} (w_j-w_i)^\beta \Bigg]\Bigg[ \prod_{K-n < i < j\le K} (w_j-w_i)^\beta \Bigg]
\Bigg[ \prod_{i\le K-n} \prod_{j=K-n+1}^K (y_-+(1-\varphi)(w_j-y_-)-w_i)^\beta \Bigg]
\nonumber \\
 &
\times  e^{- N \frac{\beta}{2}\big[ \sum_{j\le K-n}  V_\by(w_j)+
\sum_{j> K-n} V_\by ( y_-+ (1-\varphi) (w_j-y_-) ) \big] -
2\beta \sum_{j\le K-n} \Theta\big( N^{-\xi}w_j \big) },
\non
\end{align}
where in the $n$ particle sector we changed variables to
$$
w_j:=x_j \quad \mbox{for}\quad j\le K-n,\qquad
w_j:=y_-+(1-\varphi)^{-1}(x_{j}-y_-) \quad \mbox{for} \quad K-n+1\le j\le K.
$$
We also exploited the fact that for $x_j\ge y_-\ge 0$ we have $\Theta(N^{-\xi}x_j)=0$.

Now we compare $Z_\varphi$ with $Z_{\varphi=0}$.  We fix $n$ and we work in each sector separately.
The mixed interaction terms  can be estimated by
\be\label{mixed}
 \big[ y_-+(1-\varphi)(w_j-y_-)-w_i\big]^\beta  \ge \big[ (1-\varphi)(w_j-w_i)\big]^\beta
\ee
for any $w_i\le y_-\le w_j$.
To estimate the effect of the scaling in the potential term $V_\by$, we fix
 a  parameter $M$ with $CN^\xi\le M \le K$.
For $j>K-n$, i.e $w_j\in [y_-,y_+]$, we write
\begin{align}\label{ch21n}
e^{- N \frac{\beta}{2}  V_\by (y_-+ (1-\varphi) (w_j-y_-) ) }
   & = e^{- N \frac{\beta}{2}  V_\by^*  (y_-+ (1-\varphi) (w_j-y_-) ) } \\ \nonumber
&\;\; \times \prod_{K+1\le k \le K+M} ( y_k- y_-- (1-\varphi) (w_j-y_-) )^\beta.
\end{align}
with the definition
\be\label{V*defn}
    V^*_\by(x) : = V\big(xN^{-2/3}\big)
-\frac{2}{N} \sum_{k>K+M} \log|x-y_k|, \qquad x \in [y_-, y_+].
\ee
Notice that the index $k$ is always between 1 and $N$,
so any limits of summations automatically include this condition as well.

For the potential $V^*_\by$ we have
\be\label{V*n}
\Big |  V^*_\by (y_-+ (1-\varphi) (w_j-y_-))  -  V^*_\by( w_j ) \Big |
\le  \max_{x\in[y_-, y_+]} \big|\big( V_\by^*(x)\big)'\big|\; a \varphi
\le CN^{-1+C\xi}\varphi,
\ee
where we have used $|w_j-y_-|\le a = N^\xi K^{-1/3} $  and $j>K-n$.
The derivative of  $V_\by^*$ will be  estimated in \eqref{V*dern}.
In summary, from \eqref{V*n} we have the lower bound
$$
 e^{- N \frac{\beta}{2}   V_\by^* (y_-+ (1-\varphi) (w_j-y_-)) }
 e^{ N \frac{\beta}{2}  V_\by^* ( w_j) }  \ge e^{-C\varphi N^{C\xi} }, \qquad j>K-n.
$$
For the other factors in \eqref{ch21n}, we use
$ y_k- y_-- (1-\varphi) (w_j-y_-) \ge  (1-\varphi) (y_k-w_j) $ if $k\ge K+1$,
thus
$$
\prod_{K+1\le k \le K+M} ( y_k- y_-- (1-\varphi) (w_j-y_-) )^\beta
\ge (1-\varphi)^{M\beta}
 \prod_{K+1\le k \le K+M} (y_k-w_j)^\beta.
$$
Choose $M=N^{C\xi}$.
 After multiplying these
estimates for all $j\in I$
we thus have the bound
\begin{align}
   Z_\varphi &\ge \sum_{n=0}^K (1-\varphi)^{n+\beta n(n-1)/2 +\beta n(K-n) + \beta nN^{C\xi}} e^{-C\varphi n N^{C\xi}}
\int\!\!\ldots\!\!\int_{-\infty}^{y_-}\Big( \prod_{j=1}^{K-n} \rd w_j\Big)
  \int\!\!\ldots \!\!\int_{y_-}^{y_+} \Big( \prod_{j=K-n+1}^{K} \rd w_j\Big) \nonumber\\
& \times \Bigg[ \prod_{i < j\le K} (w_j-w_i)^\beta \Bigg]
\; e^{- N \frac{\beta}{2} \sum_{j\le K}  V_\by(w_j)-
2\beta \sum_{j\le K-n} \Theta\big( N^{-\xi}w_j \big) }.
\non
\end{align}
Since $n\le K$, we can
estimate
\be\label{1minphi}
 (1-\varphi)^{n+\beta n(n-1)/2 +\beta n(K-n) + \beta nN^{C\xi}} e^{-C\varphi n N^{C\xi}}\ge (1-\varphi)^{CK^2},
\ee
and after bringing this factor out of the summation, the remaining sum is just
$Z_{\varphi=0}$. We  thus have
$$
\frac{Z_\varphi}{Z_0}  \ge  (1-\varphi)^{CK^2}.
$$
Now we choose $\varphi := sK^{-1/3} a^{-1} =sN^{-\xi}$.
 Therefore the $\sigma_{\by}$-probability of $ y_{K+1}-x_{K} \ge  s K^{-1/3}  =a\varphi$
can be estimated by
$$
\P^{\sigma_{\by}}  (  y_{K+1}-x_{K} \ge  sK^{-1/3}  ) =\frac{Z_\varphi}{Z_0}
\ge 1-   Cs  K^2.
$$
  This proves \eqref{k3n}.

For the proof of \eqref{k34}, we first insert the characteristic function of the set
$$
\cG_0:= \big\{ \bx\; : \; |x_j-\alpha_j|\le CN^{\xi}j^{-1/3}, \; j\in I\big\}.
$$
into the integral defining $Z_\varphi$ and denote the new quantity by $Z_\varphi^\cG$.
Clearly $Z_\varphi\ge Z_\varphi^\cG$ and
by the rigidity bound \eqref{32the} we know that
$$
  \frac{Z_0-Z_0^\cG}{Z_0}= \P^{\sigma_\by}(\cG_0^c)\le Ce^{-N^c},
$$
thus (with $\varphi =sN^{-\xi}$ as above)
\be
\P^{\sigma_{\by}}  (  y_{K+1}-x_{K} \ge  sK^{-1/3}  ) =\frac{Z_\varphi}{Z_0}
\ge \frac{Z_\varphi^\cG}{Z_0^\cG}\big(1- Ce^{-N^c}\big) \ge  \frac{Z_\varphi^\cG}{Z_0^\cG} -Ce^{-N^c}
\label{PPs}
\ee
since $Z_\varphi^\cG\le Z_0^\cG$.

To estimate $Z_\varphi^\cG/Z_0^\cG$, we follow the previous proof
with two modifications. First we notice that the summation over $n$ in the
definition of $Z_\varphi$ is restricted to $n\le N^{2\xi}$ on the set $\cG_0$,
since no more than $N^{C\xi}$ particles can fall into the interval $[y_-, y_+]$
if they are approximately regularly spaced.

The other change concerns the estimate of the mixed terms \eqref{mixed} which will be
improved to
\begin{align}\non
  \prod_{i\le K-n} \prod_{j=K-n+1}^K \big[ y_-+(1-\varphi)(w_j-y_-)-w_i\big]^\beta  & \ge
 \prod_{i\le K-n} \prod_{j=K-n+1}^K \Big[(w_j-w_i)\big(1- \frac{\varphi (w_j-y_-)}{w_j-w_i}
 \big)\Big]^\beta \\
&\ge   \prod_{j=K-n+1}^K \Bigg[ \Big( 1- \sum_{i\le K-n} \frac{\varphi (w_j-y_-)}{w_j-w_i}\Big)_+
  \prod_{i\le K-n}  (w_j-w_i) \Bigg]^\beta  \nonumber
\end{align}
for any $w_i\le y_-\le w_j$. Here we used that $\prod_i (1-a_i)\ge \big( 1-\sum_i a_i\big)_+$ for any
 numbers $0\le a_i\le 1$. On the set $\cG_0$ we have, by definitions of $x_j$ and $w_j$, that
$$
 \sum_{i\le K-n} \frac{\varphi (w_j-y_-)}{w_j-w_i} \le  2\varphi
\sum_{i\le K-n} \frac{x_j-y_-}{x_j-x_i} .
$$
Recall  that $x_i\le y_-\le x_j$ and thus $ \frac{x_j-y_-}{x_j-x_i} \le 1$.
For indices $i\le K-CN^\xi$
with a sufficiently large $C$
we can replace $x_i$ with $\alpha_i\sim i^{2/3}$ with replacement error $CN^\xi i^{-1/3}$ which is
smaller than $ y_- - \al_i \le x_j-x_i  $. Together with $ x_j-x_i  \ge c (K^{2/3}- i^{2/3})$,   we have
\be\label{neb}
 \sum_{i\le K-n} \frac{\varphi (w_j-y_-)}{w_j-w_i}
\le
C\varphi N^\xi +
C\varphi N^\xi K^{-1/3} \sum_{i\le K-CN^\xi} \frac{1}{K^{2/3}- i^{2/3}} \le
 C\varphi N^\xi \log N.
\ee
The bound \eqref{neb} will be used $n$-times, for all $j >K-n$.

Collecting these new estimates, instead of \eqref{1minphi} we have
the following prefactor depending on $\varphi$:
$$
 (1-\varphi)^{n+\beta n(n-1)/2 + \beta nN^{C\xi}} (1-C\varphi N^\xi \log N)_+^n e^{-C\varphi n N^{C\xi}}
\ge 1 -CN^{C\xi}\varphi
$$
using  $n\le N^{2\xi}$. This gives
$$
    \frac{Z_\varphi^\cG}{Z_0^\cG} \ge 1 -CN^{C\xi}\varphi
$$
which, together with  \eqref{PPs} and with the choice  $\varphi := sK^{-1/3} a^{-1} =sN^{-\xi}$
gives \eqref{k34}.

The proof of \eqref{k3n}--\eqref{k34}
 for $\sigma_{0}$
is very similar, just
the $k=K+1$ factor is missing from \eqref{ch21n}
in case of  $j=K$.
 This modification does  not alter the basic estimates.
This concludes the proof of Lemma~\ref{43n}.
\end{proof}

\begin{proof}[Proof of Theorem~\ref{lr2}]
Recalling the definition of $ \sigma_0$ and setting $X:= y_{K+1}- x_{K}$
for brevity, we have
\be\label{k4}
\P^{ \sigma_\by} [ X \le s K^{-1/3}  ] =
 \frac { \E^{ \sigma_0}  [  \mathds{1} ( X \le s K^{-1/3} )  X^\beta ] }
{ \E^{ \sigma_0 }  [  X^\beta ]}.
\ee
{F}rom \eqref{k3n} with $\sigma_{0}$
 we have
$$
\E^{ \sigma_0}  [  \mathds{1}
 ( X \le s  K^{-1/3} )  X^\beta ] \le C  (sK^{-1/3})^\beta  K^2s,
$$
 and with the choice $s=cK^{-2}$ in \eqref{k3n} (with $\sigma_0$)
we also have
$$
\P^{ \sigma_0}\left  (  X \ge    \frac{c }{ K^{1/3}K^2}  \right )  \ge 1/2
$$
with some positive constant $c$.
This implies that
$$
\E^{ \sigma_0 }  [  X^\beta ] \ge \frac 1 2  \left (
\frac{c }{ K^{1/3}K^2}
\right )^\beta.
$$
We have thus proved from \eqref{k4} that
$$
\P^{ \sigma_\by} [  X \le s K^{-1/3}   ] \le
C  (sK^{-1/3})^\beta  K^2s  (K^{1/3}K^2)^\beta
= C \left (  K^2 s   \right ) ^{\beta + 1},
$$
 i.e., we obtained \eqref{k5n}.
 The bound \eqref{k5nlow}
  follows
similarly from \eqref{k34}  but with the choice $s=N^{-2C\xi}$
(where $C$ is the constant in the exponent in \eqref{k34}),
  and this completes the proof of Theorem~\ref{lr2}.
\end{proof}

\section{Heuristics for the correlation decay in GUE}\label{app:cordec}

In this section we give a quick heuristic argument to justify the estimate (\ref{opdecay}),
 for a covariance w.r.t. the GUE measure. More precisely,
for $V(x)=x^2/2,\beta=2$, we have
$$
\langle N^{2/3}i^{1/3}(\lambda_i-\gamma_i); N^{2/3}j^{1/3}(\lambda_j-\gamma_j)\rangle_\mu\sim \left(\frac{i}{j}\right)^{\frac{1}{3}},
$$
for all $N^\delta\leq i\ll j\leq N/2$, where $\delta>0$ is a small constant ($i\geq N^\delta$ is just a technical hypothesis allowing an easier use of Hermite polynomials asymptotics  hereafter,  the result should hold true
 without this condition).
In the following, all but the first step can be made easily rigorous by following the method in \cite{Gus2005}: formula (\ref{eqn:EN}) was easier in the context of diverging covariances (this divergence holds when $|i-j|\ll j$, i.e., $\theta_i<\delta$ with the notations of Corollary \ref{cor:Gaussian}), for polynomially vanishing ones it would require a new rigorous argument.
In this section $A\sim B$ means that $c B\leq A\leq c^{-1}B$ for some constant $c>0$ independent of $N$.
\begin{enumerate}
\item Let $\mathcal{N}(x)=\left|\{\ell:\lambda_\ell\leq x\}\right|$.
 Then asymptotic  covariances  of $\lambda_i$'s  are related to those of $\mathcal{N}(\gamma_i)'s$ ($\gamma_i$ is the typical location, defined in (\ref{gammadef})) by the following formula:
\begin{equation}\label{eqn:EN}
\langle N^{\frac{2}{3}}i^{\frac{1}{3}}\lambda_i, N^{\frac{2}{3}}j^{\frac{1}{3}}\lambda_j\rangle_\mu\sim \langle \mathcal{N}(\gamma_i),\mathcal{N}(\gamma_j)\rangle_\mu.
\end{equation}
This relies on the idea that eigenvalues with close enough indexes move together, so for any $x$ and $y$ the events $\{\lambda_i-\gamma_i\leq x(\gamma_{i+1}-\gamma_i),
\lambda_j-\gamma_j\leq y(\gamma_{j+1}-\gamma_j)\}$ and $\{\mathcal{N}(\gamma_i)-i\geq x,\mathcal{N}(\gamma_j)-j\geq y\}$ have very close
probability.
To be made rigorous, this step would require that for $\e>0$ small enough and any $i_1\in\llbracket i-N^\e,i+N^\e\rrbracket,\,
j_1\in\llbracket j-N^\e,j+N^\e\rrbracket$ we have
$$
\max(\langle \lambda_i-\lambda_{i_1},\lambda_j\rangle_\mu,
\langle \lambda_i,\lambda_j-\lambda_{j_1}\rangle_\mu)
\ll
\langle \lambda_i,\lambda_j\rangle_\mu.
$$
This is expected to be true  since $\langle \lambda_a ; \lambda_b \rangle \sim (a/b)^{1/3}$
for $a\ll b$ should imply that $\langle \lambda_a -\lambda_{a'} ; \lambda_b \rangle \sim N^\e \pt_a (a/b)^{1/3}$
if $b-a\le N^\e$, therefore 
 $$
\max(N^\e\partial_i(i/j)^{1/3},N^\e\partial_j(i/j)^{1/3})\sim N^\e\max(i^{-2/3}j^{-1/3},i^{1/3}j^{-4/3})
\ll (i/j)^{1/3}.
$$
\item For $\beta=2$, the spectral measure
is a determinantal point process (with kernel $K_N$ normalized such that $N^{-1}K_N(x,x)\to(2\pi)^{-1}\sqrt{(4-x^2)_+}$), and an elementary calculation gives
\begin{equation}\label{ENex}
\langle \mathcal{N}(\gamma_i),\mathcal{N}(\gamma_j))\rangle_\mu=\iint_{(-\infty,\gamma_i]\times[\gamma_j,\infty)}|K_N(x,y)|^2\rd x\rd y.
\end{equation}
\item Via the Christoffel-Darboux formula, the correlation kernel can be expressed in terms of two successive Hermite polynomials. The Plancherel-Rotach asymptotics then allow us to prove that in the above integral, the main contribution comes from the domain $I\times J=[-2+N^{-2/3+\e},\gamma_i]\times[\gamma_j,0]$ where
$\e>0$ is small enough. In this domain one can prove (see (5.4) in \cite{Gus2005}) that
$K_N(x,y)$ is  asymptotically  equivalent to
$$
\frac{\Ai(-(nF(x))^{2/3})\Ai'(-(nF(y))^{2/3})-\Ai'(-(nF(x))^{2/3})\Ai(-(nF(y))^{2/3})}{x-y},
$$
where Ai is the Airy function and $F(x)\sim (x+2)^{3/2}$ as $x$ decreases to $-2$.
Thanks to the estimates
\begin{align*}
&\Ai(-r)\sim r^{-1/4}\left(\cos(2/3\,r^{3/2}-\frac{\pi}{4})+\OO(r^{-3/2})\right),\\
&\Ai'(-r)\sim r^{1/4}\left(\sin(2/3\,r^{3/2}-\frac{\pi}{4})+\OO(r^{-3/2})\right),
\end{align*}
as $r\to\infty$, we can approximate $K_N(x,y)$ by
\begin{equation}\label{eqn:oscilla}
\left(\frac{2+y}{2+x}\right)^{1/4}\frac{\cos(\frac{2}{3}n^{3/2}(2+x)-\frac{\pi}{4})\sin(\frac{2}{3}n^{3/2}(2+y)-\frac{\pi}{4})}{x-y},
\end{equation}
by noting that in $I\times J$ we have $2+y\gg 2+x$.
\item The end of these heuristics consists in the following calculation, where we use that the square of the oscillating term in (\ref{eqn:oscilla})
averages to $1/2$ (note that the frequencies go to $\infty$), and we note $U\times V=[N^{\e},i^{2/3}]\times[j^{2/3},N^{2/3}]$:
$$
\iint_{I\times J}|K_N(x,y)|^2\rd x\rd y\sim \iint_{U\times V}\left(\frac{v}{u}\right)^{1/2}\frac{1}{(u-v)^2}\rd u\rd v
\sim
\int_U u^{-1/2}\rd u
\int_V v^{-3/2} \rd v\sim\left(\frac{i}{j}\right)^{1/3}.
$$
\end{enumerate}
One concludes using the above equation, (\ref{eqn:EN}) and (\ref{ENex}).

\end{document}